\documentclass[10pt,twoside]{article}
\usepackage{graphics}
\usepackage{amsmath, amsfonts, amsthm, amssymb, amscd,a4wide}
\usepackage{longtable}
\usepackage{booktabs}
\usepackage{hyperref}
\usepackage{url}
\usepackage{mathrsfs} 
\usepackage{epsfig, subfig}
\usepackage{caption}
\usepackage{slashed}
\usepackage{xcolor}
\usepackage{hyperref}
\hypersetup{linktoc = all}                % hyperref settings
\hypersetup{hidelinks}
\hypersetup{bookmarksnumbered}

\newtheorem{definition}{Definition}[section]

\newtheorem{lemma}[definition]{Lemma}

\newtheorem{proposition}[definition]{Proposition}

\numberwithin{equation}{section}
\numberwithin{table}{section}
\numberwithin{figure}{section}
\newcommand \Yb {\overline Y}
\newcommand \Ybar {\overline Y}
\newcommand \trans {\text{tran}}
\newcommand \loc {\text{loc}}
\newcommand \hchi {\widehat \chi}
\newcommand \erf {\text{erf}}

\newcommand \hvarphi {\widehat \varphi}
\newcommand \hlambda {\widehat \lambda}

\newcommand \norm {\text{norm}}
\newcommand \Knorm {K^\norm}
\newcommand \anorm {a^\norm}

\newcommand \bse {\begin{subequations}}
\newcommand \ese {\end{subequations}}

\usepackage[most]{tcolorbox}

\usepackage[most]{tcolorbox} 
\newtcolorbox{fini}{%
     enhanced, breakable, size=minimal, colframe=white, parbox=false, after={\par\vspace{2\baselineskip}}, 
     before upper={\indent}, colback=white, 
     overlay = {\draw[line width=2pt] (frame.north east) -|
                       ([xshift=3mm]frame.east)|-(frame.south east);},
     overlay first={\draw[line width=2pt] (frame.north east) -|
                           ([xshift=3mm]frame.south east);},
     overlay middle={\draw[line width=2pt] ([xshift=3mm]frame.north east) -- 
                              ([xshift=3mm]frame.south east);},
     overlay last={\draw[line width=2pt] ([xshift=3mm]frame.north east)|-
                          (frame.south east);},
}
\newtcolorbox{adone}{%
     enhanced, breakable, size=minimal, colframe=white, parbox=false, after={\par\vspace{2\baselineskip}}, 
     before upper={\indent}, colback=white, 
     overlay = {\draw[densely dotted, line width=2pt] (frame.north east) -|
                       ([xshift=3mm]frame.east)|-(frame.south east);},
     overlay first={\draw[densely dotted, line width=2pt] (frame.north east) -|
                           ([xshift=3mm]frame.south east);},
     overlay middle={\draw[densely dotted, line width=2pt] ([xshift=3mm]frame.north east) -- 
                              ([xshift=3mm]frame.south east);},
     overlay last={\draw[densely dotted, line width=2pt] ([xshift=3mm]frame.north east)|-
                          (frame.south east);},
}
\newtcolorbox{plfok}{%
     enhanced, breakable, size=minimal, colframe=white, parbox=false, after={\par\vspace{2\baselineskip}}, 
     before upper={\indent}, colback=white, 
     overlay = {\draw[dashed, line width=2pt] (frame.north east) -|
                       ([xshift=3mm]frame.east)|-(frame.south east);},
     overlay first={\draw[dashed, line width=2pt] (frame.north east) -|
                           ([xshift=3mm]frame.south east);},
     overlay middle={\draw[dashed, line width=2pt] ([xshift=3mm]frame.north east) -- 
                              ([xshift=3mm]frame.south east);},
     overlay last={\draw[dashed, line width=2pt] ([xshift=3mm]frame.north east)|-
                          (frame.south east);},
}
\newcommand \Yt {\widetilde Y}
 
\newcommand \per {\text{per}}
\newcommand \seed {\text{seed}}

\newcommand \Kseed {K^\seed}
\newcommand \Kper {K^\per}
\newcommand \Lbf {\mathbf L} 

\newcommand \EEE E
\newcommand \Omegab {\overline \Omega}

\newcommand \Cbf {\mathbf C}

\newcommand \chih {\widehat \chi}
\newcommand \varphih {\widehat \varphi}

\newcommand \bei {\begin{itemize}}
\newcommand \eei {\end{itemize}}

\newcommand \Span {\text{Span}}

\newcommand \be   {\begin{equation}}
\newcommand \bel {\be\label}
\newcommand \ee   {\end{equation}}

\newcommand \supp {{\text{supp}}}
\newcommand \Id   {{\text{Id}}}

\newcommand \ZZ    {\mathbb{Z}}
\newcommand \RR    {\mathbb{R}}
\newcommand \NN    {\mathbb{N}}

\newcommand \Dcal    {\mathcal{D}}

\newcommand \Ccal    {\mathcal{C}}

\newcommand \Hcal    {\mathcal H}
\newcommand \Hcalt   {\widetilde{\mathcal H}}

\newcommand \Kcal    {\mathcal{K}}

\newcommand \RN    {{\RR^N}}

\newcommand \del   {\partial}

\newcommand \la         \langle
\newcommand \ra     \rangle

\newcommand \RD {{{\mathbb R}^D}}

\begin{document} 

\title{Mesh-free error integration in arbitrary dimensions: 
\\
a numerical study of discrepancy functions} 

\author{Philippe G. LeFloch\footnote{Laboratoire Jacques-Louis Lions, Centre National de la Recherche Scientifique, Sorbonne Universit\'e, 4 Place Jussieu, 75252 Paris, France. Email: {\tt contact@philippelefloch.org}.}
\, 
 and Jean-Marc Mercier\footnote{MPG-Partners, 136 Boulevard Haussmann, 75008 Paris, France. Email: {\tt jean-marc.mercier@mpg-partners.com.}
\newline
Preprint version: October 2019. Revised: November 2019.
}
}

\date{}

\maketitle

\begin{abstract}   
We are interested in mesh-free formulas based on the Monte-Carlo methodology for the approximation of multi-dimensional integrals, and we investigate their accuracy when the functions belong to a reproducing-kernel space. A kernel typically captures regularity and qualitative properties of functions ``beyond'' the standard Sobolev regularity class. We are interested in the issue whether quantitative error bounds can be a priori guaranteed in  applications (e.g.~mathematical finance but also scientific computing and machine learning).
Our main contribution is a numerical study of the error discrepancy function based on a  comparison between several numerical strategies, when one varies the choice of the kernel, the number of approximation points, and the dimension of the problem.
We consider two strategies in order to localize to a bounded set the standard kernels defined in the whole Euclidian space (exponential, multiquadric, Gaussian, truncated), namely, on one hand the class of {\sl periodic kernels} defined via a discrete Fourier transform on a lattice and, on the other hand, a class of {\sl transport-based kernels}.
%-----------------------------------
First of all, relying on a Poisson formula on a lattice, together with heuristic arguments, we discuss the derivation of theoretical bounds for the discrepancy function of periodic kernels.
Second, for each kernel of interest, we perform the numerical experiments that are required in order to generate the optimal distributions of points and the discrepancy error functions. Our numerical results allow us to validate our theoretical observations and 
provide us with quantitative estimates for the error made with a kernel-based strategy as opposed to a purely random strategy. 
\end{abstract}

%\setcounter{tocdepth}{1}
%\tableofcontents 
%
% 
%===========================================================================

\section{Introduction}
\label{section:introd}

\vskip.05cm \par{\bf An error approximation formula.}
We are motivated here by applications to partial differential equations arising continuum physics,  including the development of mesh-free methods in fluid dynamics and material sciences \cite{BFBL,GuLiu,Koester, LiLiu,Liu-2003,Nakano,OhDavis,ZhouLi}.
Specifically, we are interested in approximating multi-dimensional integrals via Monte-Carlo-type formula and deriving error estimates, in which the dependency with respect to the dimension of the problem and other important parameters is specified in a  
quantitative manner. By revisiting this problem of multivariate integration, our purpose is to clarify the derivation and validity of such estimates whose importance has been highlighted in recent years in artificial intelligence, for mesh-free computations of partial differential equations, and mathematical finance. The existing literature emphasizes the role of Sobolev-type spaces, while we would like here to stress the importance of kernel-based Hilbert spaces. In many applications, one is interested in preserving certain a priori structure that are available a priori and the choice of a kernel is dictated by properties (symmetry, scaling, regularity, decay, etc.) that should be incorporated in the approximation algorithms. Therefore, it is desirable to have a flexible framework that encompasses a wide class of kernels, as we consider in the present paper. 

More specifically, within a given Hilbert (or Banach) space we seek to optimize the choice of the interpolation points in an integral approximation formula and establish a sharp error estimate within the chosen class of regularity and decay. 
Two parameters are of primary interest, namely, the dimension $D \geq 1$ of the problem and the number of interpolation points $N \geq 1$, and it is essential to have quantitative estimates with a specified dependency in $N,D$ that can be determined from the kernels of interest. 
In the present paper, we contribute to this general objective and provide a systematic study and comparison of several classes of kernels, which we refer to as periodic kernels  and transported kernels. Our periodic framework for periodic kernels is motivated by work by Cohn and Elkies \cite{CohnElkies} who studied the problem of sphere packing. Our result depends upon a ``kernel density'' function which arises as a key factor in a quantitative bound. We build here on many earlier works on the subject, including contributions in approximation theory \cite{BBO,FGE, FGE2,MN,Opfer,Salehi,Wendland,Wendland-book}. 

%--------------------------------------------------------------------------------------------------

\vskip.05cm \par{\bf The discrepancy function associated with a kernel.} 
To any kernel $K: \Omega \times \Omega \to \RR$ defined on a bounded and open subset $\Omega \subset \RD$
and satisfying a positivity condition (see Section~\ref{section-21}), we associate a Banach space $\Hcal_K^{s,p}(\Omega)$ of real-valued functions defined on $\Omega$ with regularity exponent $s>0$ and integrability exponent $p \in [1, +\infty)$. 
(More generally, the Lebesgue measure could be replaced by a probability measure.) 
Then, an ``abstract'' error integration estimate reads, for any $N \geq 1$ and any function $\varphi \in \Hcal_K^{s,p}(\Omega)$,
\bel{eq:20D} 
\inf_{x^1, \ldots, x^N \in \Omega}
\Big| {1 \over |\Omega|} \int_\Omega  \varphi(x) \, dx - {1 \over N}  \sum_{1 \leq n \leq N}  \varphi(x^n) \Big|
\leq 
E_K^{s,p}(N,D) 
\, \|\varphi\|_{\Hcal_K^{s,p}(\Omega)}, 
\ee 
in which the {\it discrepancy function} $E_K^{s,p}(N,D)$ is independent of $\varphi$. Hence, \eqref{eq:20D} provides us ---in the class of functions under consideration--- a factorization of the error in two contributions:   $\|\varphi\|_{\Hcal_K^{s,p}(\Omega)}$ measures the regularity of the function while  
the discrepancy function is related to the best distribution of $N$ points in $\Omega$. The challenge is to control  $E_K^{s,p}(N,D)$, which can be expressed in several forms: 
\bei 

\item[1.] In the physical space $\Omega$, the function $E_K^{s,p}(N,D)$ can be formulated with a pseudo-distance associated with the kernel. 

\item[2.] In suitable spectral variables determined from an operator naturally associated with the kernel, the function $E_K^{s,p}(N,D)$ takes a rather explicit form involving the eigenfunctions and eigenvectors of this operator. 

\eei 

\noindent However, both formulations are difficult to work with directly ---except in dimension $D=1$. So, we introduce below a third standpoint which is more efficient in order to control and its dependency with respect to $D$ and $N$, that is, we introduce the class of ``lattice-based'' kernels (in a tensorial form), as we call them. In this context, we can express the function $E_K^{s,p}(N,D)$ via:  

\bei 

\item[3.] a Poisson formula in dual discrete Fourier variables associated with a lattice (see next section).  

\eei 
\noindent Interestingly, for this latter class of kernels, quantitative estimates can be established that involve the the notion of a ``lattice density'' function, as we explain it in this paper, and shed some light on the problem of the curse of dimensionality.  A priori and quantitative error bounds are obtained at any order of accuracy at the expense of possibly increasing the regularity of the functions under consideration. Importantly for the applications, the error function is controlled quantitatively in a given functional framework.

\vskip.05cm \par{\bf Evaluation of the discrepancy function.}  
We focus attention on a selected list of kernels which we construct by a nonlinear transformation of four translation-invariant kernels,  that is, of the form $K(x,y) = \varphi(x-y)$. We choose kernels that are commonly used in the applications, namely the exponential\footnote{which is sometimes also refered to as the Mat\'ern kernel}, multiquadric, Gaussian, and truncated kernels. Their Fourier transform $\hvarphi$ defined on $\RD$ is known explicitly and is listed in Table~\ref{FDK}. 

These kernels are defined in $\RD$ and we proceed by ``localizing'' them to a bounded domain $\Omega$,
taken to be the unit cube $[0,1]^D$ for simplicity in the presentation. We propose two methods for such a localization of a kernel $K$ defined on $\RD\times \RD$: 
\bei 
\item The {\it periodic version} $\Kper$ of $K$ defined from a discrete Fourier transform.

\item The {\it transported version} $K^\trans$ of $K$ defined via a nonlinear transport map. 

\eei
This provides us with eight kernels (listed in Table \ref{FDK2}) and our main purpose is to investigate the discrepancy function associated with each of them. 
 
In principle, we could use numerically any one of the three expressions of the error function which we derive below and attempt to minimize it over the set of $Y$. In most cases, this requires a computation which, in general, cannot be done explicitly and a numerical integration of this function would be very costly, especially in large dimensions. We discuss this below. In particular, due to the (non-convex) form of the kernel,  minimizing the error function in the physical space is computationally challenging. 
By introducing a periodic version based on the discrete Fourier transform, we arrive at an expression that is computationally tractable. 
We are able to make comparisons between these kernels and investigate the rate of convergence while comparing with the case when the points are randomly chosen. Our numerical results confirm and support our theoretical discussion. 
 
%---------------------------------------------------------  

\vskip.05cm \par{\bf Applications and perspectives.}
The material in this paper should be useful for analyzing mesh-free methods for computing solutions to partial differential equations and deriving quantitative bounds for algorithms used in pattern recognition and artificial intelligence. The estimate discussed here provides us with a key building block in order to  establish an error analysis of the transported mesh-free method presented in the companion paper \cite{PLF-JMM-4}. The method therein can be regarded as a generalization of the Lagrangian mesh-free method use in computational fluid dynamics, but also allows to include Navier-Stokes-type diffusive terms. 

Most of the literature on error integration estimates is focused on functions with Sobolev regularity while we are interested here in functions with regularity adapted to specific applications. For instance, the standard choice of radially-symmetric kernels leads to functional spaces that are variants of Sobolev spaces and, in particular, are invariant by translations. Allowing more general kernels allows one to describe local (direction-dependent) properties of functions. 
For instance, a kernel we discuss below is adapted to measure the regularity of functions of the form $\varphi =\sum_{0<n_1< \ldots< n_k \le D } \varphi_{n_1,\ldots,n_k}(x_{n_1},\ldots,x_{n_k})$, relevant in mathematical finance.  
The strategy in \cite{PLF-JMM-2,PLF-JMM-4} is now applied in industrial applications \cite{JMM-SM} and its accuracy can be explained in the light of the present study. 
This is relevant when considering the valuations of complex financial products 
(including the so-called American exercising) written on a large number of underlyings, and 
aiming at computing rapidly complex risk measures; see \cite{PLF-JMM-2}.

%----------------------------------- 

\vskip.05cm \par{\bf Outline of this paper.}
In Section \ref{section--2}, we present some basic material on reproducing kernel spaces.  In Section \ref{section--3}, we discuss our methodology for constructing the two classes of kernels of main interest. In Section~\ref{sec-select}, the kernels studied in the present paper are presented and some their properties  discussed. In Section \ref{section--4}, we derive several expressions of the discrepancy function, depending whether physical, spectral, or Fourier variables are used and, next, in Section~\ref{sec-opti} we derive estimates on the discrepancy error. In Section~\ref{sec-numerics}, we present and discuss our numerical results for each of the kernels of interest.

% while some technical estimates are postponed to the appendix.

%----------------------------------------------------------------------------------------------------------------------------------------
\begin{center}
\begin{table}
%[h]
\centering
\begin{tabular}{|l||c||c||c||c|}
  \hline
    & exponential (E) & multiquadric (M) & Gaussian (G) & truncated (T)
\\
  \hline
&&&&
\\
   $\chi$ & $\exp{(-|x|_1)}$ 
%   $e^{-|x|_1}$ 
& $(1+|x|^2)^{-(D+1)/2}$ & $\exp(-|x|^2/2)$ & $\sup(1-|x|,0)^D$ \\
  \hline 
&&&&
\\
   $\hchi$ & $(1+|\xi|^2)^{-(D+1)/2}$ & $\exp(-|\xi|)$ & $\exp(-|\xi|^2/2)$ &  (see \cite{Wendland-book}) \\
  \hline
\end{tabular}
\caption{Four kernels and their Fourier transforms on $\RD$}
\label{FDK}
\end{table}
\end{center}  
%----------------------------------------------------------------------------------------------------------------------------------------

\

%-------------------------------------------------------------------------------------------------------------------------------------------------- 
\begin{table}
%[h]
\hskip-.5cm
\begin{tabular}{|l||c||c||c||c|}
  \hline
    & exponential (E) & multiquadric (M) & Gaussian (G) & truncated (T)
\\
  \hline
%   $\varphi(x)$ 
&&&&
\\
\text{Tra}
& $\exp(-|\erf^{-1}(x)|_1)$ & $\big( 1+|\erf^{-1}(x)|^2 \big)^{-(D+1)/2}$ & $\exp(-|\erf^{-1}(x)|^2/2)$ & $\sup\big( 1-|\erf^{-1}(x)|,0 \big)^D$ 
\\
  \hline
%   $\varphi(x)$ 
&&&&
\\
\text{Per}
& $\sum \exp(- |x+\alpha|_1)$ & $\sum (1+|x+\alpha|^2)^{-(D+1)/2}$ & $\sum \exp(-|x+\alpha|^2/2)$ & $\sum \sup(1-|x+\alpha|,0)^D$ 
\\
  \hline
\end{tabular}

\caption{Transported (Tra) and periodic (Per) kernels on $[-1,1]^D$ (sum over $\alpha \in \ZZ^D$)}
\label{FDK2}
\end{table}
%---------------

%===========================================================================

\section{Functional framework based on a reproducing kernel}
\label{section--2} 

\subsection{Discrete setup}
\label{section-21}

\vskip.05cm \par{\bf The class of admissible kernels.}
Since we are primarily interested in kernels defined on a bounded set, in the present section we restrict attention  to this class ---although we will allow ourselves to manipulate kernels defined
on the whole Euclidian space and treated as ``seed data'' in order to generate the kernels of actual interest.
A reproducing kernel provides a convenient way to generate a broad class of Hilbert spaces (or, more generally, Banach space); cf.~\cite{FGE,Wendland}.  A bounded and continuous function $K: \Omega \times \Omega \to \RR$ on a bounded open set  $\Omega \subset \RD$ 
is called an {\it admissible kernel} if it satisfies
 (1) the {\it symmetry property:} $K(x,y) = K(y,x)$ for all $x, y \in \Omega$, and 
(2) the {\it positivity property:} for any collection of $N$ distinct points $Y = (y^1, \ldots, y^N)$ in $\Omega$, the symmetric matrix $K(Y,Y) = \big( K(y^m,y^n) \big)_{1 \leq n,m \leq N}$ is positive definite in the sense that $a^T K(Y,Y) a > 0$ for all $a \in \RN \setminus \{0\}$. 
It is said to be {\it uniformly positive} if there exists a uniform constant $c>0$ such that for any collection of distinct points $Y$ one has $a^T K(Y,Y) a \geq c \, |a|^2$ for all $a \in \RN$.

Clearly, any admissible kernel also satisfies  
\be
\aligned
& K(x,x) \geq 0, 
\qquad 
&&  K(x,y)^2 \leq K(x,x) \, K(y,y),  
\qquad x,y \in \Omega. 
\endaligned
\ee
This implies that $2 \, K(x,y) \leq K(x,x) + K(y,y)$ and, therefore, the non-negative function 
\bel{eq:defD}
D(x,y):= K(x,x) + K(y,y) - 2 K(x,y) \geq 0, 
\qquad x,y \in \Omega,
\ee
can be interpreted as a ``pseudo-distance'' in view of the properties $D(x,x) = 0$ and $D(x,y) = D(y,x)$. (The triangle inequality need not hold.) Many examples of admissible kernels will be presented in the next two sections.

%-------------------------------------

\vskip.05cm \par{\bf Finite dimensional framework.}
Given any finite collection of points $Y = (y^1, \ldots, y^N)$ chosen in $\Omega$, we introduce the (finite dimensional) vector space $\Hcal_K^Y(\Omega)$ consisting of all linear combinations of the {\it basis functions} $x \mapsto K(x, y^n)$. In other words, we set  
\bel{equa:HKY}
\Hcal_K^Y(\Omega):= \Big\{\sum_{1 \leq m \leq N} a_m K(\cdot, y^m) \,  / \,  a = (a_1, \ldots, a_N) \in \RR^N  \Big\}. 
\ee 
Since $K$ is continuous, $\Hcal_K^Y(\Omega)\subset \Ccal(\Omega)$ embeds into the space $\Ccal(\Omega)$ of all continuous functions on $\Omega$.  To any two functions 
$\varphi = \sum_{1 \leq m \leq N} a_m K(\cdot, y^m)$ and $\psi = \sum_{1 \leq n \leq N} b_n K(\cdot, y^n)$, we associate the bilinear expression 
\label{equa-bilin-norm}
\bel{Npsi}
\la \varphi, \psi \ra_{\Hcal_K^Y(\Omega)} 
:= a^T K(Y,Y) b 
= 
 \sum_{1 \leq m \leq N} \sum_{1 \leq n \leq N} a_m b_n K(y^m,y^n)
\ee
(with $a = (a_m)$, etc.), which endows the space $\Hcal_K^Y(\Omega)$ with a Hilbertian structure with norm  
$\| \varphi \|_{\Hcal_K^Y(\Omega)}^2 :=   a^T K(Y,Y) a$. 
Now, the so-called {\it reproducing kernel property} (immediate from \eqref{Npsi})
\bel{eq:repkp} 
\la K(\cdot, y^m), K(\cdot, y^n) \ra_{\Hcal_K^Y(\Omega)} 
= K(y^m, y^n),
\ee
allows one to relate the coefficients of the decomposition of a function $\varphi = \sum_{1 \leq m \leq N} a_m K(\cdot, y^m)$ to its scalar product with the basis functions, namely
\be
\aligned
\la \varphi, K(\cdot, y^n)  \ra_{\Hcal_K^Y(\Omega)}
& = 
\sum_{1 \leq m \leq N} a_m \la K(\cdot, y^m), K(\cdot, y^n) \ra_{\Hcal_K^Y(\Omega)}
  =
\sum_{1 \leq m \leq N} a_m K(y^m, y^n) = a^T K(Y, y^n). 
\endaligned
\ee

%-----------------------------------

\vskip.05cm \par{\bf Discrete spectral decomposition.} 
Since $K(Y,Y)$ is a symmetric and positive definite matrix, it admits real and positive eigenvalues, denoted by $\lambda_Y^n >0$, together with a basis of right-eigenvectors $\zeta_Y^n \in \RD$ (with $n=1,2, \ldots, N$), satisfying 
\be
K(Y,Y) \zeta_Y^n = \lambda_Y^n \, \zeta_Y^n, \qquad n=1,2, \ldots, N.
\ee
This decomposition is useful in order to define, for any $s \geq 0$ and $p \geq 1$, the finite dimensional Banach space $\Hcal_K^{Y,s,p}(\Omega)$ of all functions of the form $\sum_{1 \leq m \leq N} a_m K(\cdot, y^m)$ with finite norm 
\bel{HKSs}
\| \varphi \|^p_{\Hcal_K^{Y,s,p}(\Omega)}  
:= \sum_{1 \leq n \leq N} (\lambda_Y^n)^{-sp} \la \varphi, \zeta_Y^n \ra_{\Hcal_K^Y(\Omega)}^p.
\ee 

%------------------------- 

\vskip.05cm \par{\bf Projection operator.} 
Consider a function $f \in \Ccal(\Omega)$ and introduce the vector $f(Y) = \big( f(y^1), \ldots, f(y^N) \big)$ consisting of the values of this function at the given points. We define its projection $P(f)$ into the discrete space $\Hcal_K^Y(\Omega)$ by setting  
\begin{subequations}
\be
P_Y(f):= a^T K(\cdot, Y), 
\qquad a:=  K(Y,Y)^{-1} f(Y). 
\ee
Clearly, this defines a projection since $P \circ P(f) = P(f)$ and, in fact, $P(\varphi) = \varphi$ for any function $\varphi = a^T K(\cdot, Y)$ belonging to the space $\Hcal_K^Y(\Omega)$. 
Moreover, the norm of this projection reads 
\be
\| P_Y(f) \|_{\Hcal_K^Y(\Omega)}^2 =  f(Y)^T K(Y,Y)^{-1} f(Y). 
\ee
\end{subequations}
%

%------------------------------- 

\vskip.05cm \par{\bf The partition of unity.} 
A basis is naturally associated with the discrete space $\Hcal_K^Y(\Omega)$, that is, 
$N$ functions $\theta_Y^n: \Omega \to \RR$ taking the values $0$ or $1$ at the points of the set $Y$. Precisely, writing $\theta_Y := (\theta^1_Y, \ldots, \theta^N_Y)$, we define  
\bse
\bel{PU}
\theta_Y:= K(Y,Y)^{-1} K(Y, \cdot). 
\ee
It follows that 
$\big( \theta_Y(y^m) \big)_{1 \leq n,m \leq N} = K(Y,Y)^{-1} K(Y,Y) = \Id$ (the identity matrix), and using the Kronecker symbol, we have $\theta_Y^n (y^m) = \delta_{nm}$, while the scalar product of any two basis functions is 
\bel{410}
\big\la \theta_Y^m,  \theta_Y^n \big \ra_{\Hcal_K^Y(\Omega)} 
=  K^{-1}(y^m, y^n).  
\ee 
This partition of unity is useful in expressing the projection of a function $f \in \Ccal(\Omega)$, namely 
$P_Y(f) = \sum_{1 \leq n \leq N} f(y^n) \theta_Y^n$.  
\ese

%---------------------------------------------------------------------------------------------------------------------------------

\subsection{Continuous setup}

\vskip.05cm \par{\bf Functional spaces.}
Given an admissible kernel $K:\Omega \times \Omega \to \RR$, we now introduce the infinite dimensional space $\Hcalt_K(\Omega)$ consisting of all finite linear combinations of the functions $K(x, \cdot)$ parametrized by $x \in \Omega$, that is, 
$
\Hcalt_K(\Omega)
:=  \Span \big\{K(\cdot, x) \, / \, x \in \Omega \big\}, 
$
which we endow with the scalar product and norm defined in the finite dimensional setup; see \eqref{equa-bilin-norm} (where now $Y$ and $N$ are no longer fixed). 
By construction, the  reproducing kernel property \eqref{eq:repkp} also holds in $\Hcalt_K(\Omega)$, i.e. 
\bse
\label{538}
\be
\la K(\cdot, x), K(\cdot, y) \ra_{\Hcalt_K(\Omega)} = K(x,y), \qquad x, y \in \Omega. 
\ee
From the Cauchy-Schwarz inequality, it follows that for any $\varphi \in \Hcalt_K(\Omega)$ and $x \in \Omega$ 
\bel{539} 
\aligned
|\varphi(x)| 
= | \la K(\cdot, x), \varphi \ra_{\Hcalt_K(\Omega)}|  
& \leq \| K(\cdot, x) \|_{\Hcalt_K(\Omega)} \, \|\varphi\|_{\Hcalt_K(\Omega)}
 = \sqrt{K(x,x)} \, \|\varphi\|_{\Hcalt_K(\Omega)}.
\endaligned
\ee   
\ese
Since the kernel is continuous and bounded, the ``point evaluation'' $\varphi \mapsto \varphi(x)$ is thus a linear and bounded functional on $\Hcalt_K(\Omega)$ (for any $x \in \Omega$). 

We have defined a pre-Hilbert space, that is, a vector space endowed with a scalar product and, in order to obtain a complete metric space, the completion of the pre-Hilbert space $\Hcalt_K(\Omega)$ is considered by taking all linear combinations based on countably many points $Y= (y^1, y^2, \ldots)$ in $\Omega$. The corresponding space is denoted by $\Hcal_K(\Omega)$ and is 
the {\it reproducing Hilbert space} generated from the kernel $K$. 

Clearly, both properties \eqref{538} remain true in $\Hcal_K(\Omega)$. Observe also that 
\be
\aligned
|\varphi(x) - \varphi(y)| 
& =  | \la K(\cdot, x) - K(\cdot, y), \varphi \ra_{\Hcal_K(\Omega)}|
\\
& \lesssim \| K(\cdot, x) - K(\cdot, y) \|_{\Hcal_K(\Omega)} \|\varphi\|_{\Hcal_K(\Omega)} 
 =  D(x,y) \|\varphi\|_{\Hcal_K(\Omega)}, 
\qquad x, y \in \Omega. 
\endaligned
\ee
Since $K$ and thus $D$ are continuous in $\Omega$, we have the embedding  
$\Hcal_K(\Omega) \subset C(\Omega)$, 
that is, all of the functions are continuous, at least.  

%---------------------------

\vskip.05cm \par{\bf Mercer decomposition.}
We now consider the linear operator $T_K: L^2(\Omega) \to  L^2(\Omega)$ defined by
$
T_K (a):= \int_\Omega K(\cdot,x) a(x) \, dx$
(for $ a \in L^2(\Omega)$) 
on the Hilbert space $L^2(\Omega)$ endowed with its standard inner product.
We have 
$$
\| T_K (a) \|_{L^2(\Omega)} \leq |\Omega|^{1/2} \| K \|_{L^2(\Omega\times \Omega)} \| a \|_{L^2(\Omega)}, 
$$
so this operator is continuous and self-adjoint. 
It is easily checked to be compact: if a sequence $a^p \rightharpoonup a$ weakly in the $L^2$ norm then $T_K (a^p) \to T_K(a)$ strongly in the $L^2$ norm. The classical spectral decomposition applies and the operator $T_K$ admits an (at most countable) non-increasing sequence of eigenvalues $\lambda_i >0$ and a corresponding set of  eigenfunctions $\zeta_i$ such that
% \bel{equa-567}
$T_K(\zeta_i) = \lambda_i \zeta_i$
and the family $\big\{\zeta_1, \zeta_2, \ldots \big\}$ forms an orthonormal basis of $L^2(\Omega)$ for the $L^2$ inner product.  Furthermore, since the kernel is continuous and bounded, the eigenfunctions $\zeta_i$ are continuous, at least. 

%---------------

We then introduce the kernel 
\bel{eq;shdg}
L(x,y) = \sum_{j =1,2, \ldots}  \lambda_j \, \zeta_j (x)\zeta_j (y), \qquad x,y \in \Omega,  
\ee
in which the sum converges in the $L^2$ sense. This kernel is admissible since, for any collection of points $Y = (y^1, \ldots, y^N)$ in $\RN$, the bilinear form 
$L^Y(x, x') = \sum_{1 \leq n,m \leq N} L(y^n,y^m) x_n x'_m$ (with $x,x' \in \RN$), 
satisfies 
$$
L^Y(x,x) = \sum_{j = 1,2, \ldots}  \lambda_j \sum_{1 \leq n,m \leq N} 
\zeta_j (y^n)\zeta_j (y^m) x_n x_m 
= 
\sum_{j=1,2, \ldots}  \lambda_j \Big| \sum_{1 \leq n \leq N} \zeta_j (y^n) x_n \Big|^2 \geq 0. 
$$ 
In fact, one can check that $L$ coincides with the given kernel $K$ and \eqref{eq;shdg} represents its spectral decomposition. 
The family of functions $\big\{\lambda_i^{1/2} \zeta_i \big\}_{i \geq 1}$ is an orthonormal basis of the space $\Hcal_K(\Omega)$, as follows from the defining relation $T_K \zeta_i = \lambda_i \zeta_i$, namely\footnote{$\delta_{ij}$ being the Kronecker symbol} 
\be
<\lambda_i^{1/2}  \zeta_i, \lambda_j^{1/2} \zeta_j>_{\Hcal_K(\Omega)} 
= \lambda_i^{1/2} \lambda_j^{1/2}
\frac{1}{\lambda_i}<T_K \zeta_i,\zeta_j>_{\Hcal_K(\Omega)} 
= 
{\lambda_j^{1/2} \over \lambda_i^{1/2}} 
\la \zeta_i,\zeta_j \ra_{L^2(\Omega)} = \delta_{ij}. 
\ee
This series representation of the elements in $\Hcal_K(\Omega)$ is referred to as the {\it  Mercer representation} (which may be 
non-unique).  In short, Mercer theorem states that any admissible kernel $K: \Omega \times \Omega \to \RR$ 
can be viewed as a positive, self-adjoint and compact operator on $L^2(\Omega)$. 

%---------------------------

\vskip.05cm \par{\bf Banach spaces.}
Based on the Mercer representation, we can define the $L^p$-based spaces at any order of differentiability. 
Namely, for any $s \geq 0$ and $p \in [1, +\infty)$ we consider the Banach space 
\bel{259}
\Hcal_K^{s,p}(\Omega)
:= 
\big\{
\varphi \in L^p (\Omega) \, / \, 
\sum_{i=1,2, \ldots} \lambda_i^{-ps/2} 
\big|
\la \varphi,\zeta_i \ra_{L^2(\Omega)}
\big|^p  
< +\infty \big\}
\ee
endowed with the norm 
$
\big(\| \varphi\|_{\Hcal_K^{s,p}(\Omega)} \big)^p 
:= \sum_{i=1,2, \ldots} \lambda_i^{-ps/2} 
\big|\la \varphi,\zeta_i \ra_{L^2(\Omega)} \big|^p.
$
When $p$ is chosen to be $2$, the space $\Hcal_K^s(\Omega):= \Hcal_K^{s,2}(\Omega)$ is a Hilbert space endowed with the inner product 
\be
\la f,g \ra_{\Hcal_K^s(\Omega)}
= \sum_{i=1,2, \ldots}  (\lambda_i )^{-s} \la f,\zeta_i \ra_{L^2(\Omega)} \la g,\zeta_i \ra_{L^2(\Omega)}. 
\ee
In particular, $\Hcal_K(\Omega) = \Hcal_K^1(\Omega) = \Hcal_K^{1,2}(\Omega)$ and we recover the Hilbert space defined earlier.

%===========================================================================

\section{Methodology for defining classes of kernels}
\label{section--3} 

\subsection{Translation-invariant and radially-symmetric kernels on $\RD$}  

\vskip.05cm \par{\bf Translation invariant kernels.}
Our first task now is to provide some preliminary material and a classification involving the notions of translation-invariant kernel, periodic kernel, tensorial kernel, and radially-symmetric kernel. In the present section we allow ourselves to 
introduce kernels defined on the whole of $\RD$, although Section~2 was restricted to kernels defined on a bounded set; namely, we will use kernel defined on $\RD$ only for defining the kernels of interest defined on a bounded set. 

We begin with the class of the form $K(x,y) = \chi(x-y)$ (the examples in Table \ref{FDK} being of this type), which we refer to as 
\bse
{\it translation-invariant kernels}  
\bel{equa-tranbslK} 
\aligned
& K(x,y) = \chi(x-y), \quad x, y \in \RD.   
\endaligned
\ee
This family is parametrized by a {\it generating function} $\chi: \RD \to \RR$, 
which (after normalization) must satisfy (as we check below)
\bel{eq:shdh6}
\aligned
& \chi(0) = 1, \qquad  \chi(-x) = \chi(x),  &&& x \in \RD,
\\
& \chih(\xi) \geq 0 \quad &&& \xi \in \RD. 
\endaligned
\ee 
\ese

% Hence, we can state the following positivity condition on the Fourier transform of the generating function of a translation-invariant kernel in $\RD$: 
%More precisely, in order for \eqref{equa-tranbslK} to actually determine an {\sl admissible} kernel, we need to give conditions under which the desired positivity property is guaranteed. The so-called Bochner theorem precisely addresses this question when 
%
%Namely, the connection between the class of positive measures (in Fourier variables) and admissible kernels is as follows. Any translation-invariant admissible kernel is the Fourier transform of a probability measure, say $\mu: \RD \to \RR_+$, that is,  \be
%K(x,y) = \muh(x-y), 
%\qquad 
%\muh(x): =  \int_\RD e^{-i x \cdot \xi} \, d\mu(\xi), 
%\qquad x, y \in \Omega. 
%\ee 

%----------------------------
 
\vskip.05cm \par{\bf Positivity property.}
Under the assumption \eqref{eq:shdh6}, let us consider a collection $Y = (y^1, \ldots, y^N)$ in $\RD$ and the associated bilinear form 
$K(Y,Y)(\xi, \xi) = \sum_{1 \leq n,m \leq N} K(y^n,y^m) \xi_n  \xi_m$
for 
$\xi \in \RN$. 
We obtain the positivity property  
$$
\sum_{1 \leq n,m \leq N} 
\big\la \chih, e^{-i<y^n-y^m,\cdot>} \big\ra_{\Dcal',\Dcal} \, \xi_n  \xi_m 
=  
\big\la \chih, \big| \sum_{n} \xi_n e^{-i<y^n,\cdot>} \big|^2 \big\ra_{\Dcal',\Dcal} \geq 0,
$$
in which $\la \cdot,\cdot\ra_{\Dcal',\Dcal}$ is the duality braket for distributions. 
We thus find 
$$
\aligned
& \iint_{\RD \times \RD} K(x,y) \varphi(x) \varphi(y) \, dxdy
  = \iint_{\RD \times \RD} \chi(x-y) \varphi(x) \varphi(y) \, dxdy 
\\
& = \int_{\RD} (\chi \star \varphi)(x)\varphi(x) \, dx
= \int_{\RD} \widehat{(\chi \star \varphi)}(\xi) \varphih(\xi) \, d\xi
=
\int_{\RD} |\varphih|^2 \chih \, d\xi \ge0, 
\endaligned
$$
or equivalently  
% \bel{MERCER}
$ \iint_{\RD \times \RD} K(x,y) \varphi(x) \varphi(y) \, dxdy \geq 0$.

%-------------------------------------------------------------------------

\vskip.05cm \par{\bf Radially-symmetric kernels.}
Among translation-invariant kernels, the class of {\it radially-symmetric kernels} is an important subclass and corresponds to the case where the function $\chi$ that only depends on the modulus of its argument only, that is, 
\bel{equa-tranbslK-r} 
K(x,y) = \chi(|x-y|), \quad x, y \in \RD, 
\ee
where the generating function $\chi: \RR \to \RR$ is now regarded as a function of a real variable. 
%
%Observe that the positivity condition requires a condition on the function $\chi$. Schoenberg's theorem, that shows that, if a function $\chi$ is used to define a kernel via \eqref{equa-tranbslK-r} for all dimensions $D$, then it cannot have a compact support.
%Furthermore, a radially-symmetric kernel $K(x,y) = \chi(|x-y|)$ generated by a continuous function $\chi:[0,\infty) \mapsto \RR$ 
%is positive if and only if $\chi$ is a linear superposition of Gaussian kernels, i.e.~for some probability measure $\lambda$ on $\RR_+$ 
%\be
%\chi(r) = \int_0^{+\infty} e^{-r^2s^2} \, d\lambda(s). 
%\qquad r >0. 
%\ee 

%------------------------------------------------------------------------------------------------------------------------

\subsection{Localization via scaling: the transported kernels}

The  translation-invariant fail to be sufficiently localized (say in the sense that $K$ fails to be 
$L^1(\RD \times \RD)$.), and we now explain how a transportation map can be applied in order to ``localize'' a kernel to a bounded set. 

\bse

\begin{proposition} \label{TRK}
Let $\Kseed: \RD \times \RD \to \RR$ be a bounded admissible kernel and $\mu$ be a sufficiently regular, probability measure such that $\supp (\mu)$ is a convex set (with, therefore, $\mu \ge 0$ and $\mu(\RD)=1$). Then the kernel
%%%%%%%%%%%%%%\footnote{the subscript $T$ stands for ``localized''.} 
\bel{TLOC}
K^\loc(x,x'):= \Kseed(x,x') \mu(x)\mu(x'), \qquad x, x' \in \RD, 
\ee
is admissible and belongs to $L^1(\RD \times \RD)$. Moreover, let $\Omega \subset \RD$ be any open and convex subset with normalized volume $|\Omega| = 1$, and consider a transportation map for the measure $\mu$, that is, a one-to-one map $S: \Omega \to \supp(\mu)$ such that $S=\nabla h$ for some convex function $h: \Omega \to \RR$ and $S_\# \mu = dy$ (the Lebesgue measure). Then, 
\bel{TK}
K^\trans(y,y'):= \Kseed(S(y),S(y')), \qquad y, y' \in \Omega
\ee
defines an admissible kernel, referred to as the {\bf transported kernel} associated with $\Kseed$ and $\mu$.  
Furthermore, for any $\varphi \in L_\mu^1(\RD)$ (the set of functions that are integrable for the measure $\mu$) and any choice of points $x^n = S(y^n)$ one has 
\bel{eq:20D-more} 
\int_\RD  \varphi(x) \, d\mu(x) - {1 \over N}  \sum_{1 \leq n \leq N}  \varphi(x^n) 
= 
\int_\Omega  (\varphi \circ S)(y) \, dy - {1 \over N}  \sum_{1 \leq n \leq N}  (\varphi \circ S)(y^n). 
\ee
\end{proposition} 

\ese

In the statement above, the transportation maps satisfies $\int_\RD \varphi d\mu = \int_\Omega (\varphi \circ S) dy$ for any continuous function $\varphi \in L_\mu^1(\RD)$. Provided $S$ is sufficiently regular, one has $S_\# \mu = |\det \nabla S| \circ S^{-1} dy$, where $|\det \nabla S|$ is the Jacobian of $S$. The convexity of $\supp (\mu)$ ensures that $S$ is one-to-one and continuous from $\Omega$ onto the support set $\supp(\mu)$. 
Furthermore, thanks to the localization argument above, deriving an error estimate associated with the left-hand side of \eqref{eq:20D-more} reduces to \eqref{eq:20D}, and the role of the measure $\mu$ is eliminated. 
Within the framework of kernel spaces, the relation between the $\mu$-weighted norm on $\RD$ and the un-weighted norm on the bounded set $\Omega$ is as follows (with obvious notation): 
$$
\|\varphi \|_{\Hcal_{\Kseed, \mu}^{s,p}(\RD)}
= 
\|\varphi \|_{\Hcal_{K^{\loc, s,p}}(\RD)}
= \|\varphi \circ S \|_{\Hcal_{K^{\trans,s,p}}(\Omega)}.
$$

\begin{proof} For any function $\varphi \in L_\mu^1(\RD)$ we have 
$\iint_{\RD \times \RD} \varphi(x) \varphi(y) d\mu(x) d\mu(y) = \Big( \int_{\RD } \varphi d\mu \Big)^2$, which is positive.  Hence, being the product of two admissible kernels, we see that the kernel $K^\loc$ is admissible. The transported kernel $K^\trans$ is also admissible since for all relevant $\psi$ 
$$
\iint_{\Omega \times \Omega} K^\trans(y,y') \psi(y)\psi(y') dy dy'
= \iint_{\RD \times \RD} K^\loc(x,x') (\psi \circ S^{-1}) (x) (\psi \circ S^{-1}) (x') \, 
d\mu(x) d\mu(x') \geq 0.
$$ 
\end{proof}

%---------------------------------------------------------------------------------------------------------------------------------

\subsection{Localization via periodization: the periodic kernels}
 \label{loc-peri}

 \par{\bf A discrete lattice.}
Motivated by Cohn and Elkies's work \cite{CohnElkies} on the sphere packing problem, we propose to embed the support of a general kernel in a periodic lattice. By periodicity, we can always extend to $\RD$ the kernel defined to an elementary cell. Importantly, the terms arising in the spectral decomposition can be controlled almost explicitly, thanks to a Poisson decomposition formula associated with the lattice. 

A family of $D$ vectors $l_1, l_2, \ldots, l_D \in \RD$ being given, we consider their convex hull $\Cbf \subset \RD$ which serves as the fundamental cell of our discrete lattice and whose volume is denoted by $| \Cbf |$. 
By suitably translating $\Cbf$, we thus generate the periodic lattice 
$
\Lbf:= \Big \{\sum_{1 \leq d \leq D} \alpha_d  l_d \, \big/ \,  \alpha = (\alpha_1,\ldots,\alpha_D) \in \ZZ^D \Big \}, 
$
and we denote its dual by 
$
\Lbf^*:= \Big\{\alpha^* \in \RD \, \big/ \,  <\alpha, \alpha^*> \in \ZZ \, \text{ for all } \alpha \in \Lbf \Big \}. 
$
We denote by $l_1^*, \ldots, l_D^*$ the vectors generating the elementary cell $\Cbf^* \subset \RD$ of the dual lattice. 
Next, we introduce the discrete Hilbert space $l^2(\Lbf^*)$ consisting of all functions defined on the dual lattice, endowed with the inner product 
\be
\la  f, g \ra_{l^2(\Lbf^*)} 
:= {1 \over |\Cbf|} \sum_{\alpha^* \in \Lbf^*} f(\alpha^*) g(\alpha^*), 
\qquad f, g: \Lbf^* \to \RR. 
\ee

%------------------ 

\vskip.05cm \par{\bf Poisson formula.}
Consider the Fourier transform of a real-valued function $\varphi: \RD \to \RR$ defined on $\RD$ and let us restrict it to the dual lattice, that is, consider the discrete values 
\be
\varphih(\alpha^*):= \int_\RD \varphi(x)\, e^{-2i \pi <x,\alpha^*>} \, dx.  
\ee
Then, the so-called Poisson formula reads 
\bel{PF}
\sum_{\alpha \in \Lbf} \varphi(\alpha+x) 
= {1 \over | \Cbf |} \sum_{\alpha^* \in \Lbf^*}  e^{2i \pi <x,\alpha^*>} \varphih(\alpha^*), 
\qquad x \in \RD.
\ee
{\sl Provided} the function $\varphi$ is supported on the cell $\Cbf$, the sum in the left-hand side contains a {\sl single term} and, therefore, 
\bse
\label{PF-deux}
\be
\varphi(x) = {1 \over | \Cbf |} \sum_{\alpha^* \in \Lbf^*}  e^{2i \pi <x, \alpha^*>} \varphih(\alpha^*), 
\qquad 
x \in \Cbf, 
\, 
\text{ provided } \supp (\varphi)\subset \Cbf 
\ee
or, with our notation, 
\bel{PFL}
\varphi(x) = \big\la  e^{2i \pi  <x, \cdot>}, \varphih  \big\ra_{l^2(\Lbf^*)}, 
\qquad
x \in \Cbf,
\, 
\text{ provided } \supp (\varphi) \subset \Cbf. 
\ee  
\ese
Hence, a function defined on the cell $\Cbf$ can be recovered  (via an discrete inverse Fourier transform) from the values of its Fourier transform on the dual lattice $\Lbf^*$. 

%-----------------------  

More generally, let us derive an identity that will be useful to us later on. 
Consider now a collection $(y^1, \ldots, y^N)$ of points in $\RD$ and, in view of \eqref{PF-deux}, let us write 
$$
\aligned
\sum_{1 \leq n,m \leq N} \sum_{\alpha \in \Lbf} \varphi(\alpha + y^n - y^m) 
& =  {1 \over | \Cbf |} \hskip-.15cm \sum_{1 \leq n,m \leq N} 
\sum_{\alpha^* \in \Lbf^*}  
\hskip-.15cm
e^{2i \pi < y^n - y^m,\alpha^*>} \varphih(\alpha^*)
  =  
{1 \over | \Cbf |} \hskip-.15cm \sum_{\alpha^* \in \Lbf^*}  
 \Big| \sum_{1 \leq n \leq N} \hskip-.15cm
e^{2i \pi <y^n,\alpha^*>} \Big|^2 \varphih(\alpha^*). 
\endaligned
$$
We can arrange that, in the left-hand side, the sum over $\alpha \in \Lbf$  reduces to a single term obtain when $\alpha = 0$.
We reach the following conclusion. 

\begin{lemma} 
\label{lem-observ}
For any function $\varphi$ supported on the elementary cell $\Cbf$ of a lattice $\Lbf$, that is, 
$\supp (\varphi)\subset \Cbf$ and for any finite collection $(y^1, \ldots, y^N)$ satisfying 
the ``localization property''  
$y^n - y^m \in \Cbf$ for all $n,m= 1, 2, \ldots, N$, 
the following identity holds: 
\bel{610-single}
\aligned
& \sum_{1 \leq n,m \leq N} \varphi(y^n - y^m) 
 = {1 \over | \Cbf |} \sum_{\alpha^* \in \Lbf^*}  
\Big| \sum_{1 \leq n \leq N} e^{2i \pi <y^n, \alpha^*>} \Big|^2 \varphih(\alpha^*).
\endaligned
\ee 
\end{lemma} 

%--------------------------------------------------------------------------------------------------------------------

\vskip.05cm \par{\bf Periodic kernels.}   
The interest of the following class of kernels lies in  the fact that their spectral decomposition can be determined almost explicitly, in terms of exponential functions defined on the dual lattice. Namely, Lemma~\ref{lem-observ} allows us to pass from the continuous physical variables on $\Cbf$ to the discrete Fourier variables on $\Lbf^*$. The role of the function $\rho$ introduced below 
is going to be played by (the restriction to the lattice of) the Fouier transform of an arbitrary kernel. 
Indeed, our definition below provides a way to transform a ``seed'' kernel $K^\seed$ defined on $\RD$ to a periodic kernel $\Kper$ defined on the lattice cell $\Cbf$. 

\bse

\begin{proposition}[Periodic kernels associated with a generating function]
\label{PKFT} 
Consider a discrete lattice $\Lbf \subset \RD$ generated from an elementary cell $\Cbf$, and let $\rho: \Lbf^* \to (0, +\infty)$ be a positive, integrable, and even function, that is,  
\be 
\rho^{1/2} \in \ell^2(\Lbf^*), \qquad 
\rho(-\alpha_*) = \rho(\alpha_*) \geq 0 \,
 \quad (\text{ with } \alpha_* \in \Lbf^*). 
\ee
Then, the discrete Fourier transform of $\rho$ extended to the whole of $\RD$, that is, 
\bel{261}
\Kper(x,y) := \big\la  e^{2i \pi <x-y, \cdot>}, \rho \big\ra_{\ell^2(\Lbf^*)}, 
\qquad x,y \in \RD
\ee
defines an admissible kernel on $\RD$ which is periodic with period $\Cbf$ and its associated Hilbert space is 
\be 
\Hcal_{\Kper}(\Cbf) 
= \Big\{\varphi \in C(\RD) \, \big/ \, \Cbf\text{--periodic}  \, / \, \varphih \, \rho^{-1/2} \in \ell^2(\Lbf^*) \Big\} 
\ee
endowed with the norm 
$\la f, g \ra_{\Hcal_{\Kper}(\Cbf)} = \la \widehat f \rho^{-1/2}, \widehat g \rho^{-1/2} \ra_{\ell^2(\Lbf^*)}$. 
\end{proposition}

\ese

More generally, provided $\rho^s \in \ell^{p}(\Lbf^*)$ for some $p \in [1, +\infty)$ and $s \geq 0$, we can also introduce the  Banach space  
\bel{58}
\Hcal_{\Kper}^{s,p}(\Cbf) 
= \Big\{\varphi \in C(\RD) \, \big/ \, \Cbf\text{--periodic}    \, \big/ \, \varphih \rho^{-s/p} \in \ell^{p'}(\Lbf^*)
\Big\}. 
% \quad s \ge 1
\ee
% embeds into the space of continuous functions (here $1/p + 1/p'=1$). 

\begin{proof} Since $\rho^{1/2} \in \ell^2(\Lbf^*)$ it is clear that the expression
$$
\Kper(x,y) =  {1 \over | \Cbf |} \sum_{\alpha^* \in \Lbf^*}  e^{2i \pi <x-y, \alpha^*>} \rho(\alpha^*)
$$
is finite and, in fact, globally bounded on $\RD$. It is symmetric since $\chi$ is even. The positivity condition follows from  
$$
\aligned
\sum_{n, m} a_n a_m \Kper(y^n,y^m) 
& =  
\sum_{n, m} a_n a_m  \big\la  e^{2i \pi <y^n - y^m, \cdot>}, \rho \big\ra_{\ell^2(\Lbf^*)}
 =  \sum_{\alpha_* \in \Lbf^*} \Big| \sum_n 
e^{2i \pi <y^n, \alpha_*>} a_n \rho(\alpha_*) \Big|^2  \geq 0.
\endaligned 
\qedhere
$$
\end{proof}

%---------------------------------------------------------------------------------------------------------------

\subsection{Further generating techniques}

\vskip.05cm \par{\bf The tensor technique.}
\bse
A broad class of examples on $\RD$ can be obtained by tensor decomposition from an admissible kernel in one dimension, say $\Kseed: \RR \times \RR \to \RR$, namely 
\be
K(x,y) = \prod_{1 \leq d \leq D} \Kseed(x_d, y_d), \qquad x, y \in \RD. 
\ee 
This applies, particularly, to a translation-invariant kernel, in which case we choose any even function $\chi^\seed: \RR \to \RR$ and set 
\be
K(x,y) = \prod_{1 \leq d \leq D} \chi^\seed(x_d - y_d), \qquad x, y \in \RD. 
\ee
\ese

%---------------------------

\vskip.05cm \par{\bf An example: the tensorial truncated kernel.} 
\bse
The kernel $K_T(x,y) = \big( 1 - |x-y| \big)^+$ 
has the form above and 
can also be expressed as a convolution, namely 
\be
K_T(x,y) = \chi(x-y) = (\lambda \ast \lambda)(x-y), 
\qquad \lambda(x) = 1_{[-1/2,1/2]}(x), \qquad x \in \RR. 
\ee
Here, we have $\chih(\xi) = \hlambda^2(\xi)$ and $\hlambda(\xi)= \frac{\sin(2\pi\xi)}{2\pi\xi}$. 
More generally, in dimension $D$ we consider 
\bel{eq=notre noyauD}
K(x,y) =  \prod_{1 \leq d \leq D}  \Kseed(x_d,y_d) = \prod_{1 \leq d \leq D} \big( 1 - |x_d-y_d| \big)^+, 
\ee
written also as a convolution 
$K(x,y) = (\lambda \ast \lambda)(x-y)$ with 
$\lambda(x) = 1_{[0,1]^D}(x)$. 
Again, we have $K(x,y) = \chi(x-y)$ with $\chih(\xi) = \hlambda^2(\xi)$, and now 
$\hlambda(\xi) = \prod_{1 \leq d \leq D}  \frac{\sin(2\pi\xi_d)}{2\pi\xi_d}$. 
\ese 

%------------------------------------------------

\vskip.05cm \par{\bf The normalization technique.}
The transformations below can also  serve as building blocks in order to adapt existing examples to a particular application. If $K$ is an admissible kernel on $\RD$ we can consider
\bel{equa-Ktt}
\Knorm(x,y) = {K(x,y) \over K(x,x)^{1/2}K(y,y)^{1/2}}, \qquad x, y \in \RD. 
\ee
Clearly, we have $\Knorm(x,x) = 1$ and the required admissibility conditions are easily checked. 
The coefficients $a_n$ and $\anorm_n$ of the corresponding decomposition \eqref{equa:HKY} of a function 
in $\Kcal_K(\RD)$ and $\Kcal_{\Knorm}(\RD)$, respectively, are related by 
$ \anorm_n = a_n K(y^n, y^n)^{1/2}$, so that the two spaces are quite similar from the application standpoint. 
We thus regard \eqref{equa-Ktt} as a normalization procedure. 

%---------------------------

\vskip.05cm \par{\bf Taking sums and products.}
If $K_1, K_2$ are admissible kernels in $\RD$, then it is also easily checked that, $a,b > 0$ being some given constants, 
\be
K_3(x,y) = a K_1(x,y) +b K_2(x,y), 
\qquad 
K_4(x,y) = K_1(x,y) K_2(x,y), \qquad 
\ee
With the notation used in Section \ref{section--2} the matrices $K_3(Y,Y)$ and $K_4(Y,Y)$ are symmetric positive definite, as follows easily from a standard linear algebra argument.

%--------------------- 

\vskip.05cm \par{\bf Zonal kernels.}
If a function $\phi$ is such that $(x,x') \mapsto \phi(xx')$ is a one-dimensional kernel, then the following formula 
\be
K_{\text{zonal}}(x,y) = \phi(<x,y>), \qquad x, y \in \RD 
\ee
(where $<x,y>$ stands for the Euclidian inner product) 
defines an admissible kernel. This class is often used by the artificial intelligence community.
  
%--------------------- 

\vskip.05cm \par{\bf Convolution kernels.}  
\bse
Another class of translation-invariant kernels can be generated by choosing a function 
\be
\lambda \in \L^2(\RD),
\quad  
\int_\RD \lambda(x) \lambda(-x) \, dx \neq 0, 
\quad
\int_\RD (\hlambda)^2\, d\xi = 1 
\ee
and then defining our generating function $\chi$ by the convolution formula   
\be
K(x,y):= \chi(x-y) = (\lambda \star \lambda)(x-y),
 \qquad (\lambda \star \lambda)(y) = {\int_{\RD} \lambda(x) \lambda(y-x) \, dx \over  \int_\RD \lambda(x) \lambda(-x) \, dx}, 
\qquad y \in \RD. 
\ee 
Indeed, it is clear that $\chi(0) = 1$ and $\chih = (\hlambda)^2 \geq 0$, so that the Fourier transform of $\chi$ is a probability measure. This class of kernels is used in machine learning, for instance in combination of multi-layer neural networks.
\ese

%=========================================================================================

\section{Designing kernels on a bounded domain}
\label{sec-select}

\subsection{Standard kernels taken as seed data}
\label{section-43}

We focus our attention to the standard examples listed in Table \ref{FDK}, that are radially-symmetric, translation-invariant kernels $K(x,y) = \chi(x-y)$. 
 
\vskip.05cm \par{\bf Exponential kernel $K_E$.} 
The choice $\chih_E(\xi) = (1 + |\xi|^2)^{-m}$ (with $m > D/2$) leads to the standard Sobolev space $\Hcal_{K_E}(\RD) = W^{m,2}(\RD)$, and is a standard choice in the numerical analysis literature. 

%---------------------------

\vskip.05cm \par{\bf Multiquadric kernel $K_M$.} 
The choice $\chih_M(\xi) = e^{-|\xi|}$ is only Lipschitz continuous at the origin and is relevant for representing sufficiently smooth functions with polynomial decay, hence provides (slightly) more information than the Gaussian one (below).

%---------------------------

\vskip.05cm \par{\bf Gaussian kernel $K_G$.} 
The choice $\chih_G(\xi) = e^{-|\xi|^2}$ provides an exponential decay in both the Fourier and the physical spaces. Both functions $\chih_G \geq 0$ and $\chi_G \geq 0$ are globally positive. The Gaussian kernel is adapted to the description of smooth and fast decaying functions which have``almost'' compact support in physical and Fourier variables. Hence, the function space $\Hcal_{K_G}(\RD)$ is ``small'' and, in term, provides limited ``information'' on the functions.  

%---------------------------

\vskip.05cm \par{\bf Truncated kernel $K_T$.} 
A more interesting and also quite standard choice is obtained by truncation in the physical space, namely,
 $\chi_T(x) = (1- |x|_1)_+^l$ (with $l \geq D/2$) 
where the notation $a_+:= \sup(a,0)$ stands for the positive part. This kernel is only Lipschitz continuous in the physical space. 

%---------------------------------------------------------------------------------------------------------------------------------

\subsection{Periodic kernels of interest} 
\label{ELBK}

\vskip.05cm \par{\bf Objective.}
We now present, in their rescaled form, the four periodic kernels already listed in Table \ref{FDK2}
(and associated with each of the examples in Table~\ref{FDK}). Periodic kernels are translation-invariant, i.e. $K^\per(x,y) = \chi^\per(x-y)$ and we plot the corresponding functions $\chi^\per$ in Figure~\ref{lbkerns}. Moreover, we also plot the level set of the Fourier transform of $\hchi$ in  Figure~\ref{fig:LEVELSET} for the two-dimensional case. 
For the sake of simplicity in the notation, since the lattice and the dual lattice coincide we simply write $\alpha$
(instead of $\alpha^*$) for a general element of the lattice or dual lattice.

%-----------------------------------------------------------------------------------------------------------------------------------------------
\begin{figure}
%[h]

\centering{\includegraphics[width=0.45\linewidth]{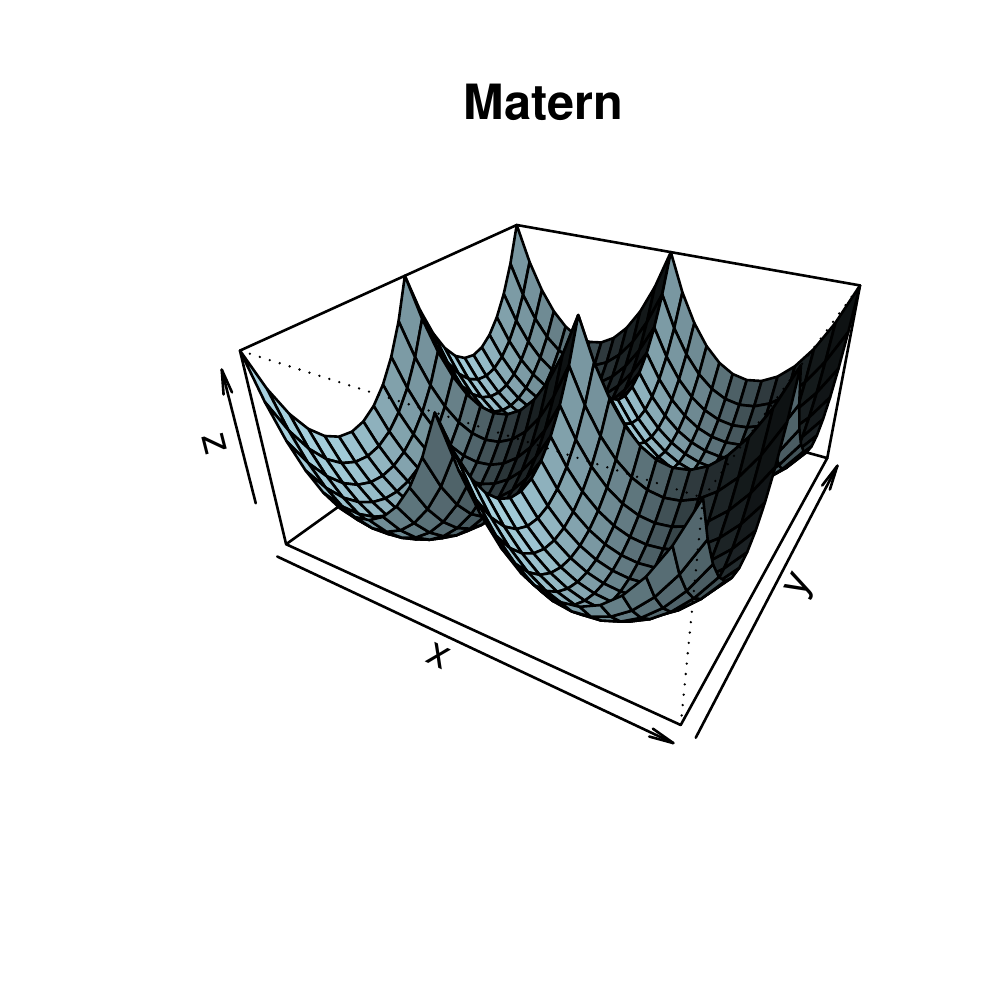} 
% \subcaption{Periodic exponential}
\includegraphics[width=0.45\linewidth]{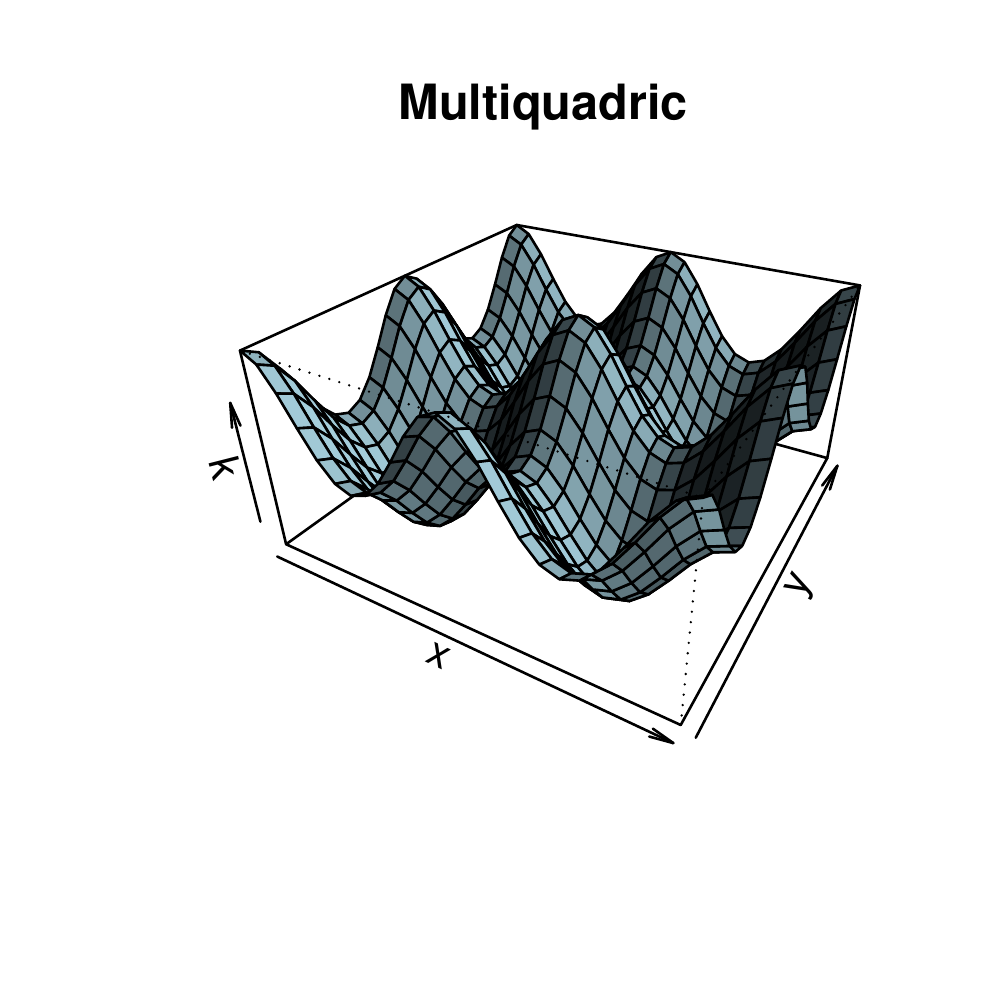} 
\includegraphics[width=0.45\linewidth]{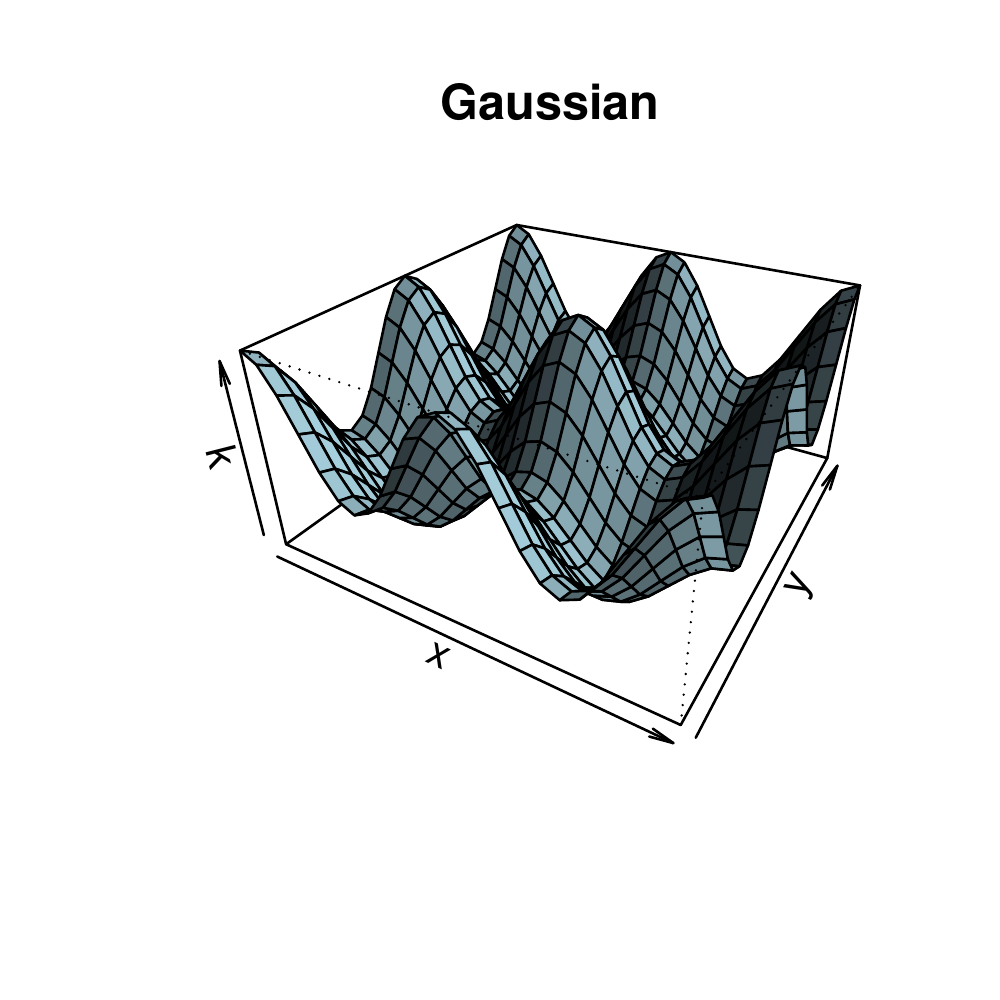} 
\includegraphics[width=0.45\linewidth]{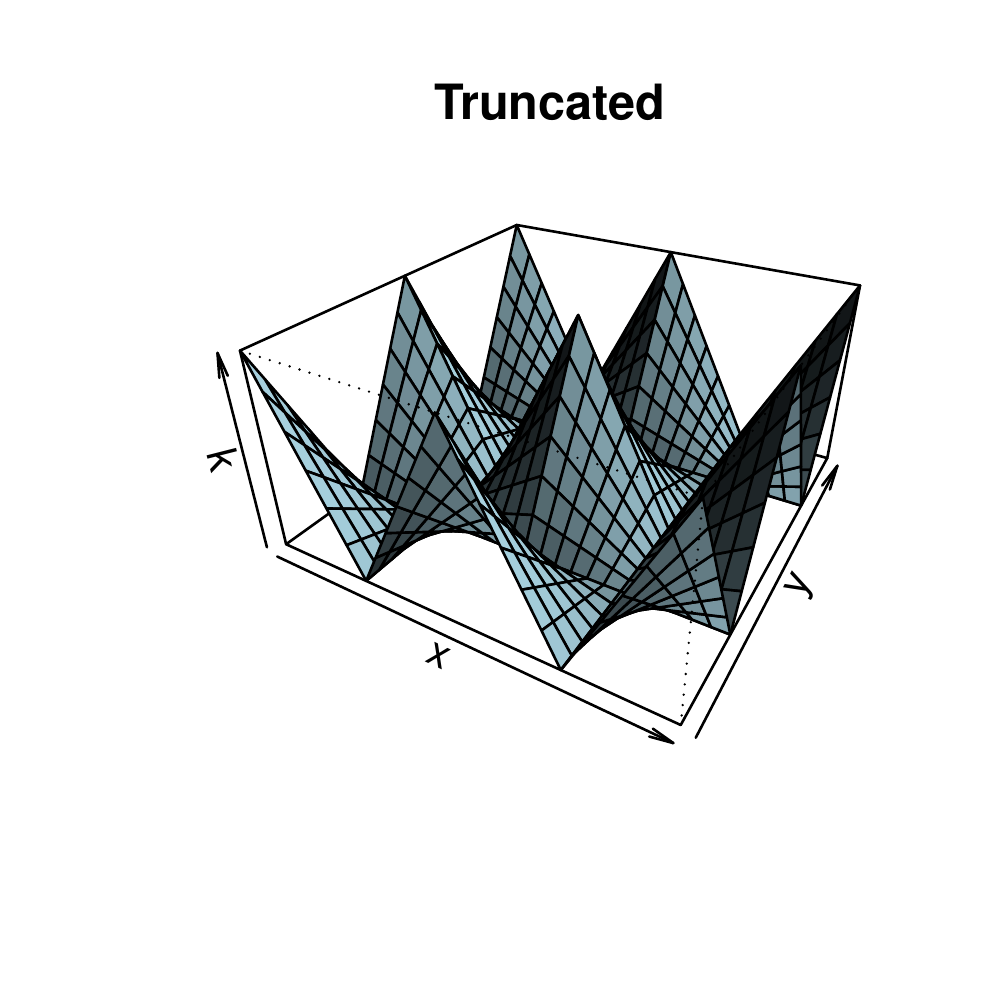} 
}

\caption{\label{lbkerns} Plots of periodic kernels in two dimensions: exponential, multiquadric, Gaussian, and truncated, respectively.}\label{fig:latticebk}
\end{figure}

%-----------------------------------------------------------------------------------------------------------------------------------------------
\begin{figure}

\centering{\includegraphics[width=0.45\linewidth]{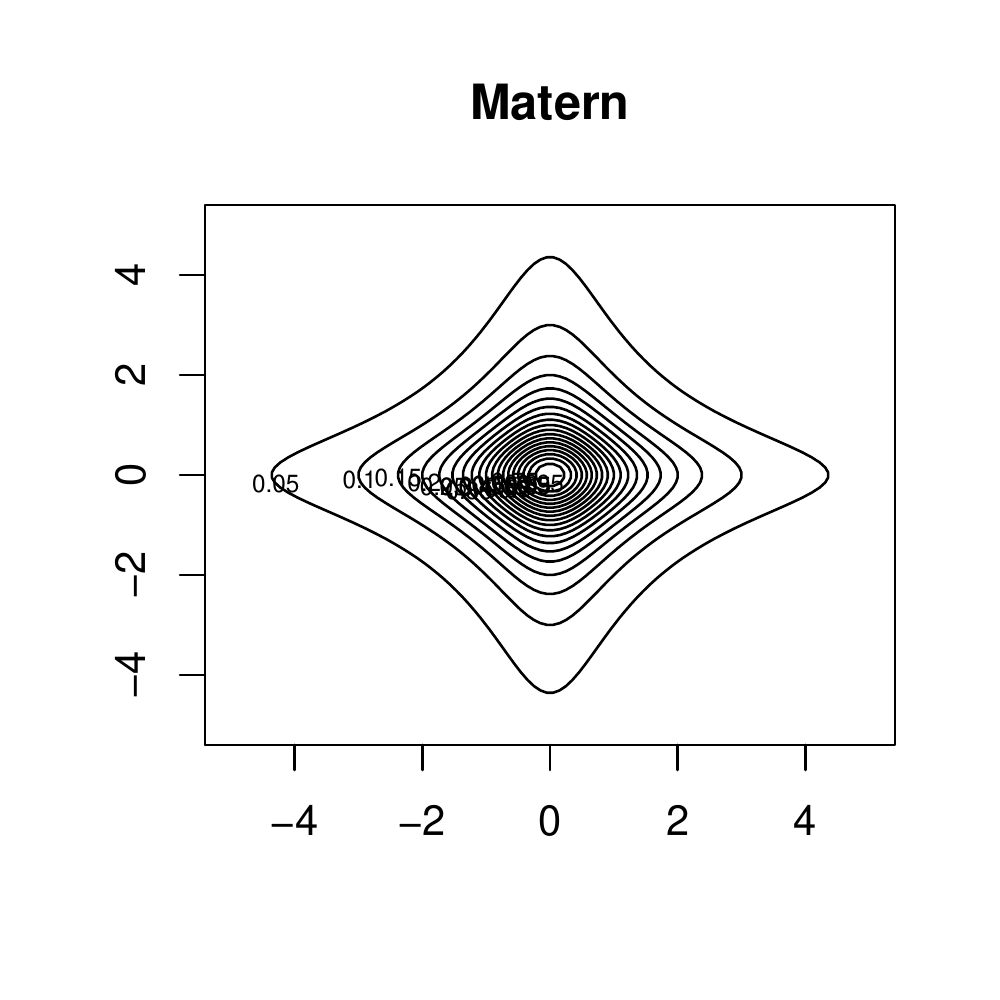} 
\includegraphics[width=0.45\linewidth]{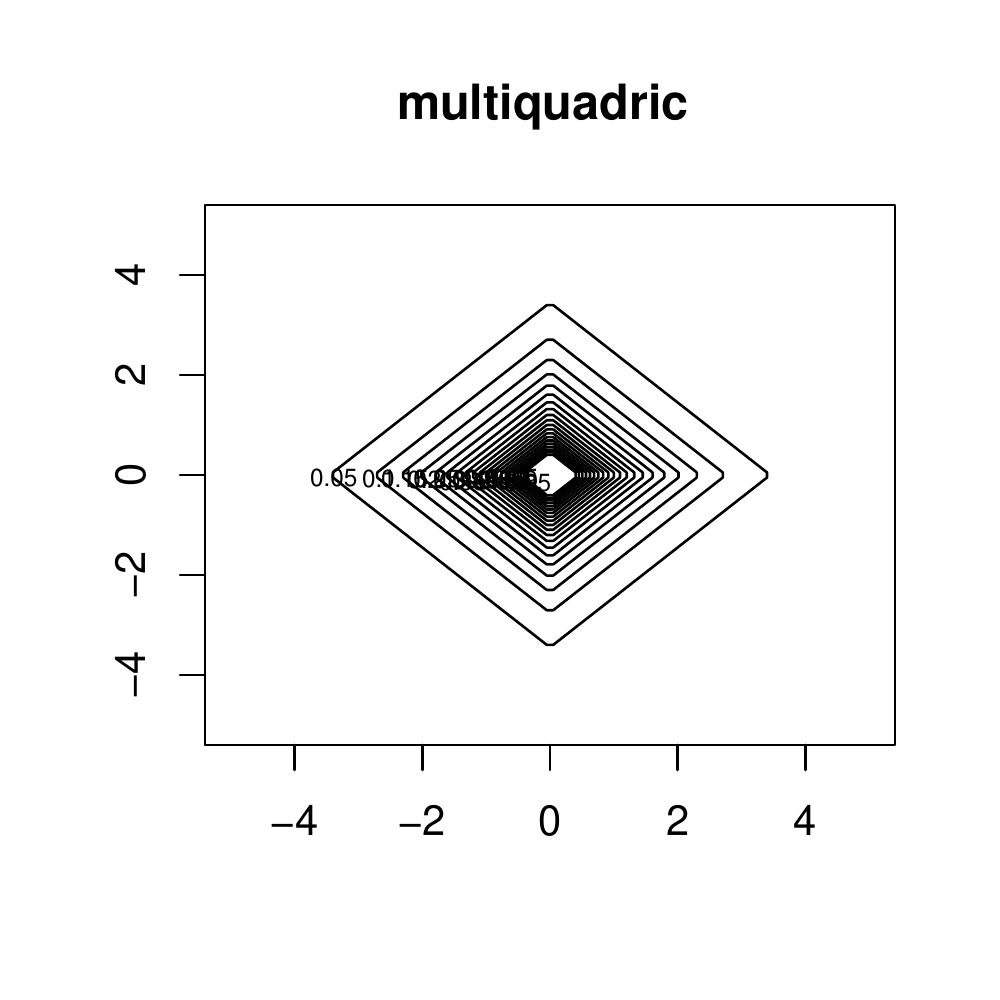} \includegraphics[width=0.45\linewidth]{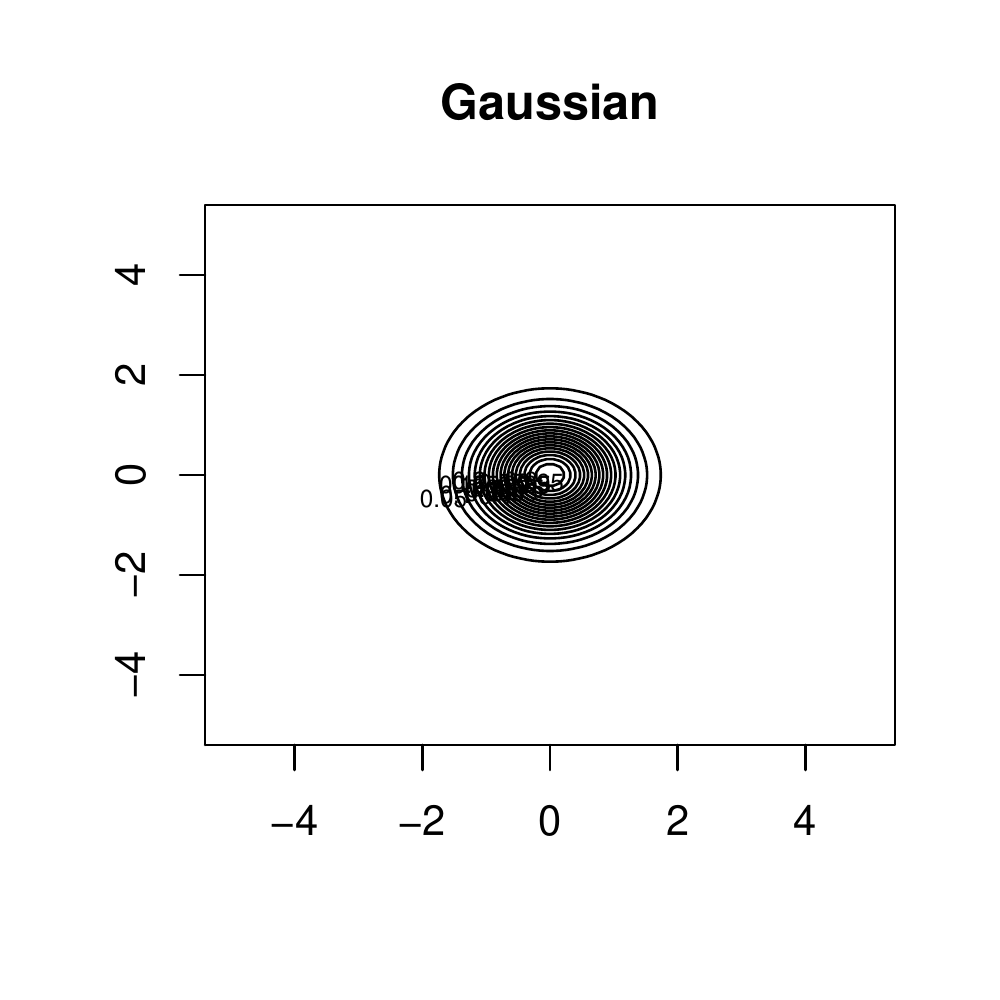} \includegraphics[width=0.45\linewidth]{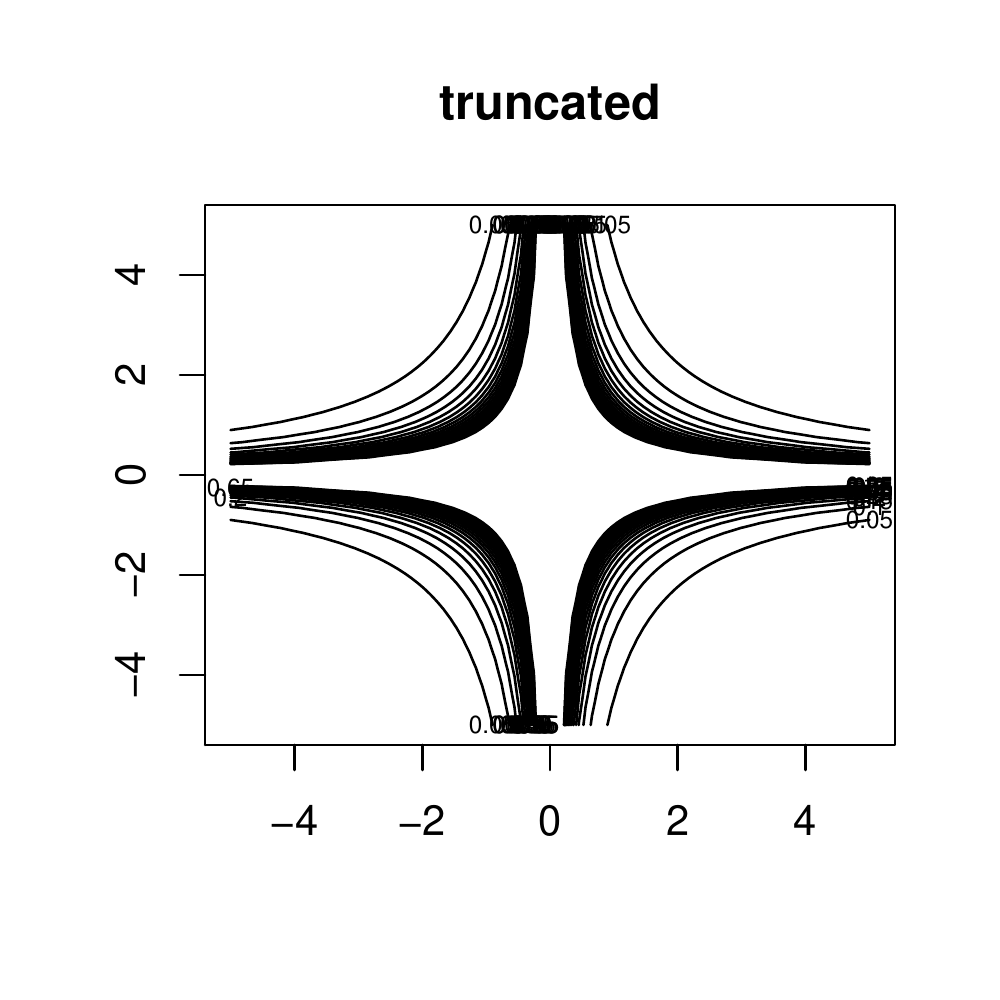} 
}
\caption{\label{LEVESET} Level sets of periodic kernels in two dimensions:  exponential, multiquadric, Gaussian, and truncated, respectively.}\label{fig:LEVELSET}
\end{figure}
%-----------------------------------------------------------------------------------------------------------------------------------------------

\vskip.05cm \par{\bf Periodic tensorial exponential kernel $K^\per_E$.}
\bse
Consider the one-dimensional exponential kernel given by
\bel{exponential}
K_E(x,y) =  \exp(-|x-y|) = \chi_E(|x-y|), \qquad x, y \in\RR,
\ee
with $\hchi_E(\xi)= \frac{2}{1+4 \pi^2 \xi^2}$. 
We make this kernel tensorial and periodic using the localization method in Section~\ref{loc-peri} (see also \eqref{PER} below), that is, for $x,y \in [0,1]^D$  
\be
K_E^\per(x,y) = \chi_E^\per(x-y) = \sum_{\alpha \in \ZZ^D} \frac{e^{2 i \pi <x-y,\alpha>}}{\prod_{d=1}^D (1+4 \pi^2 \alpha_d^2 / \tau_D^2)} 
= \prod_{d=1}^D \sum_{\alpha_d \in \ZZ} \frac{ e^{2 i \pi (x_d-y_d)\alpha_d}}{1+ 4 \pi^2\alpha_d^2 / \tau_D^2}. 
\ee
Here, we have also introduced a parameter $0 < \tau_D \to 0$ as $D \to +\infty$. 
Observe that the Fourier transform of
$\tau_D e^{-\tau_D|x|}$ is
$\frac{2}{1+4 \pi^2 \alpha_d^2 / \tau_D^2}$. Thus, thanks to the
Poisson formula \eqref{PF}, our kernel coincides with 
\be
\aligned
K^\per_E(x,y) 
& = \sum_{\alpha \in \ZZ^D} \tau_D \exp(-\tau_D |x-y + \alpha|_1) 
  = \tau_D \sum_{\alpha \in \ZZ^D}  \prod_{d=1}^D \exp(-\tau_D |x_d-y_d + \alpha_d|) 
\\
& = \prod_{d=1}^D \frac{2 \tau_D}{e^{\tau_D} - 1}\Big(\exp(\tau_D|x_d-y_d|)+\exp\big(\tau_D(1-|x_d-y_d|)\Big).
\endaligned
\ee
We plot the corresponding function $\chi_E^\per$ in Figure~\ref{fig:KERND-MA} for several dimensions. The Banach space associated to this kernel (defined in \eqref{58}) reads 
\bel{MAT1P}
  \Hcal_{K^\per_E}^{s,p}([0,1]^D) = \Big\{\varphi \text{ is periodic on $[0,1]^D$} \, 
\Big/ \,
 \Big\| \Big(\prod_{d=1}^D(1+4 \pi^2\alpha_d^2 / \tau_D^2) \Big)^{s/p}\hvarphi \Big\|_{\ell^{p'}(\ZZ^D)} 
<  + \infty\Big\}. 
\ee 
In particular, the space $\Hcal_{K_E^\per}^{1,1}([0,1]^D)$ has been found to be relevant in mathematical finance.

\ese

%--------------------------------------------------------------------------------------------
\begin{figure}

\centering{\includegraphics[width=0.3\linewidth]{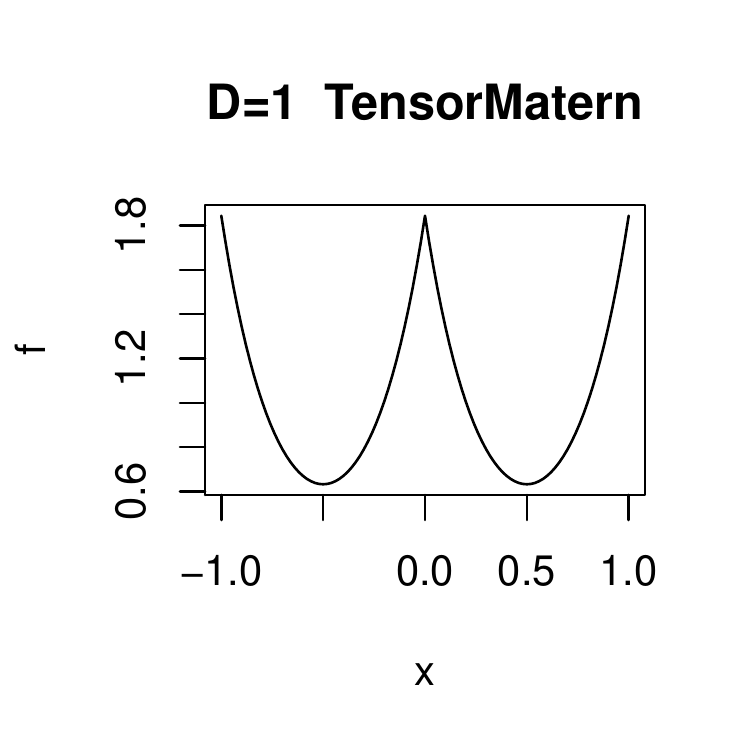} \includegraphics[width=0.3\linewidth]{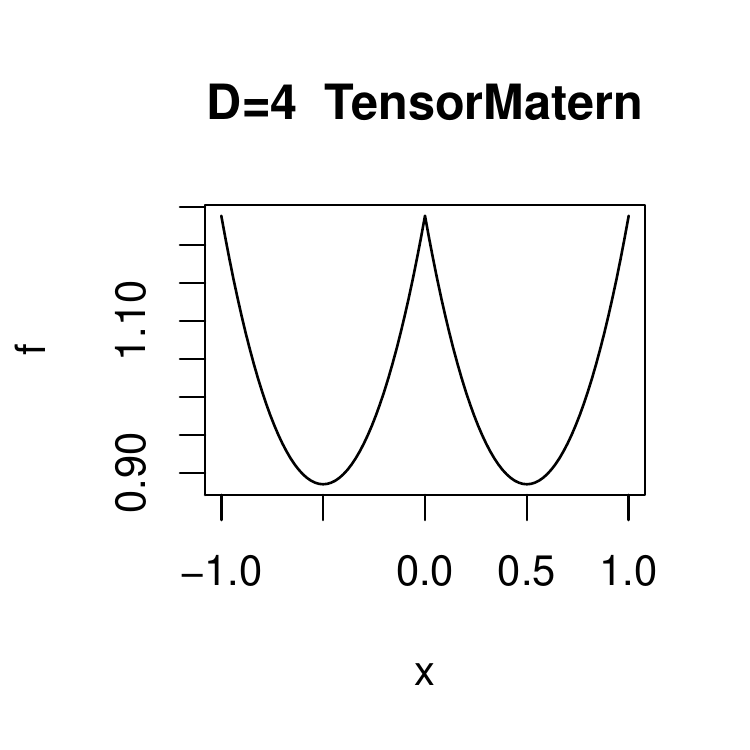} \includegraphics[width=0.3\linewidth]{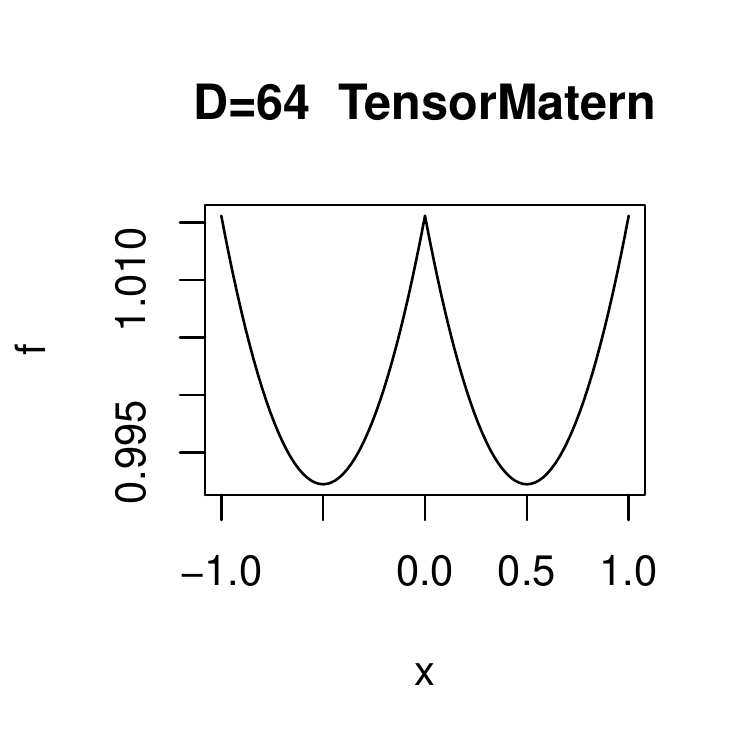} 
}
% \label{MT:figs}  
\caption{Periodic exponential kernel in dimension $D=1,4,64$.}
\label{fig:KERND-MA}
\end{figure}
%--------------------------------------------------------------------------------------------

\vskip.05cm \par{\bf Periodic Gaussian kernel.}
The Gaussian kernel is translation-invariant, radial, and smooth, and reads 
\be
K_G(x,y) = \exp(- |x-y|^2) = \prod_{d=1}^D \exp(-|x_d-y_d|^2) 
= \chi_G(x-y), 
\ee
with $\hchi_G(\xi) = 2^{-D/2} \exp(-|\xi|^2/4)$. 
We define its periodic version as 
\bel{KJ}
K_G^\per(x,y) 
:= \frac{1}{\tau^D} \sum_{\alpha \in \ZZ^D} \prod_{d=1}^D \exp(- \frac{|x_d-y_d + \alpha|^2}{4 \tau^2}) 
=  \prod_{d=1}^D\vartheta_3 (i \pi |x_d-y_d|,\tau_D),  
\ee
where
$\vartheta_3 (z;\tau_D )= \sum _{n \in \ZZ} \big( \exp(- \pi^2 \tau) \big)^{n^2}\cos(2 \pi n z)$
is nothing but the third Jacobi-theta function. Here, $\tau_D$ is a numerical parameter. 
The diagonal term enjoys the following decay 
$
 K_G^\per(x,x)  = \vartheta_3 (0;\tau_D )^D \simeq(1 + 2 e^{-\pi^2 \tau_D})^D, 
$
and we thus choose $\tau_D = \frac{\ln(2D)}{\pi^2}$ to ensure
$K_G^\per(x,x) \le e$. Let us denote (with some abuse of notation) 
$K_G^\per(x,y) = \chi_G^\per(x-y) = \prod_d \chi_G^\per(x_d - y_d)$. We plot the function $\chi_G^\per$ in Figure~\ref{fig:KERND-GAUSSIAN} for several dimensions.

%-------------------------------------------------------------------------------------------------------- 
\begin{figure}

\centering{\includegraphics[width=0.3\linewidth]{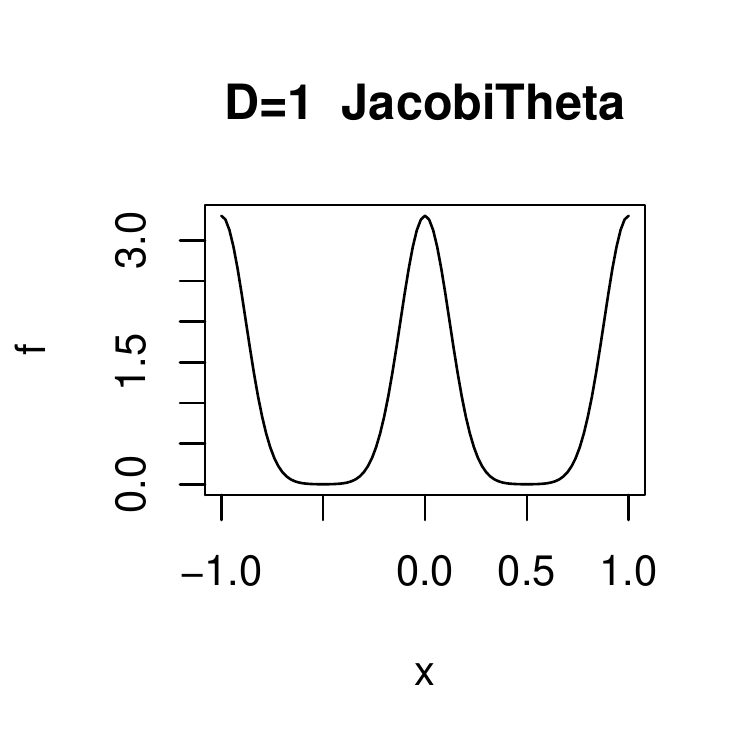} \includegraphics[width=0.3\linewidth]{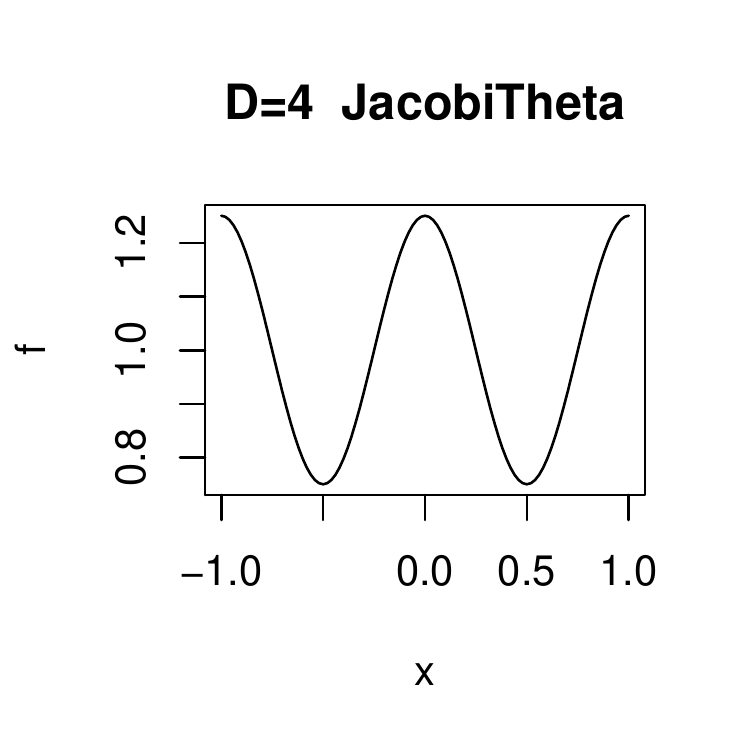} \includegraphics[width=0.3\linewidth]{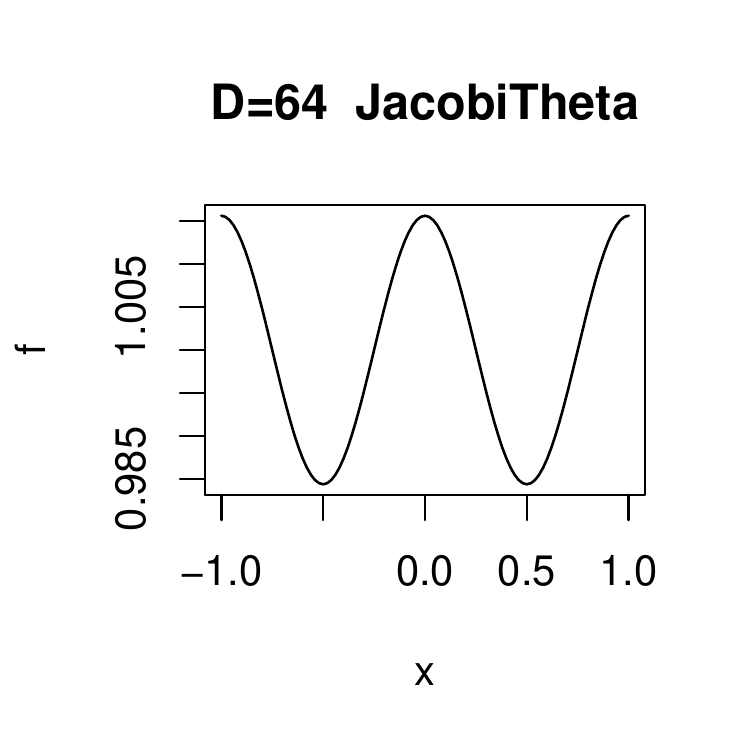} 
}
%% \label{GA:figs} 
\caption{Periodic Gaussian kernel in dimension $D=1,4,64$}\label{fig:KERND-GAUSSIAN}
\end{figure}
%----------------------------------------------------------------------------------------------------------------------------------------------

\vskip.05cm \par{\bf Periodic multiquadric kernel.} 
The one-dimensional multiquadric kernel is translation-invariant, radially-symmetric, and smooth, and reads 
\bse
\bel{MQ}
K_M(x,y) = \frac{1}{1+|x-y|^2} = \chi_M(|x-y|), \qquad x,y \in \RR, 
\ee
with $\hchi_M(\xi)= \exp(-|\xi|)$, so this kernel is nothing but the Fourier transform of the exponential kernel. above.
 We define its periodic version by setting, for all $x,y \in [0,1]^D$, 
\be
K_M^\per(x,y) 
:= \sum_{\alpha \in \ZZ^D} \prod_{d=1}^D e^{2 i \pi <x-y,\alpha>-\tau_D|\alpha|_1} = \prod_{d=1}^D \sum_{\alpha_d \in \ZZ} e^{2 i \pi (x_d-y_d)\alpha_d-\tau_D|\alpha_d|}, 
\ee
where $\tau_D$ is a parameter. Observe that the inverse Fourier transform of $\tau_D \exp(-\tau_D|x|)$ is
$\frac{1}{1+\alpha_d^2 / \tau_D^2}$. Thanks to the Poisson formula, the kernel $K_M^\per$ coincides with 
(for $x,y \in [0,1]^D$) 
\be
K_M^\per(x,y) 
= \prod_{d=1}^D\sum_{\alpha_d \in \ZZ} \frac{1}{1+(x_d-y_d + \alpha_d)^2 / \tau_D^2} = \prod_{d=1}^D \frac{ \sinh(2 \pi \tau_D)}{\cosh(2 \pi \tau_D)-\cos(2 \pi (x_d-y_d))}.
\ee
The diagonal contribution enjoys the following asymptotics 
$$
K_M^\per(x,x) = \prod_{d=1}^D\sum_{\alpha_d \in \ZZ} \frac{1}{1+\alpha_d^2 / \tau_D^2} = \Big( \frac{ \sinh(2 \pi \tau_D)}{\cosh(2 \pi \tau_D)-1} \Big)^D = \coth(\pi \tau_D/2)^D.
$$
Hence we choose $\tau_D = \frac{2\coth^{-1}(1+1/D)}{\pi}=\frac{2}{\pi}(\ln(2+1/D) - \ln(1/(D+1)))$
to ensure that $K_M^\per(x,x) \le e$ uniformly for any dimension. Finally, we write $K_M^\per(x,y) = \chi^\per_M(x-y) = \prod_d \chi_M^\per(x_d - y_d)$ (with some abuse of notation), and we plot the function $\chi_M^\per$ in Figure~\ref{MQ:figs} for several dimensions.
\ese

%------------------------------------------------------------------------------------------------------------

\begin{figure}

\centering{\includegraphics[width=0.3\linewidth]{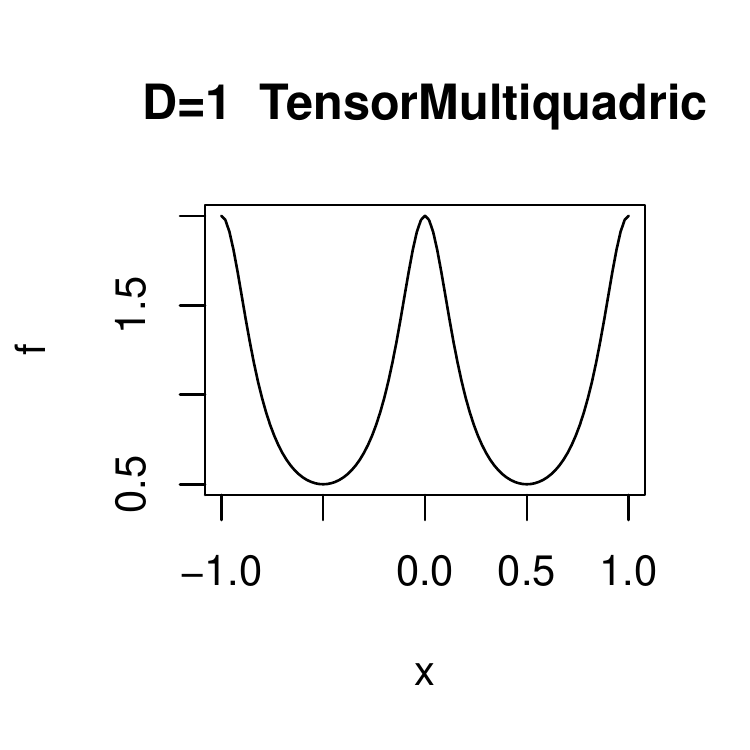} \includegraphics[width=0.3\linewidth]{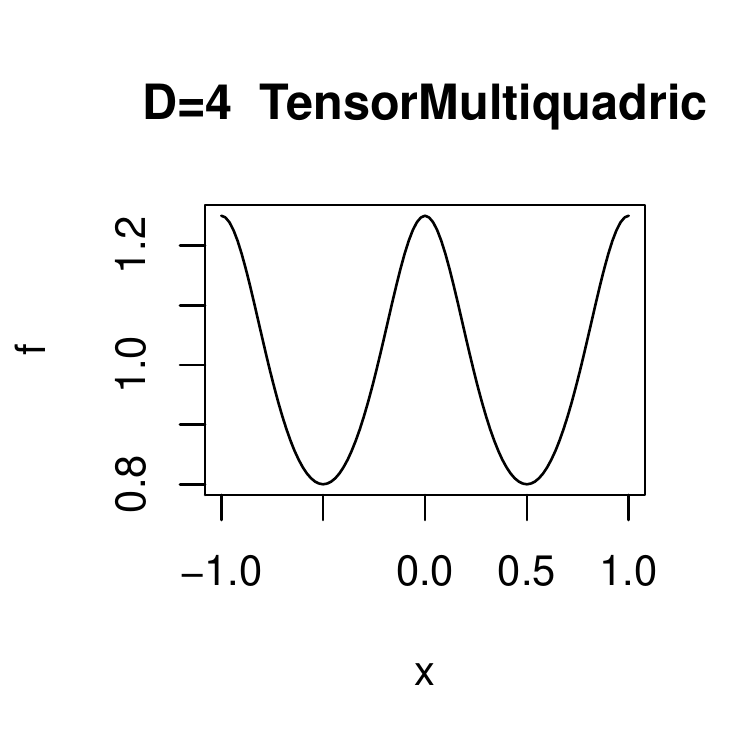} \includegraphics[width=0.3\linewidth]{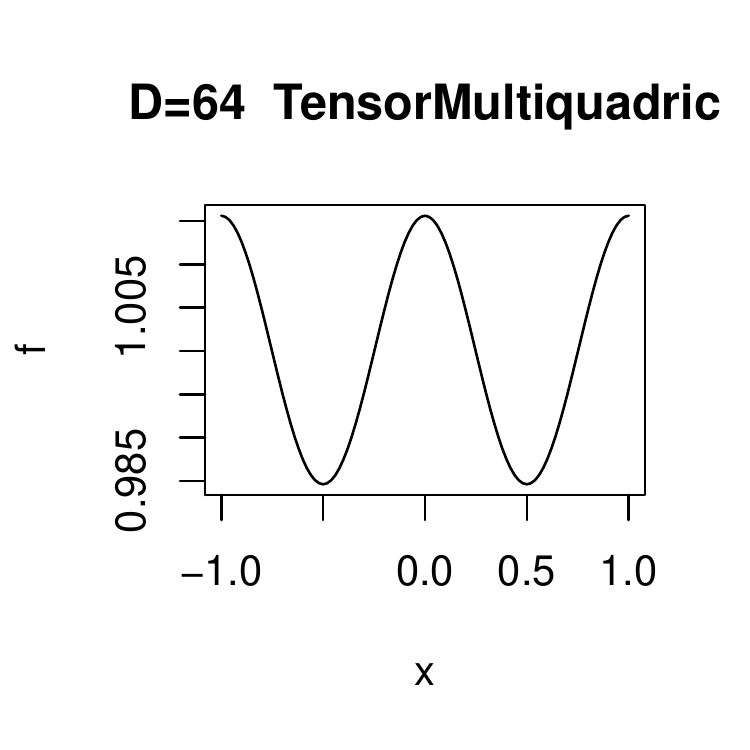} 

}

\caption{\label{MQ:figs} Periodic multiquadric kernel in dimension $D=1,4,64$.}
\label{fig:KERND-MQ}
\end{figure}
%------------------------------------------------------------------------------------------------------------------------------------------------

\vskip.05cm \par{\bf Periodic truncated kernel.} 
\bse
The truncated kernel is translation-invariant and Lipschitz continuous only, and reads 
\be
K_T(x,y)=  \sup(1- |x-y|,0) = \chi(x-y), \quad \hvarphi(\xi) =  \frac{ \sin^2(\pi \xi / 2)}{(\pi\xi/2) ^2}. 
\ee
We emphasize that its tensor product 
\be
\prod_{d=1}^D \sup(1- 2|x_d-y_d|,0)= \prod_{d=1}^D \chi(x_d-y_d), 
\qquad 
\hvarphi(\xi) =  \prod_{d=1}^D \frac{ \sin^2(\pi \xi_d / 2)}{(\pi\xi_d/2) ^2}
\ee
typically arises when designing a finite difference scheme on a Cartesian grid (say of the type $z_d^n = n/N$). 
Using again the Poisson formula \eqref{PF} we define its periodic version as 
\be
K_T^\per(x,y)
:= \sum_{\alpha \in \ZZ^D} \prod_{d=1}^D \tau_D\sup(1- \tau_D|x_d-y_d + \alpha_d|,0) 
=  \sum_{\alpha^* \in \ZZ^D} \prod_{d=1}^D \tau_D\frac{ \sin^2(\pi \alpha_d / \tau_D)}{ (\pi \alpha^*_d / \tau_D)^2}e^{2i \pi (x_d-y_d)\alpha^*_d}, 
\ee
where $1 \le \tau_D \le 2$ is a parameter, chosen so that the sum is finite, thus 
\be
K_T^\per(x,y)= \prod_{d=1}^D \sum_{\alpha_d \in \{-1,0,1\}} \tau_D\sup(1- \tau_D|x_d-y_d + \alpha_d|,0).  
\ee
The diagonal term is $K_T^\per(x,x) = \tau_D^D$, and we thus choose $\tau_D = 1+ 1/D$ in order to ensure that $K_T^\per(x,x) \le e$.
Writing $K_T^\per(x,y) = \chi_T^\per(x-y) = \prod_d \chi_T^\per(x_d - y_d)$, we plot $\chi_T^\per$ in Figure~\ref{fig:KERND-TRUNCATED} for several dimensions. 
\ese

%---------------------------------------------------------------------------------------------------------------------------------

\begin{figure}

\centering{\includegraphics[width=0.3\linewidth]{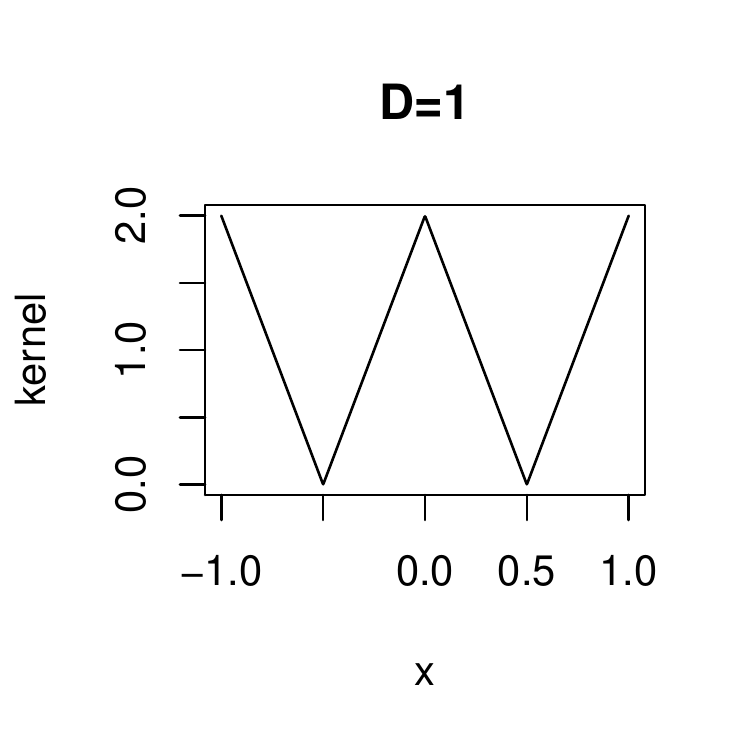} \includegraphics[width=0.3\linewidth]{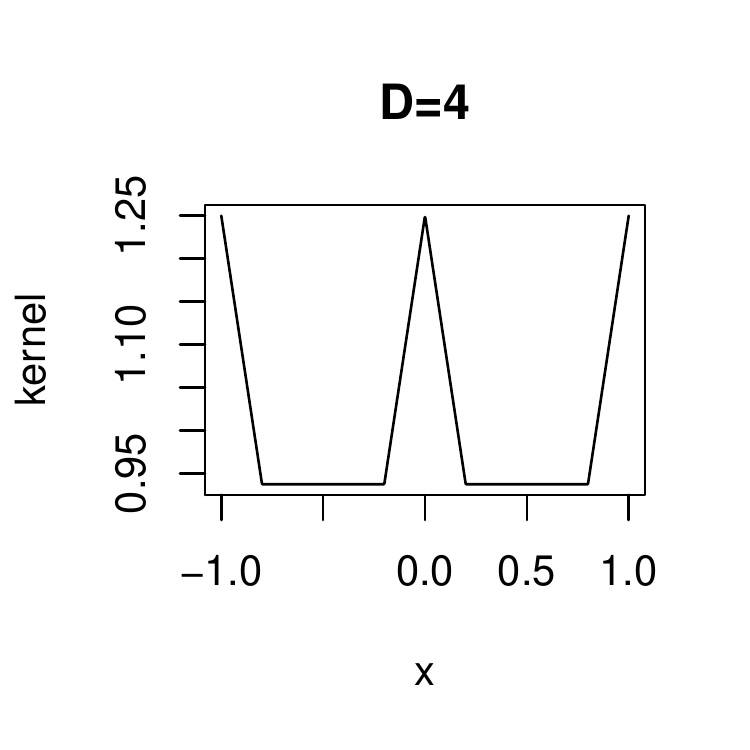} \includegraphics[width=0.3\linewidth]{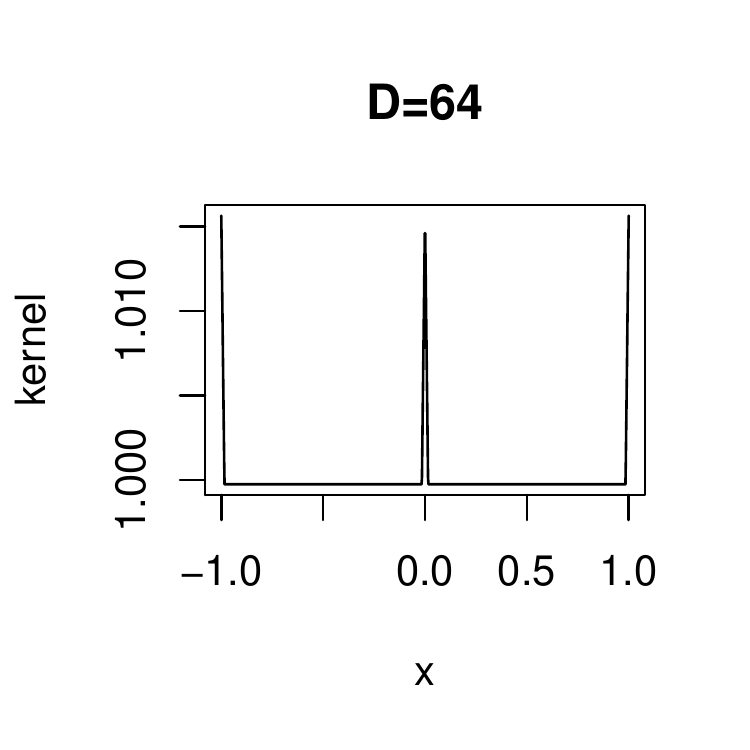} 
}
% \label{TR:figs} 
\caption{Periodic truncated kernels in dimension $D=1,4,64$}
\label{fig:KERND-TRUNCATED}
\end{figure}
%---------------------------------------------------------------------------------------------------------------------------------

%%---------------------------------------------------------------------------------------------------------------------------------
%\begin{table}
%\caption{Four compact kernels on $[-1,1]^D$}
%\centering
%\begin{tabular}{|l||c||c||c||c|}
%  \hline
%    & exponential & multiquadric & Gaussian & truncated \\
%  \hline
%   $\varphi(x)$ & $\exp(-|\erf^{-1}(x)|)$ & $\frac{1}{(1+|\erf^{-1}(x)|^2)^{\frac{D+1}{2}}}$ & $\exp(-\frac{|\erf^{-1}(x)|^2}{2})$ & $\sup(1-|\erf^{-1}(x)|,0)^D$  \\
%   \hline
%\end{tabular}
%\label{FDK3}
%\end{table}

%--------------------------------------------------------------------------------------------------------------------------------------------------------
\begin{figure}

\centering{\includegraphics[width=0.45\linewidth]{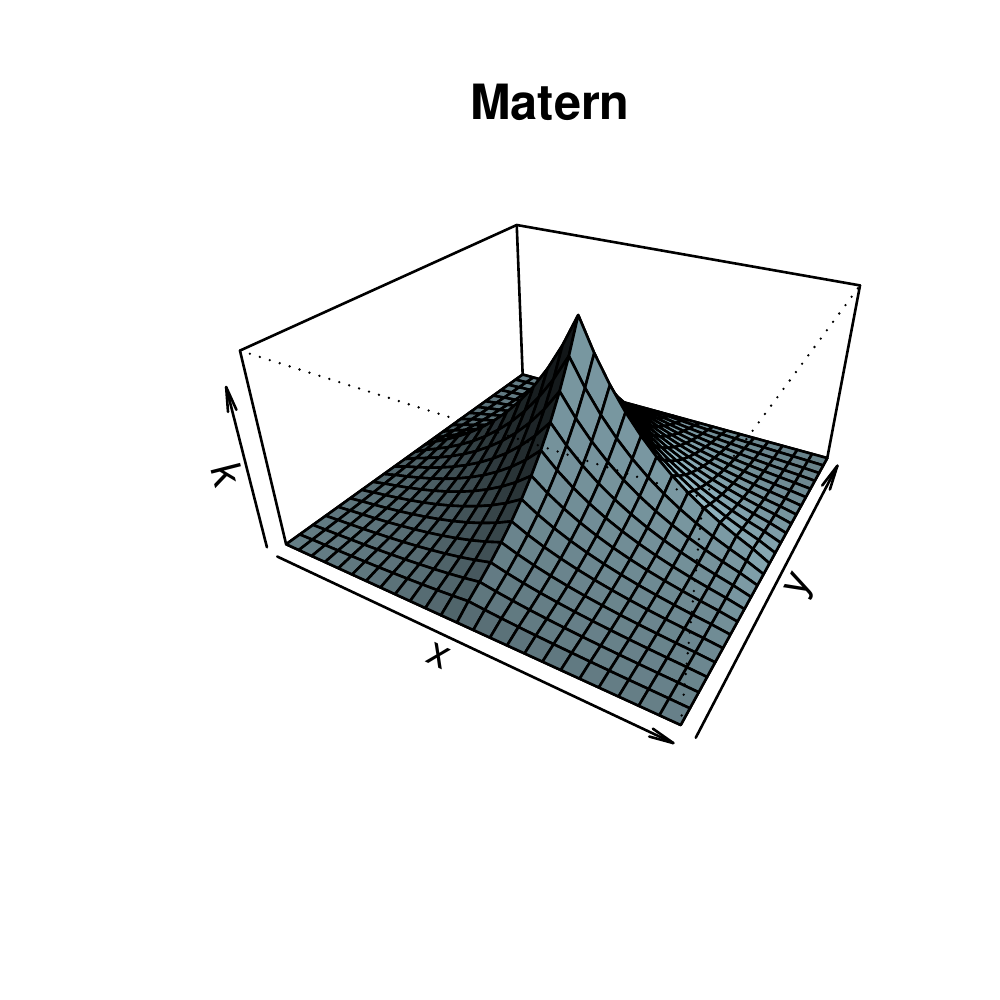} \includegraphics[width=0.45\linewidth]{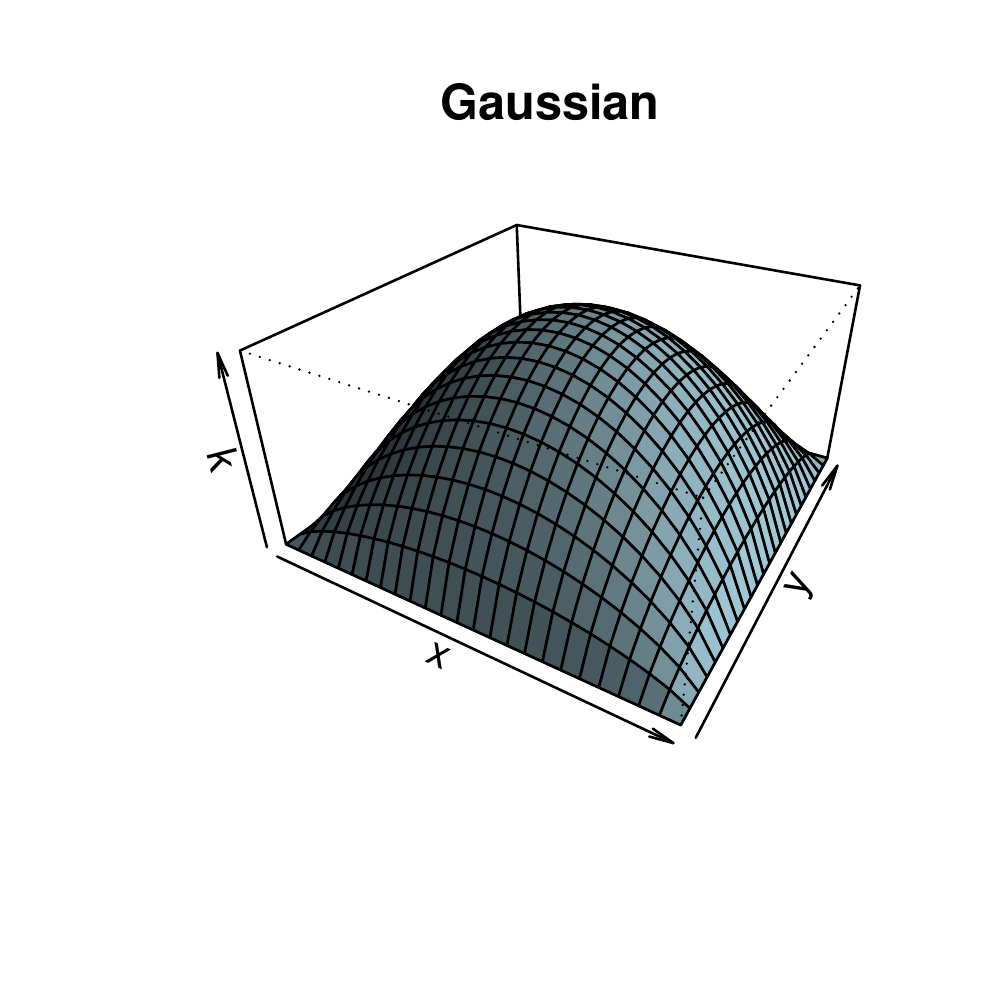} \includegraphics[width=0.45\linewidth]{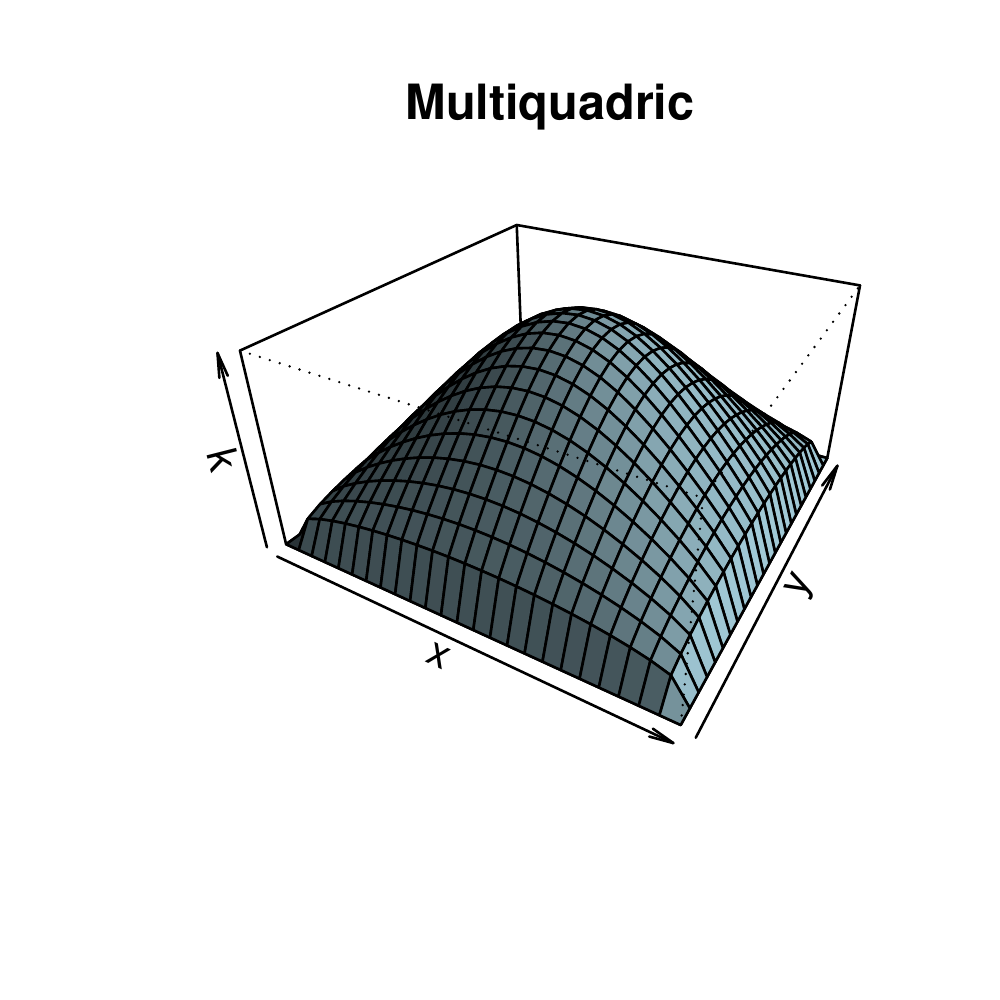} \includegraphics[width=0.45\linewidth]{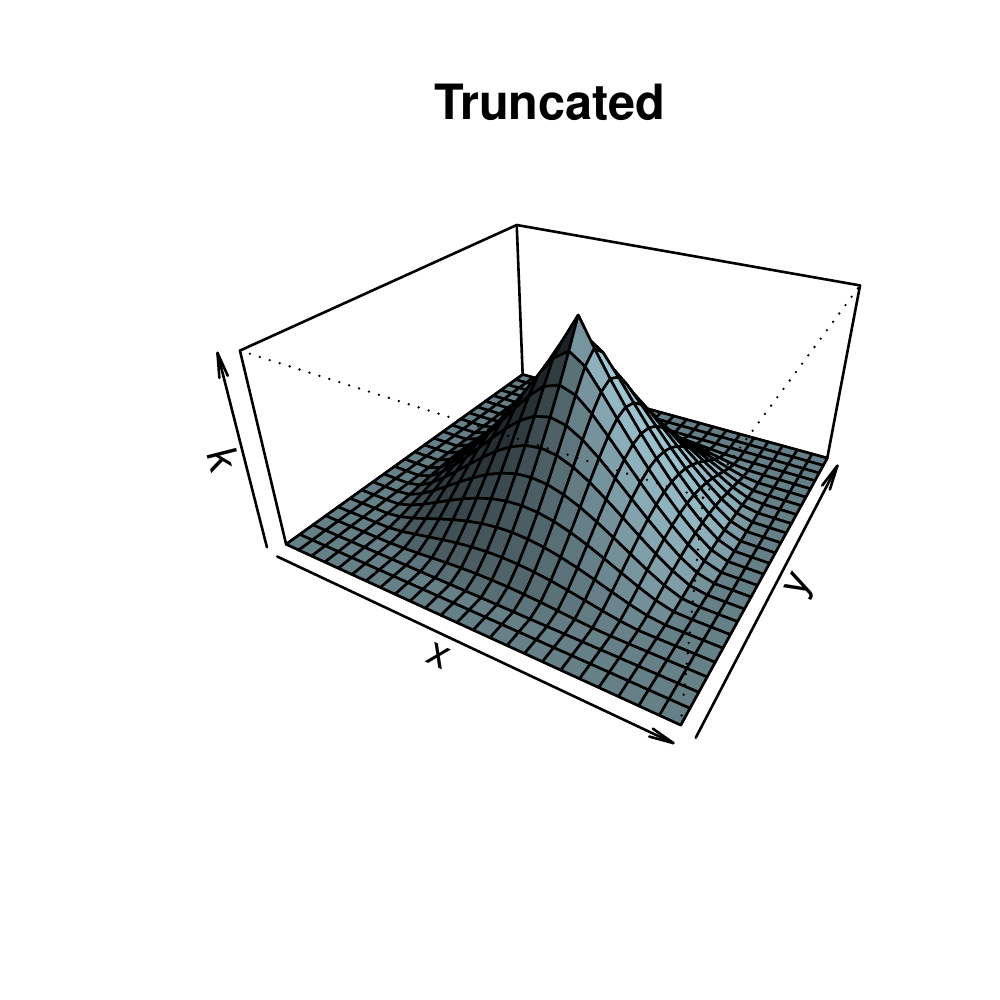} 
}
\caption{\label{kerns} Transported kernels in two dimensions: exponential, multiquadric, Gaussian, and truncated, respectively.}\label{fig:TIK}
\end{figure}

%-------------------------------------------------------------------------------------------------------------------------------------

\subsection{Transported kernels of interest}
\label{ETK}

\vskip.05cm \par{\bf Objective.}
An admissible kernel $K$ defined on an open set $\Omega$ is said to be a {\it compactly supported} if it extends to a continuous function on the closure $\Omegab$ and this extension vanishes on the boundary $\del\Omega$. For our numerical experiments, we design four compactly supported kernels defined on the unit cube $[0,1]^D$,  determined from the four examples in Table~\ref{FDK}.  Figure~\ref{fig:TIK} displays $K(x,1/2)$ for the two-dimensional case.

%-------------------------------------------------------------

\vskip.05cm \par{\bf Transported tensorial exponential kernel.}
\bse
In view of the standard expression of the exponential kernel, we introduce the following transported version defined on the unit cube $[0,1]^D$
\be
K_E^\trans(x,y) 
= {1 \over \beta_D} \exp\Big(-\tau_D \Big| \erf^{-1}(2x-1)-\erf^{-1}(2y-1) \Big|_1\Big), 
\ee
where $\erf^{-1}(x) = \Big( \erf^{-1}(x_1), \ldots, \erf^{-1}(x_D)\Big)$ and $|x|_1=\sum_d |x_d|$ and we 
the $\erf$ function reads $\erf(x) = (2 / \sqrt{\pi}) \int_0^x e^{-y^2} dy$. By Proposition \ref{TRK}, this kernel corresponds to the transport of the localized kernel
\be
K_E^\loc(x,y) = {1 \over \beta_D} \exp\big(-\tau_D  |x-y|_1\big) \exp\big(-|x|^2 - |y|^2\big), 
\ee
obtained from the map $S(x) = \erf^{-1}(2x-1)$. 
Here, we choose 
$$
\tau_D = \frac{\sqrt{\pi}}{D}, \qquad \beta_D = \Big( e^{\tau_D^2/4}(1 - \erf(\tau_D / 2) \Big)^D
$$
with, independently of the dimension $D$,  
$$
K_E^\trans(x,x) = \frac{1}{\beta_D} \simeq(1+\frac{\tau_D}{\sqrt{\pi}})^D \le e, 
\qquad 
\iint_{[0,1]^D \times [0,1]^D} K_E^\trans(x,y) \, dxdy = 1.
$$
\ese

%-------------------------

\vskip.05cm \par{\bf Transported multiquadric kernel.}
In view of the expression of multiquadric kernel, we consider the following localized version on the unit cube $[0,1]^D$
\bse
\be
K_M^\trans(x,y) = \frac{\beta_D}{(1+\tau_D^2 |\erf^{-1}(2x-1)-\erf^{-1}(2y-1))|^2)^{(D+1)/2}}, 
\ee
corresponding to the following transported kernel
\be
K_M^\loc(x,y) = \frac{\beta_D \exp(-|x|^2 - |y|^2)}{(1+\tau_D^2|x-y|^2)^{D}}, \quad x,y \in \RD, 
\ee
where we choose 
$$ 
 \tau_D=\sqrt{\frac{2}{D}}, 
\qquad	\beta_D =  \Big( \frac{ 1 }{ \sqrt{\pi} \exp(1 / \tau_D^2) (1 - \erf(1/\tau_D))(1/\tau_D)} \Big)^D \simeq \Big( 1 + \frac{\tau_D^2}{2} \Big)^D,
$$
as $\tau_D \to 0$, determined so that
$$
K_M^\trans(x,x) = \beta_D \le e, \quad \int_{[0,1]^D}\int_{[0,1]^D} K^\trans(x,y)dxdx \simeq\frac{\beta_D }{\Big( \sqrt{\pi} \exp(\tau_D^2 / 2) (1 - \erf(1/\tau_D) ( 1 / \tau_D)\Big)^D} = 1. 
$$

\ese
%----------------------------------------------------------

\vskip.05cm \par{\bf Transported Gaussian kernel.}
In view the expression of the Gaussian kernel, we introduce the following localized version on the unit cube $[0,1]^D$
\bse
\be
K_G(x,y) = \phi(x-y) 
= \beta_D \exp\Big(- \tau_D^2 \sum_{d=1}^D \big( \erf^{-1}(2x_d-1)-\erf^{-1}(2y_d-1)\big)^2\Big), 
\qquad x,y \in [0,1]^D, 
\ee
where we choose 
$
\tau_D = \sqrt{\frac{2}{D}}$
and $\beta_D = (1 + \tau_D^2)^{D/2}$, 
determined so that 
$$
K_G^\trans(x,x) = \beta_D \le e, 
\qquad 
\int_{[0,1]^D}\int_{[0,1]^D} K_G^\trans(x,y)dxdx \simeq\frac{\beta_D}{(1 + \tau_D^2)^{D/2}} = 1. 
$$ 
\ese

%------------------

\vskip.05cm \par{\bf Transported truncated kernel.}
In view of the expression of the truncated kernel, we consider the following localized version on the unit cube $[0,1]^D$
\bse
\be
K_T^\trans(x,y) = \frac{\beta_D}{(1+\tau_D|\erf^{-1}(2x_d -1)-\erf^{-1}(2y_d-1))|^2)^{D}}, 
\qquad x,y \in [0,1]^D, 
\ee
corresponding to the following transported kernel
\be
K_T^\loc(x,y) = \beta_D  \frac{\exp(-|x|^2 - |y|^2)}{(1+\tau_D|x-y|)^{D}}, \qquad x,y \in \RD. 
\ee
\ese 

%------------------------------------------

\subsection{Further constructions}

\vskip.05cm \par{\bf A compactly supported, tensor product kernel.}
\bse
\label{KERNEL--ODK}
The example given now is not translation-invariant and is supported on the unit interval $\Omega = [0,1]$: 
\bel{ODK}
\Kseed(x,y):= y(1-x)1_{y\leq x} + x(1-y)1_{x\leq y}, \qquad x, y, \in \RR. 
\ee
Observe that $\Kseed(0,y)=\Kseed(1,y)=0$, and that the Fourier technique above does not apply.
For future reference we compute 
\bel{K1DC} 
\begin{aligned}
& \del_y \Kseed(x,y)  =  (1 - x)1_{y <  x} - x 1_{y >  x}, 
\qquad
&& \del^2_y \Kseed(x,y)= - \delta_x
\\
& \int_{[0,1]} \Kseed(x,y) dy  =  \frac{1}{2}x(1-x),
 \qquad 
&& \iint_{[0,1]^2} \Kseed(x,y) dx dy =  \frac{1}{12}, 
\end{aligned}
\ee
together with the following integration by part formula: 
\be
\int_{[0,1]} \del_2 \Kseed(\cdot, x)\del_2 \Kseed(\cdot, y) dz 
= \int_{[0,1]} \del_2 ( \Kseed(\cdot, x) \del_2 \Kseed(\cdot, y) ) dz 
+ \Kseed(x,y) 
=  
\Kseed(x,y). 
\ee
Hence, the space $\Hcal_{\Kseed}(0,1)$ is the homogeneous Sobolev space 
$\dot H_0^1([0,1])$ of functions that vanish at the boundary, defined by density from the set of smooth functions 
supported on $(0,1)$ in the $\ell^2(0,1)$-norm. 
%$$
%\Hcal_K(0,1) = H_0^1(0,1)
%= \overline{\Big\{\varphi \in \Ccal_0^\infty([0,1]) \Big\}}^{\int_{[0,1]} |\del \varphi|^2dx <  + \infty}.
%$$ 
The eigenvalues and eigenfunctions of the spectral decomposition 
%\eqref{equa-567} 
is given by solving 
$\int_{[0,1]} \Kseed(\cdot,x) \zeta_i(x) dx = \lambda_i \zeta_i$ 
or, equivalently, $\del_x^2  \zeta_i = -\frac{1}{\lambda_i} \zeta_i$. We find 
$\zeta_i(x) = \sin(i \pi x ),\qquad \lambda_i = \frac{1}{( i \pi)^2}$
for $i = 1, 2, \ldots$ 
\ese

%--------------------------------------------------------------------------------------------  

\vskip.05cm \par{\bf A multi-dimensional version.}
\bse
Using the kernel $\Kseed$ in \eqref{ODK}, in general dimensions we consider $\Omega = [0,1]^D$ and the following kernel
\bel{KNDC}
K(x,y) =  \prod_{1 \leq d \leq D}  \Kseed(x_d,y_d),
= \prod_{1 \leq d \leq D}  \Big( y_d(1-x_d)1_{y_d\leq x_d} + x_d(1-y_d)1_{x_d\leq y_d} \Big). 
\ee
In view of 
\be
\int_\Omega (\del_1  \ldots \del_D)  K(\cdot, x) (\del_1  \ldots \del_D) K(\cdot, y) dz
= 
\prod_{1 \leq d \leq D}  \Big(\int_{[0,1]} \del_d  \Kseed(x_d,\cdot) \del_d \Kseed(y_d,\cdot) dz \Big) = K(x,y), 
\ee
the space generated by this kernel corresponds to the norm 
$
\int_{(0,1)^D} \Big(|\varphi(x)|^2 + |\del_1 \ldots \del_D \varphi(x)|^2 \Big) dx.
$ 
Similarly to the one-dimensional case, we compute  
\bel{KNDC-deux} 
\aligned
& \del_d K(x,y)  =  (\del \Kseed)(x_d,y_d)\prod_{e \neq d} \Kseed(x_e,y_e),
\qquad
&& \prod_{1 \leq d \leq D}  \del_d^2 K(x,y)= (-1)^D\delta_x
\\
& \int_{[0,1]^D} K(x,y) dy  =  \prod_{1 \leq d \leq D}  \frac{1}{2}x_d(1-x_d), 
\qquad 
&& \iint_{[0,1]^D \times [0,1]^D} K(x,y) dx dy =  \frac{1}{12^D}.
\endaligned
\ee 
The eigenvalues and eigenfunctions of the the spectral decomposition are now 
\bel{EGND}
\zeta_\alpha(x) = \prod_{1 \leq d \leq D}  \sin(\alpha_d \pi x_d ),\qquad \lambda_\alpha = \frac{1}{\prod_{1 \leq d \leq D} ( \alpha_d \pi)^2}, \qquad \alpha = (\alpha_1,\ldots,\alpha_D) \in \NN^D. 
\ee 
\ese

%=============================================================================

\section{Three formulation of the discrepancy function}
\label{section--4}

\subsection{Formulation in physical variables}
\label{EFPV}

We now factor out the integration error into (1) a factor depending upon the {\sl regularity} of the function $\varphi$ under consideration, measured in the $\Hcal_K$-norm, which is is independent of the choice of the interpolation points, and 
(2)  a factor depending solely upon the kernel $K$ and the mesh points, which is independent of the choice of the function. 
Three equivalent formulations of the second term are now derived in the physical, spectral, or discrete Fourier variables. Although these formulations are in principle equivalent, they shed a very different light on the problem of interest.  
We begin with an expression in the physical variables, which we will not use directly for the derivation of actual estimates, as it appears to be difficult to work with. 

%--------------------------------------------

\begin{proposition}[Factorization in physical variables]
\label{PEFPV}
Consider an admissible kernel $K= K(x,y)$ defined on an open set $\Omega \subset \RD$. Then, for any function $\varphi$ in the  Hilbert Space $\Hcal_K(\Omega)$, one has
\bel{EKPV}
\Big|\int_\Omega  \varphi(x) dx - {1 \over N}  \sum_{1 \leq n \leq N}  \varphi(y^n) \Big| 
\leq 
E_K(Y,N,D) \,  \|\varphi\|_{\Hcal_K(\Omega)}
\ee
for any set $Y= (y_1, \ldots,y_N)$ of points in $\Omega$, with the error function
\bse
\bel{equa:45} 
E_K(Y,N,D)^2: =  \frac{1}{N^2} \sum_{n,m=1}^N \iint_{\Omega \times \Omega} \Big(   K(x,y) +   K(y^n,y^m) -  K(y^n, y)  -  K(x,y^m) \Big) dxdy
\ee
or, equivalently,  
\be
E_K(Y,N,D) = \| e_K(Y) \|_{\Hcal_K(\Omega)}, 
\qquad 
e_K(Y):= \int_\Omega   \Big( K(\cdot,x)   - {1 \over N}  \sum_{1 \leq n \leq N} K(\cdot,y^n) \Big) dx. 
\ee
\ese
\end{proposition}

We observe that the integrand in \eqref{equa:45} is non-negative and can be also written in terms of the pseudo-distance $D$ (see \eqref{eq:defD}), that is, 
\be 
 E_K(Y,N,D)^2  =
 -\frac{1}{2N^2} \sum_{n,m=1}^N \iint_{\Omega \times \Omega}
\Big( D(x,y) +  D(y^n,y^m) -  D(y^n, y)  -  D(x,y^m) \Big) dxdy.
\ee

\begin{proof} In  view of \eqref{Npsi} and the identity $\varphi(x) = 
\la \varphi, K(\cdot,x) \ra_{\Hcal_K(\Omega)}$, we find 
$$
{1 \over N}  \sum_{1 \leq n \leq N}  \varphi(y^n) 
= {1 \over N}  \sum_{1 \leq n \leq N}  \big\la \varphi, K(\cdot,y^n) \big \ra_{\Hcal_K(\Omega)}
= 
\big\la \varphi, {1 \over N}  \sum_{1 \leq n \leq N}  K(\cdot,y^n) \big \ra_{\Hcal_K(\Omega)}.
$$
In agreement with the statement of the proposition, we write 
$$ 
\aligned
& \int_\Omega \varphi(x) dx  - {1 \over N}  \sum_{1 \leq n \leq N}  \varphi(y^n)
 =  \big\la \varphi, e_K(Y) \big \ra_{\Hcal_K(\Omega)},
\qquad
\quad
e_K(Y)  = \int_\Omega   \Big( K(\cdot,x)   - {1 \over N}  \sum_{1 \leq n \leq N} K(\cdot,y^n) \Big) dx, 
\endaligned
$$
and, by the Cauchy-Schwarz inequality, 
$ |E_Y(\varphi)| 
\leq  
\| \varphi \|_{\Hcal_K(\Omega)}
\|e \|_{\Hcal_K(\Omega)}. 
$
Hence, the integration error splits into two contributions, as expected. By expanding the integrand using \eqref{Npsi} we find the equivalent expression 
$$
\| e_K(Y) \|_{\Hcal_K(\Omega)}^2 = \Big\| \int_\Omega      K(\cdot,x) dx \Big\|^2_{\Hcal_K(\Omega)} + \frac{1}{N^2} \sum_{n,m=1}^N  K(y^n,y^m) -  \frac{2}{N} \sum_{1 \leq n \leq N} \int_\Omega  K(x,y^n) dx 
$$
and, using \eqref{Npsi},  
$\big\| \int_\Omega      K(\cdot,x) dx \big\|^2_{\Hcal_K(\Omega)} =  \int_{\Omega \times \Omega} K(x,y) dxdy$.
Moreover, in view of $K(x,y) = \frac{1}{2} \big( K(x,x) + K(y,y) - D(x,y) \big)$, the derivation of \eqref{equa:45} is completed. 
\end{proof}

%--------------------------------------------------------------------------------------------------------------------------------
 
\subsection{Formulation in spectral variables} 
\label{EFSA}

\begin{proposition}[Minimizing the error function in spectral variables]
\label{MEFSA} 
Consider an admissible kernel $ K: \Omega \times \Omega \to \RR$ together with 
its Mercer representation $(\lambda_i, \zeta_i )_{i \geq 1}$  
%(cf.~\eqref{equa-567}). 
Then the error function   
\bel{458-deux}
E_K^{s,p} (D,N) = \inf_{y^1,\ldots,y^N  \in \Omega} \Big( \sum_{i \geq 1} \lambda_i^{sp/2} \Big( \int_\Omega   
 \zeta_i (x)  dx - {1 \over N}  \sum_{1 \leq n \leq N} \zeta_i (y^n)  \Big)^p   \Big)
\ee
(for $s \geq 0$ and $p \in [1, =\infty)$) satisfies 
\bse
\be
E_K^{s,p}  
 \leq \Big( \sum_{i > N} \lambda_i^{sp/2} \Big( \int_\Omega    \zeta_i (x)  dx - {1 \over N}  \sum_{1 \leq n \leq N} \zeta_i (y^n)  \Big)^p   \Big), 
\ee
where $y^1,\ldots,y^N$ solve the system of equations
\bel{CONDSA}
 \int_\Omega    \zeta_i (x)  dx - {1 \over N}  \sum_{1 \leq n \leq N} \zeta_i (y^n)  = 0. 
 \ee
\ese 
\end{proposition} 

\begin{proof}
The expression $K_N(x,x') = \sum_{1 \leq n \leq N} \lambda_n \zeta_n(x)\zeta_n(x')$ defines an admissible kernel, for which we have the following orthogonal decomposition: 
$\Hcal_K(\Omega) = \Hcal_{K_N}(\Omega) \oplus \Hcal_{K_N}^{\perp}(\Omega)$. Consider the minimization problem
$$
E_{K_N}  
 = \inf_{y^1,\ldots,y^N  \in \Omega} \Big( \sum_{1 \leq i \leq N} \lambda_i \Big( \int_\Omega    \zeta_i (x)  dx - {1 \over N}  \sum_{1 \leq n \leq N} \zeta_i (y^n)  \Big)^2   \Big)
$$
and 
$$
\sum_{1 \leq i \leq N} \lambda_i \Big(  \int_\Omega    \zeta_i (x)  dx - {1 \over N}  \sum_{1 \leq n \leq N} \zeta_i (y^n) \Big) (\nabla \zeta_i) (y^n) = 0, \quad n=1,\ldots, N. \qedhere
$$
\end{proof}
 
%---------------------------

Whenever the eigenvalues and functions are known explicitly, the spectral decomposition associated with the kernel can be used in combination with the following formula for the error function. Observe that we can now deal with the spaces $\Hcal_K^{s,p}(\Omega)$ with general integrability and differentiability exponents, and recall that that $E_K(Y,N,D) = E_K^{1,2}(Y)$.

\begin{proposition}[Factorization in spectral variables] 
\label{PEFSA}

Consider an admissible kernel $K= K(x,y)$ defined on an open set $\Omega \subset \RD$. Then, for any function $\varphi$ in the corresponding Banach space $\Hcal_K^{s,p}(\Omega)$ associated with exponents $s>0$ and $p \in [1, +\infty)$. Then, for any set $Y= (y^1, \ldots,y^N)$ of points in $\Omega$ and any function $\varphi \in \Hcal_K^{s,p}(\Omega)$, one has
\bel{equa:inequ}
\Big|\int_\Omega  \varphi(x) dx - {1 \over N}  \sum_{1 \leq n \leq N}  \varphi(y^n) \Big| 
\leq 
E_K^{s,p}(Y) \,  \|\varphi\|_{\Hcal_K^{s,p}(\Omega)},
\ee
where the error function (with $1/p+1/p'= 1$) 
\bel{457}
E_K^{s,p'}(Y)  
: = \Big( \sum_{i \geq 1} \lambda_i^{sp'/2} \Big( \int_\Omega    \zeta_i (x)  dx - {1 \over N}  \sum_{1 \leq n \leq N} \zeta_i (y^n)  \Big)^{p'}   \Big)^{1/p'}
= \| e_K(Y) \|_{\Hcal_K^{s,p'}(\Omega)}
\ee
is expressed in terms of the eigenfunctions and eigenvectors $\zeta_i, \lambda_i$ of the Mercer representation of the kernel $K$.   
\end{proposition}

\begin{proof} We consider the Mercer representation 
%\eqref{equa-567} 
and use the relations $T_K \zeta_i = \lambda_i \zeta_i$ and with $e_K = \sum_i \la e_K, \zeta_i \ra_{\ell^2(\Omega)} \zeta_i$. 
We have 
$$
\aligned
\la e_K, \zeta_i \ra_{\ell^2(\Omega)} 
& = \int_\Omega  \int_\Omega   \Big( K(y,x) dx  - {1 \over N}  \sum_{1 \leq n \leq N} K(y,y^n)  dx\Big) \zeta_i (y) dy 
  = \lambda_i \, \Big( \int_\Omega    \zeta_i (x)  dx - {1 \over N}  \sum_{1 \leq n \leq N} \zeta_i (y^n)  \Big)
\endaligned
$$
and, by H\"older inequality, we obtain  
$$
\aligned
\big\la \varphi, e_K \big \ra_{\Hcal_K(\Omega)} 
& = \sum_{i=1,2, \ldots}  \lambda_i^{-1} \la \varphi,\zeta_i \ra_{\ell^2} \lambda_i   \Big( \int_\Omega   \zeta_i (x)  dx - {1 \over N}  \sum_{1 \leq n \leq N} \zeta_i (y^n)  \Big) 
\\
 & =  \Big\la \Big( \lambda_i^{-s/2} \la \varphi,\zeta_i \ra_{\ell^2} \Big)_{i \geq 1},  
\lambda_i^{s/2} \Big( \int_\Omega    \zeta_i (x)  dx - {1 \over N}  \sum_{1 \leq n \leq N} \zeta_i (y^n)  \Big)_{i \ge0}  
\Big \ra_{\ell^p(\NN), \ell^{p'}(\NN)} 
\\
& \leq \|\varphi\|_{\Hcal_K^{s,p}(\Omega)} \Big( \sum_i \lambda_i^{sp'/2} \Big( \int_\Omega    \zeta_i (x)  dx - {1 \over N}  \sum_{1 \leq n \leq N} \zeta_i (y^n)  \Big)^{p'}  \Big)^{1/p'}.  \hfill \qedhere
\endaligned 
$$ 
\end{proof} 

%-------------------------------------------------------------------------------------------------------------------------

\subsection{Formulation in discrete Fourier variables}
\label{EFFV}

Assuming now more structure on the kernel and, specifically, assuming that it is based on a discrete lattice, 
we express the error function in a third form based on the Poisson formula. 
 
\begin{proposition}[Factorization in discrete Fourier variables] 
\label{PEFFV} 
Consider a kernel based on a discrete lattice $\Lbf$  with elementary cell $\Cbf$, say 
$K^\per(x,y):=  {1 \over |\Cbf|} \big\la  e^{2i \pi <x-y, \cdot>}, \rho \big\ra_{\ell^2(\Lbf^*)}$, determined from a generating function $\rho$ defined on the dual lattice $\Lbf^*$. 
Then, for any set $Y= (y^1, \ldots,y^N)$ of points in $\Cbf$ and any $\Lbf$--periodic function 
$\varphi \in \Hcal_{K^\per}^{s,p}(\Omega)$ (for some $s>0$ and $p \in [1, +\infty)$ with $1/p+1/p'= 1$), 
one has \eqref{equa:inequ} with 
\bel{458}
\Big(E_{K^\per}^{s,p}(Y)\Big)^p  
: = \sum_{\alpha^* \in \Lbf^*} \rho^s \Big| \int_\Cbf e^{2i \pi <x,\alpha^*>} \, dx 
      - {1 \over N} \sum_n e^{2i \pi <y^n,\alpha^*>} \Big|^p. 
\ee 
\end{proposition}

\begin{proof} Provided $ \supp (\varphi)\subset \Cbf$, we have 
$$
\aligned
 	\int_\Cbf \varphi(x) \, dx - {1 \over N} \sum_n \varphi(y^n) 
&=  \big\la  \int_\Cbf e^{2i \pi <x,\cdot>} \, dx - {1 \over N} \sum_n e^{2i \pi <y^n,\cdot>}, \varphih  \big \ra_{\ell^2(\Lbf^*)}
:=  \big\la  \widehat{e}, \varphih  \big \ra_{\ell^2(\Lbf^*)}
\\
& = \big\la  \frac{\varphih}{\rho^{s/p}}, \rho^{s/p}\widehat{e}  \big \ra_{\ell^2(\Lbf^*)}, 
\endaligned
$$
where the error function $\widehat{e}$ is defined as the Fourier transform of $e:= 1_\Cbf   - {1 \over N}  \sum_{1 \leq n \leq N} \delta_{y^n}$, hence 
$$
\Big| \int_\Cbf \varphi(x) \, dx - {1 \over N} \sum_n \varphi(y^n) \Big| 
\leq  \Big\| 
 \rho^{s/p} 
\big( \int_\Cbf e^{2i \pi <x,\cdot>} \, dx - {1 \over N} \sum_n e^{2i \pi <y^n,\cdot>} 
\big) 
\Big\|_{\ell^p (\Lbf^*)} \| \varphi \|_{\Hcal_{K^\per}^{s,p}(\Cbf)}. 
\hfill \qedhere
$$ 
\end{proof}

%===========================================================================

\section{Controlling the discrepancy function}
\label{sec-opti}

\subsection{Physical variables}

%  minimizers and optimal estimate in one dimension}

We seek for a set of points minimizing the error function and, first in the physical variables, we obtain the following results. 

\begin{proposition}[Minimizing the error function in physical variables]
\label{MEFPV}
Let $K: \Omega \times \Omega \to \RR$ be an admissible kernel defined on some open set $\Omega$, and consider the error function \eqref{EKPV}.
Provided the kernel is convex with respect to each variable, 
the functional $Y = (y^1, \ldots,y^N) \mapsto (E_K(Y,N,D))^2$ is a positive (non-strictly) convex with
positive infimum
$$
E_K(N,D):= \inf_{Y=(y^1, \ldots,y^N)} E_K(Y,N,D) > 0.
$$
Moreover, if $Y$  be a minimizer, then the gradient of the functional vanishes at $Y$, namely 
\bel{291}
\int_\Omega \nabla  K(x,y^n) dx =  {1 \over N}  \sum_{m= 1}^N  (\nabla K)(y^n,y^m), 
\qquad n=1, 2, \ldots, N. 
\ee
Furthermore, for all functions $\Hcal_K^Y(\Omega)$ (as defined in \eqref{equa:HKY}), 
the following integration formula holds (without error term): 
\bel{529}
\int_\Omega  \varphi(x) dx = {1 \over N}  \sum_{1 \leq n \leq N}  \varphi(y^n),
\qquad   \varphi \in \Hcal_K^Y(\Omega). 
\ee
\end{proposition}

Here, $\nabla K(x,y)$ stands here for any of $\nabla_x K(x,y)$ or $\nabla_y K(x,y)$, since the kernel $K$ is symmetric in its two arguments. Before proceeding with the proof, the following comments are in order: 
\bei

\item The functional is be {\sl totally symmetric} with respect to its arguments so, clearly, a minimizer is never unique. Yet, suppose that our kernel $K(x,y)$ is concave, then the following semi-discrete algorithm
\bel{292}
\frac{d}{dt} y^n =  \int_\Omega \nabla  K(x,y^n) dx - {1 \over N}  \sum_{m, n = 1}^N  (\nabla K)(y^n,y^m), 
\qquad n=1, 2, \ldots, N
\ee
converges toward a minimum of the functional $E_K$. Observe also that the dynamical system under consideration involves two terms:  $\nabla K(y^n,y^m)$ tends to push points away from each other, while the  integral term tends to attract the points toward the mass-center of $\Omega$. These two competitive effects leads to a non-trivial distribution of the points. 

\item With  the partition of unity defined in \eqref{410}, any minimizer must satisfy 
$
\int_\Omega  \delta^n(x) dx = {1 \over N}, 
$
for 
$
n=1, 2, \ldots, N. 
$
\eei

\begin{proof}
Let us set $e:= {1 \over N}  \sum_{1 \leq n \leq N} \int_\Omega   \Big( K(\cdot,x) dx  - K(\cdot,y^n) \Big) dx$. We have denoted by $e_Y:= \sum_n e_Y(y^n) \delta_{Y}^n $ its projection on the approximation space $\Hcal_K^Y(\Omega)$. It satisfies 
$$
E_K(Y,N,D)^2 = \|e\|_{\Hcal_K(\Omega)}^2 = \|e_Y \|_{\Hcal_K^Y(\Omega)}^2 + \|e - e_Y\|_{\Hcal_K}^2.
$$
Thus $E_K(Y,N,D)^2 \geq  \|e - e_Y\|_{\Hcal_K}^2 > 0$, and, as we are going to check, a minimizer makes the term 
$\|e_Y \|_{\Hcal_K^Y(\Omega)}^2$ to vanish.

In view of the definition \eqref{457} 
and the symmetry property $\del_{y^n}K(y^n,y^m) = \del_{y^n}K(y^m,y^n)$, we obtain 
($n \in[1, 2, \ldots, N]$, $d \in [1, \ldots, D]$) 
$$
\del_{y^n_d}(E_K(Y,N,D))^2 =  \frac{2}{N^2} \sum_{n}   \del_{y^n_d} K(y^n,y^k)  -  \frac{2}{N} \int_\Omega   \del_{y^n_d} K(x,y^n) dx, 
$$
establishing \eqref{291}. By computing second-order derivatives, we see that Hessian $\del_{y^n_d} \del_{y^n_e}  (E_K(Y,N,D))^2$ reads as follows ($d,e \in \big\{1, \ldots, D \big\}$): 
\bel{eq:912}
		 \frac{2}{N^2} \sum_{m, n} \del_{y^n_d} \del_{y^n_e} K(y^n,y^m) -  \frac{2}{N} \int_\Omega   \del_{y^n_d} \del_{y^n_e} K(x,y^n) dx.
\ee
Moreover, recalling \eqref{410}, we compute
$$
\aligned
 e_Y  & = \sum_n e_Y(y^n) \delta_{Y}^n  = \big\la \int_\Omega   K(Y,Y)(x) dx, \delta_{Y} \big \ra_{\RN}  - {1 \over N}  \sum_{n} K(\cdot, y^n) 
 = \big\la \int_\Omega   \delta_Y(x) dx, K_{Y} \big \ra_{\RN}  - {1 \over N}  \sum_{n} K(\cdot, y^n).
\endaligned
$$
Hence, we arrive at
$$
\aligned
\|e_Y \|_{\Hcal_K(\Omega)}^2 
& = \big\la  \sum_{n} \Big( \int_\Omega  \delta^n(x) dx - {1 \over N}  \Big)  \delta_{y^n},  \big\la \int_\Omega   \delta_Y(x) dx, K_{Y} \big \ra_{\RN}  - {1 \over N}  \sum_{n} K(\cdot, y^n)  \big \ra_{\Dcal',\Dcal} 
\\
& = \sum_{m,n} K(y^n,y^m) \Big( \int_\Omega  \delta^n(x) dx - {1 \over N}  \Big) \Big( \int_\Omega  \delta^m(x) dx - {1 \over N}  \Big) 
  \geq C \| \Big(\int_\Omega  \delta^n(x) dx - {1 \over N}  \Big)_n\|_{\RN}^2.
\endaligned
$$
In particular, the minimum vanishes when $ \int_\Omega  \delta^n(x) dx = {1 \over N} $, and this establishes \eqref{529}. Hence, suppose that this relation is true, then from  \eqref{410} we deduce
$$
\frac{1_N}{N}= \int_\Omega \delta_Y(x) \, dx = K(Y,Y)^{-1} \int_\Omega K(Y,Y)(x) \, dx,
$$
implying  
$K(Y,Y) \frac{1_N}{N} = \int_\Omega K(Y,Y)(x) \, dx$
or, equivalently, $K(Y,Y)^m$ is a stochastic matrix, which is \eqref{529}.
\end{proof} 

%-------------------------------

\vskip.05cm \par{\bf Best discrepancy sequences in one dimension.} 
We are now in a position to analyze an example in dimension $D=1$, where we have fully explicit expressions. 

\begin{proposition}[The case of dimension $D=1$] 
\label{TCD1}
For the kernel  \eqref{KERNEL--ODK}, the equally-spaced sequence 
\bel{eq:eqlaseq} 
y^n = \frac{2 n-1}{2N}, \qquad n=1, 2, \ldots, N
\ee
 is the unique solution (up to re-ordering) that minimizes the error function $Y \mapsto E_K(Y)$ introduced in \eqref{EKPV} and, with this sequence, the following estimate holds: 
$$
	\Big|\int_{[0,1]} \varphi(x) dx - {1 \over N}  \sum_{1 \leq n \leq N}  \varphi(y^n) \Big| 
\leq \frac{1}{\sqrt{6}N} \, \| \varphi\|_{\Hcal_K^{1,2}([0,1])}, 
\qquad 
 \varphi \in \Hcal_K^{1,2}([0,1]).
$$
\end{proposition}

\begin{proof}  We begin with the case $s=$ and $p=2$, and consider an arbitrary (ordered) sequence of points, say $0<y^1 < \ldots<y^N<1$. We rewrite our functional \eqref{equa:45} as
$$
\aligned
(E_K(Y))^2 
= & \int_{([0,1] )^2} K(x,y) dxdy - {1 \over N} \sum_{1 \leq n \leq N} (y^n - 1)y^n
  + \frac{1}{N^2} \sum_{n > m,n= 1, \ldots, N} y^m (y^n-1) 
\\
& \quad + \frac{1}{N^2} \sum_{n \leq m,n= 1, \ldots, N} ( y^m - 1) y^n 
\\
= & 1/12  - {1 \over N} \sum_{1 \leq n \leq N} (y^n - 1)y^n + \frac{1}{N^2} \Big( \sum_{1 \leq n \leq N}  y^n\Big)^2
  - \frac{1}{N^2}  \sum_{n, m = 1}^N  y^n \big( 1_{m>n} + 1_{n \leq m} \big)
\endaligned
$$
or, equivalently, 
$$
 (E_K(Y))^2 = 1/12  - {1 \over N} \sum_{1 \leq n \leq N} (y^n)^2 + \frac{1}{N^2} \Big( \sum_{1 \leq n \leq N}  y^n\Big)^2 +\frac{1}{N^2} \sum_{1 \leq n \leq N}  y^n (2n - N - 1). 
$$
In view of Proposition~\ref{MEFPV}, any minimizer satisfies the system of equations
($n=1, 2, \ldots, N$) 
$$
\del_{y^n}(E_K(Y))^2 = -  \frac{2}{N}y^n + \frac{2}{N^2} \sum_{1 \leq n \leq N}  y^n + \frac{1}{N^2} (2n - N - 1) = 0, 
$$
which we put in the form
$
{1 \over N}  \sum_{1 \leq n \leq N}  y^n - \frac{1}{2}+ {1 \over N}  (n - 1 / 2) = y^n.
$
This is a linear algebraic system, which is readily solved explicitly; this leads us to \eqref{eq:eqlaseq}. 
We check immediately that 
${1 \over N}  \sum_{1 \leq n \leq N}  y^n = {1 \over 2}$
and 
$
\sum_{1 \leq n \leq N}  (y^n)^2= {N \over 3} - \frac{1}{12 N}$,  
and, finally, we evaluate the functional to be 
$$
(E_K(Y))^2 
 = {1 \over 12} + \Big( {1 \over 3} - {1 \over 12 N^2} \Big)  - \frac{1}{4} -\frac{2}{N^2} \sum_{1 \leq n \leq N}  \frac{(n-1/2)^2}{N} + \frac{2}{N^2} \sum_{1 \leq n \leq N}  \frac{n-1/2}{2} = \frac{1}{6 N^2}. 
\qedhere
$$
\end{proof}

%-----------------------------------------------------------------------------------------------------------------------------

\subsection{Spectral variables} 

For clarity in the presentation, let us repeat here our previous observation. 

\begin{proposition}[Minimizing the error function in spectral variables]
\label{MEFSA-9} 
Consider an admissible kernel $K: \Omega \times \Omega \to \RR$ and denote by $(\lambda_i, \zeta_i )_{i \geq 1}$ its Mercer representation. 
%(cf.~\eqref{equa-567}). 
Then for the error function \eqref{457}
\bel{458-deux-0}
E_K^{s,p}  
: = \inf_{y^1,\ldots,y^N  \in \Omega} \Big( \sum_{i \geq 1} \lambda_i^{sp/2} \Big( \int_\Omega   
 \zeta_i (x)  dx - {1 \over N}  \sum_{1 \leq n \leq N} \zeta_i (y^n)  \Big)^p   \Big), 
\ee
one has
\bse
\bel{EKSP-9}
E_K^{s,p}  
 \leq \Big( \sum_{i > N} \lambda_i^{sp/2} \Big( \int_\Omega    \zeta_i (x)  dx - {1 \over N}  \sum_{1 \leq n \leq N} \zeta_i (y^n)  \Big)^p   \Big), 
\ee
where $(y^1,\ldots,y^N) \in \RR^{N\times D}$ is given by the following system of equations
\bel{CONDSA-9}
 \int_\Omega    \zeta_i (x)  dx - {1 \over N}  \sum_{1 \leq n \leq N} \zeta_i (y^n)  = 0. 
 \ee
\ese 
\end{proposition} 

\begin{proof}
The expression $K_N(x,y) = \sum_{1 \leq n \leq N} \lambda_n \zeta_n(x)\zeta_n(y)$ defines an admissible kernel, for which we have an orthogonal decomposition of the form 
$\Hcal_K(\Omega) = \Hcal_{K_N}(\Omega) \oplus \Hcal_{K_N}^{\perp}(\Omega)$. It suffices to consider the minimization problem
$$
E_{K_N}  
 = \inf_{y^1,\ldots,y^N  \in \Omega} \Big( \sum_{1 \leq i \leq N} \lambda_i \Big( \int_\Omega    \zeta_i (x)  dx - {1 \over N}  \sum_{1 \leq n \leq N} \zeta_i (y^n)  \Big)^2   \Big)
$$
and 
$$
\sum_{1 \leq i \leq N} \lambda_i \Big(  \int_\Omega    \zeta_i (x)  dx - {1 \over N}  \sum_{1 \leq n \leq N} \zeta_i (y^n) \Big) (\nabla \zeta_i) (y^n) = 0, \qquad n=1,\ldots, N. \qedhere
$$
\end{proof}

%==================================================================================

\vskip.05cm \par{\bf Best discrepancy sequences in one dimension.} 
We can now revisit the example in dimension $D=1$ treated in Proposition~\ref{TCD1}.

\begin{proposition}[The case of dimension $D=1$.]
For the kernel  \eqref{KERNEL--ODK}, and the equally-spaced sequence 
\bel{eq:eqlaseq-9} 
y^n = \frac{2 n-1}{2N}, \qquad n=1, 2, \ldots, N,
\ee
the error function $Y \mapsto E_K^{s,p}(Y)$ introduced in \eqref{457} can be bounded by, for $s > 1- 1/p$ and some constant $C_{s,p}>0$: 
$$
	\Big|\int_{[0,1]} \varphi(x) dx - {1 \over N}  \sum_{1 \leq n \leq N}  \varphi(y^n) \Big| 
\leq \frac{C_{s,p}}{N^s} \, \| \varphi\|_{\Hcal_K^{s,p}([0,1])}, 
\qquad 
 \varphi \in \Hcal_K^{s,p}([0,1]).
$$
\end{proposition}
 
\begin{proof} The general case follows  using the spectral formulation \eqref{457}, where $\zeta_i(x) = \sin(i \pi x )$ and $\lambda_i = \frac{1}{( i \pi)^2}$.  We introduce
$$
\aligned
c_i 
&:= \int_{[0,1]}   \zeta_i (x)  dx - {1 \over N}  \sum_{1 \leq n \leq N} \zeta_i (y^n) 
\\
& = \int_{[0,1]}   \sin(i \pi x )  dx   - {1 \over N}  \sum_{1 \leq n \leq N} \sin(i \pi \frac{2n-1}{2N} )  
  =  \frac{1 - \cos(i \pi)}{i \pi} - \frac{\sin^2(\frac{i\pi}{2})}{N \sin( \frac{i\pi}{2 N})}. 
\endaligned
$$
Hence, $c_{2i}= 0$, $c_{2i+1} = \frac{2}{(2i+1) \pi} - \frac{1}{N \sin( \frac{(2i+1)\pi}{2 N})}$ and, in view of \eqref{457}, 
$$
\aligned
E_K^{s,p}(Y)^{p'} 
& = \sum_{i\geq 0} \frac{1}{( (2i+1)\pi)^{sp'}} \Big( \frac{2}{(2i+1) \pi} - \frac{1}{N \sin( \frac{(2i+1)\pi}{2 N})} \Big)^{p'} 
   = \sum_{i\geq 0} \lambda_{2i+1}^{sp'/2} \Big(2 \lambda_{2i+1}^{1/2}  - \frac{1}{N \sin( \frac{1}{2 N \lambda_{2i+1}^{1/2}})} \Big)^{p'}
\endaligned
$$
where $1/p + 1/(p') = 1$. 
In particular, $E_K^{1,2}(Y) = \frac{1}{\sqrt{6} N}$, as computed in Proposition \ref{TCD1}.  For general $s,p$, we get the expansion $\Big(2 \lambda_{2i+1}^{1/2} - \frac{1}{N \sin( \frac{1}{2 N \lambda_{2i+1}^{1/2}})} \Big) \simeq \frac{\lambda_{2i+1}^{-1/2}}{N^2}$ as $\frac{\lambda_{2i+1}^{-1/2}}{2N} < 1$, that is $i \pi < N$, and is bounded elsewhere.) 

%----------------------------------------- 

The error term in \eqref{457} is bounded as follows: 
$$
E_K^{s,p}(Y)^{p'} \leq \sum_{\pi i \leq N} \frac{C}{N^{2p'}} \lambda_{2i-1}^{(s - 1)p'/2}  +C\sum_{i \pi\geq N} \lambda_{2i+1}^{sp'/2}, 
$$
where the first term in the right hand-side is uniformly bounded by $\frac{C}{N^{2p'}}$ provided $(s - 1)p'>1$, that is, $s> 1/p$. The second term is bounded by the first term $\frac{C}{N^{sp'}}$ provided $sp'>1$.
Thus we can control $E_K^{s,p}(Y)$ by $C/N^s $ in the range $sp'>1$ at least. 
Observe finally that we are led to the same estimate, since 
$$
(E_K^{s,p})^{p'} \leq \sum_{i > N} \lambda_i^{sp'/2} \Big( \int_{[0,1]^D}   \zeta_i (x)  dx - {1 \over N}  \sum_{1 \leq n \leq N} \zeta_i (y^n)  \Big)^{p'} 
\lesssim \frac{1}{N^{sp'}}. 
\qedhere 
$$
\end{proof}

%---------------------------------------------------------------------------------------------------------------------------------

\subsection{Discrete Fourier variables: toward a sharp estimate} 

Our third formulation provides us with the most practical setup for the applications, when the  spectral decomposition may not be explicitly available. 

\begin{proposition}
\label{PEFFV-99} 
Let $\Lbf$ a lattice with elementary cell $\Cbf$ normalized such that $|\Cbf|=1$, and consider $\Lbf^*$ its dual lattice. Let $\alpha^* \in \Lbf^* \mapsto \rho(\alpha^*)$ be any discrete function satisfying $\rho \in \ell^1(\Lbf^*)$, $\rho\ge 0$, and $\rho(0)=1$. Then the kernel 
\bel{PER}
\Kper(x,y) = \sum_{\alpha^* \in \Lbf^*} \rho(\alpha^*)e^{2 i \pi <x-y,\alpha^*>}
\ee
is an admissible kernel with period $\Cbf$ and in the Hilbert space $\Hcal_{\Kper}(\Cbf)$ one has 
\bel{PERROR}
\Big| \int_\Omega \varphi(x)dx - \frac{1}{N}\sum_{n=1}^N \varphi(y^n) \Big| 
\leq
E_K(Y,N,D)  \, \|\varphi\|_{\Hcal_K}, 
\ee
where 
\be
E_K(Y,N,D)^2 = \frac{1}{N^2}\sum_{n,m=1}^N K(y^n,y^m)- 1,
\ee
where $Y$ is any sequence of $N$ points in $\Cbf$. Moreover, setting $\Yt:= \arg \inf_{Y \in \Cbf^N} E_K(Y,N,D)$, one has
\bel{AC}
E_K(N,D)^2 
\leq
 \sum_{n > N} {\rho(\alpha^{*n}) \over N^2} \Big|\sum_{n=1}^N e^{2i \pi <y^n,\alpha^{*n}>} \Big|^2, 
\ee
where the ordering $n \mapsto \alpha^n$ is defined so that $n \mapsto \rho(\alpha^{*n})$ 
is a decreasing sequence along the points $\alpha \in L^*$ while
 $y^1,\ldots,y^N$ are defined by solving the set of equations (when $\alpha^{*n} \neq 0$) 
\bel{COND1}
\sum_{m=1}^N e^{2i \pi < y^m,\alpha^{*n}>} = 0, \qquad 1 \leq n \leq N. 
\ee
\end{proposition}

We conjecture that the system \eqref{COND1} does admit a solution for any set $\{ \alpha^{*n} \in L^* \}_{1 \leq n \leq N}$. Numerically, for this system for each of the periodic kernels of interest we computed (see below) a numerical approximation  for a broad range of dimensions $D$ and integers $N$. 
Moreover, the existence of such points can be established rigorously for several examples of kernels. 
Importantly, we tested numerically (see below) that the following estimate is sharp:
\bel{SHARPESTIM}
E_K(N,D) \simeq \frac{1}{N} \sum_{n > N} \rho(\alpha^{*n}).
\ee
In our examples, we will be able to compute the above sum, both  numerically and analytically, and it will be proven to provide us with a {\sl sharp estimate} of the discrepancy error, altough we cannot establish rigourously the validity of this estimate. 
In the context of Proposition \ref{PEFFV}, we will compute numerically   
\bel{sumnN}
  \sum_{n > N} \rho(\alpha^{*n}) = \sum_{n \ge 0} \rho(\alpha^{*n}) - \sum_{n \le N} \rho(\alpha^{*n}) = K^\per(0,0) - \sum_{n \le N} \rho(\alpha^{*n}).
\ee

%-----------------------------------------

Another theoretical standpoint is obtained by the following ``level-set argument'', namely by relying on the approximation 
\bel{LSM}
\sum_{n > N} \rho(\alpha^{*n}) \simeq\int_{\xi / \, \rho(\xi) < m^{-1}(N)} \rho(\xi) d\xi,
\ee
where $m^{-1}$ is the inverse of the function $m(\epsilon) = |\{ \rho(\xi) > \epsilon\}|$, where $\rho(\xi)$ is a smooth extension of $\rho$ on the whole space $\RD$ satisfying 
$\int_{\RD} \rho(\xi) \, d\xi = 1$ with 
$\rho(\xi) \ge 0$ and $\rho(0) = 1$. Unfortunately, this approximation is very inaccurate in the examples we have considered. 
 
%--------------------------------------------------------------------- 

In particular, if a transport map of the function is known, that is a
one-to-one map $y :\RD \to \Omega$ such that
$dy = \rho(\xi)d\xi$, then, denoting its inverse as $\xi(y)$, the last integral reduces to measuring the following set
\bel{LSM-99}
\sum_{n > N} \rho(\alpha^{*n}) 
\simeq
\int_{y: \rho(\xi(y)) < m^{-1}(N)} dy  
= \Big| \{y: \rho(\xi(y)) < m^{-1}(N)\}\Big|. 
\ee
Let us make some further remarks: 
\bei

\item If one want to study a general kernel $K(x,y) = \chi(x-y)$ with $x,y \in \Omega$, then one can compute an upper-bound using \ref{AC} with the Fourier coefficients on the doubled lattice $L_2$, generated by the cell $\Omega_2 := \{ x-y \, / \, x,y \in \Omega \}$. This is done using the Poisson formula, i.e. 
\be
K(x,y) = \chi(x-y) 
=  \frac{\big\la  e^{2i \pi <x-y, \cdot>}, \rho\big\ra_{l^2({\mathbf L^*_2)}}}{|\Omega_2|}, 
\qquad 
\rho(\xi) = \frac{1}{(2\pi)^{D/2}}\int_{\Omega_2} \chi(\zeta)e^{-2 i \pi <\xi,\zeta>} \, d\zeta. 
\ee
However, in practical terms, computing Fourier coefficients can be a quite difficult task.

\item We can specify directly the Fourier coefficients $\chi(\alpha^*)$ in the light of Proposition \ref{PEFFV}. This obviously simplifies the problem of finding suitable parameters since they are automatically set by the two conditions $\chi(0) = 1$ and $\chi \in \ell^1(\Lbf^*)$. 

\item It is clear from this result that the price to pay, when the dimension of the problem increases, is increasing also the \textit{regularity} of the kernel, and the space $\Hcal_K$ contains functions that are more regular as the dimension increases. 

\item The lattice $\Lbf$ certainly plays an important role in finding the best constant
$E_K(N,D)$. If $\chi$ is symmetric, the present discussion connects with the sphere packing problem. For instance, for $D=2$, the best lattice is the hexagonal one, for $D=3$ the dodecahedral one (Kepler conjecture), while $D=8,24$ was treated in \cite{Viazovska}. Other cases remain opened.

\eei

\begin{proof}[Proof of Proposition \ref{PEFFV}]
To derive \eqref{PERROR}, we recall from \eqref{equa:45} that  
$$ 
E_K(Y,N,D)^2
=  \iint_{\Cbf \times \Cbf} K(x,y) dxdy 
+ \frac{1}{N^2} \sum_{n,m=1}^N K(y^n,y^m) -\frac{2}{N} \sum_n \int_\Omega K(x,y^n)  dx. 
$$
We use the fact that $K(x,y) = \chi(x-y)$ is $\Lbf$-periodic and deduce that the right-hand side is a constant, and does not depend on $y^n$,  therefore 
$$
\int_\Cbf K(x,y)  \, dx 
= C_0
= \iint_{\Cbf \times \Cbf} K(x,y)\,  dxdy,
\qquad 
 y \in \Cbf,
$$
since $|\Cbf|=1$.
%  and the condition $\chi(0)=1$ implies $C=1$.
%
Consider \eqref{PER} and the error function expressed in Fourier variable \eqref{458} with $s=1$ and $p=2$, that is, 
$$
E_K(Y,N,D)^2
 = \sum_{\alpha^* \in \Lbf^*} \chi(\alpha^{*}) \Big| \int_\Cbf e^{2i \pi <x,\alpha^*>} \, dx 
- {1 \over N} \sum_{n=1}^N e^{2i \pi <y^n,\alpha^*>} \Big|^2. 
$$
We have 
$$
\int_\Cbf e^{2i \pi<x, \alpha^{*n} >} \, dx  
= 1 \, 
\text{ if } \alpha^{*n}=0, \quad
\text{ while it is } 0 \text{ otherwise}.
$$
When $\alpha^* \neq 0$ we consider   
\bel{4581}
E_K(Y,N,D)^2
= \sum_{\alpha^* \in \Lbf^*, \alpha^* \neq 0} { \chi(\alpha^{*}) \over N^2} \Big| \sum_{n=1}^N e^{2i \pi <y^n,\alpha^*>} \Big|^2. 
\ee 
Using the ordering $n \mapsto \alpha^n$  associated with the function $\chi(\alpha^{*n}$ we arrive at  
the system of equations \eqref{COND1}.  
\end{proof}

%---------------------------------------------------------------------- 

The existence of solutions to \eqref{COND1} can for instance be established for the canonical lattice $\Cbf=[0,1]^D$ and $\Lbf=\Lbf^*=\ZZ^D$, 
and any set of points
$
\{ \alpha^{*}:  L_d \le \alpha^*_d < R_d \},
$
with size $N = \prod_{d=1}^D (R_d - L_d)$. 
Without loss of generality we assume that $L_d = 0$ and we set $N_e = \prod_{d=1}^{e-1} R_d $ 
for $e=1, \ldots, D$ (with the convention that $N_1=1$). We define a one-to-one map from the set $\{1 \le n \le N\}$ to $\{ 0 \le n_d < R_d \}_{d=1, \ldots, D}$ by using the map in 
$$
n(n_1, \ldots, n_D) 
= \sum_{d \le D} N_d n_d, 
\qquad \alpha^{*n} = \Big( n_1,\ldots,n_D \Big),
\qquad y^n =  \Big( \frac{ n_d(n)}{R_d} \Big)_{d=1, \ldots, D}. 
$$
We compute
$$
\aligned
  \sum_{n=1}^N e^{2 i \pi < y^n,\alpha^{*m}>} 
&
= \sum_{n=1}^N e^{2 i \pi \sum_d \frac{ n_d(n)}{R_d}n_d(m)} = \sum_{n_1=L_1}^{R_1-1} e^{2 i \pi \frac{n_1}{R_1}n_1^m}  \ldots  \sum_{n_D=L_D}^{R_D - 1} e^{2 i \pi \frac{ n_D^n}{R_D}n_D^m} 
\\
& =  \Big( \frac{1-e^{2 i \pi n_1^m}}{1-e^{2 i \pi \frac{ n_1^m}{ R_1}}} \Big)
 \ldots 
\Big( \frac{1-e^{2 i \pi n_D^m}}{1-e^{2 i \pi \frac{ n_D^m}{ R_D}}} \Big) = 0.
\endaligned
$$

%-------------------------------

Moreover, for computing the solutions to \eqref{COND1} for an arbitrary set $\{\alpha^{*n} \in \ZZ^D \}$ with size $N$, we can consider the functional
$
I(Y) := \sum_{n=1}^{N} \Big| \sum_{m=1}^N e^{2i \pi <y^m,\alpha^{*n}>} \Big|^2. 
$ 
It is regular and admits the derivatives  
$
\del_{y^k} I(Y)
= \sum_{n=1}^{N} 4 i \pi\alpha^{*n} \sum_{m=1}^N e^{2i \pi <y^m-y^k,\alpha^{*n}>}. 
$
Therefore, we can  use a gradient-based method in order to numerically compute the solutions, and 
we expect that any local minimum of the functional $I(Y)$ will also be a global minimum.

%===============================================================================

\section{A comparative study for a selection of kernels}
\label{sec-numerics}

\subsection{Our strategy for the numerical study}
\label{introduction}

The aim of this section is to numerically compute the optimal sequences $\Ybar$ and the discrepancy functions sdefined by (with $\Omega = [0,1]^D$) 
\be
  \Ybar = \arg \inf_{Y \in \Omega^{N}} E_K(Y,N,D), 
\qquad E_K(N,D) = E_K(\Ybar,N,D). 
\ee
We treat two classes of interest:
\begin{itemize} 

\item Transported kernel: we recall (see \eqref{equa:45}) that the following quantity provides the worst Monte-Carlo-type integration error for any admissible kernel $K=K^\loc$ localized to the set $\Omega$: 
\bel{tatitata} 
E_K(Y,N,D)^2 
= \frac{1}{N^2} \sum_{n,m=1}^N \iint_{\Omega \times \Omega} \Big(   K^\loc(x,y) +   K^\loc(y^n,y^m) -  K^\loc(y^n, y)  -  K^\loc(x,y^m) \Big) dxdy. 
\ee

\item Periodic kernels: in the context of Proposition \ref{PEFFV}, the expression of the discrepancy function takes a simpler form and in the limit  $N,D \mapsto \infty$ we can compare our results to the asymptotic expression
% rate 
\bel{AC2}
E_K(Y,N,D)^2 =  \frac{1}{N^2} \sum_{n,m=1}^N  K^\per(y^n,y^m) -  1, 
\qquad   
E_K(N,D) \leq \sum_{n > N} \rho(\alpha^{*n}).
\ee
We recall that $\alpha^n$ are chosen on the dual lattice so that $n \mapsto \chi(\alpha^{*n})$ is decreasing. We choose here the canonical lattice $\mathbf{L} = \ZZ^D$ for simplicity.
\end{itemize}

We start from the four examples in Table~\ref{FDK}, that are defined in
the whole space $\RD$, and design the eight kernels listed in Table~\ref{FDK2}. A normalization parameter, depending on the dimension was determined explicitly for each case.
In each case, we compute $E_K(Y,N,D)$ for the following two sequences: 

\begin{itemize}
\item
  $Y$ is a randomly chosen sequence; 

\item   $Y$ is a numerical sequence that approximates the optimal solution $\Ybar$.
\end{itemize}
\noindent The numerical results for $E_K(Y,N,D)$ are displayed as tables
for $N=16,32, \ldots, 512$ and $D=1,2,4,\ldots,128$.
%with $Y \in [0,1]^{N \times D}$ so that 
and we can compare the error obtained for various values $N,D$ and various kernels.
The kernels in the second line of Table~\ref{FDK2} are periodic, so that we can compute the upper bounds given by the formula \eqref{AC} and we can compare it to the numerically computed results. 

%---------------------------------------------------------------------------------------------------------------------------------

\subsection{Periodic kernels}
\label{lattice-based-kernels}

\vskip.05cm \par{\bf Strategy for periodic kernels.}
For each of the four periodic kernels in Table~\ref{FDK2}, we present three tables:

\begin{itemize}

\item[(A)] $E_K(Y,N,D)$ for a randomly chosen sequence $Y$. 
%  \in [0,1]^{N \times D}$. 

\item[(B)] $E_K(Y,N,D)$ for a numerical sequence that approximates the optimal solution $\Ybar$.

\item[(C)] $ \sqrt{{\frac{1 }{N}} \sum_{n > N} \rho(\alpha^{*n})}$ which is our theoretical (but yet heuristic) asymptotic  rate of convergence \eqref{SHARPESTIM}. 
\end{itemize}
From a computational point of view, some remarks are in order:
\begin{itemize}
\item In item (A) above, the random sequences $Y$ are computed using the standard Mersenne-Twister mt19937. The error bound $E_K(Y,N,D)$ can be computed using the left-hand formula in \eqref{AC2}.

\item In item (B) above, in order to compute the optimal $Y$ we have two strategies: 

\begin{itemize}

\item We can either minimizing directly the left-hand side of \eqref{AC2}. This optimization problem can be solved easily by a gradient descent and is computationally tractable. However, depending upon the choice of the kernel this method can show poor performances (i presence of almost vanishing gradients.

\item Or else, we can compute the first values $\alpha^{*n}$ for $n=0, \ldots, N-1$ and then solve the equation \eqref{COND1}. We recall here that 
$
\sum_{m=1}^N e^{2i \pi < y^m,\alpha^{*n}>} = 0$
for $1 \leq n \leq N$. 
We used this method and it always gave ni principle good numerical results. However, to solve the relevant system we must use an algorithm whose complexity is of order $N^3$, that leads to extremely long execution times as $N$ becomes large.

\item Once the sequence $Y$ is computed, the error $E_K(Y,N,D)$ is obtained using the left-hand side of \eqref{AC2}.
\end{itemize}

\item Finally, in item (C) above, we compute first the set of values $\chi(\alpha^{*n})$ for 
$n=1, \ldots, N$ by using a graph-search type algorithm. Observe that this step can be challenging in the case of a rather oscillatory function $\chi$ such as the truncated kernel.
\end{itemize}

%-------------------------------------------------------------------

For each periodic kernel we provide below the asymptotic formula derived from our level-set arguments, although it is not very accurate within the range of $N,D$ under consideration.
We plot in this section the distribution of sequences of points approximating the best discrepancy sequences $\Yb$. These figures corresponds to $N=256$ and $ D=2$. Observe that none of these distributions is radially-symmetric, as the corresponding periodic kernels are not.

\begin{figure}

\centering{\includegraphics[width=0.45\linewidth]{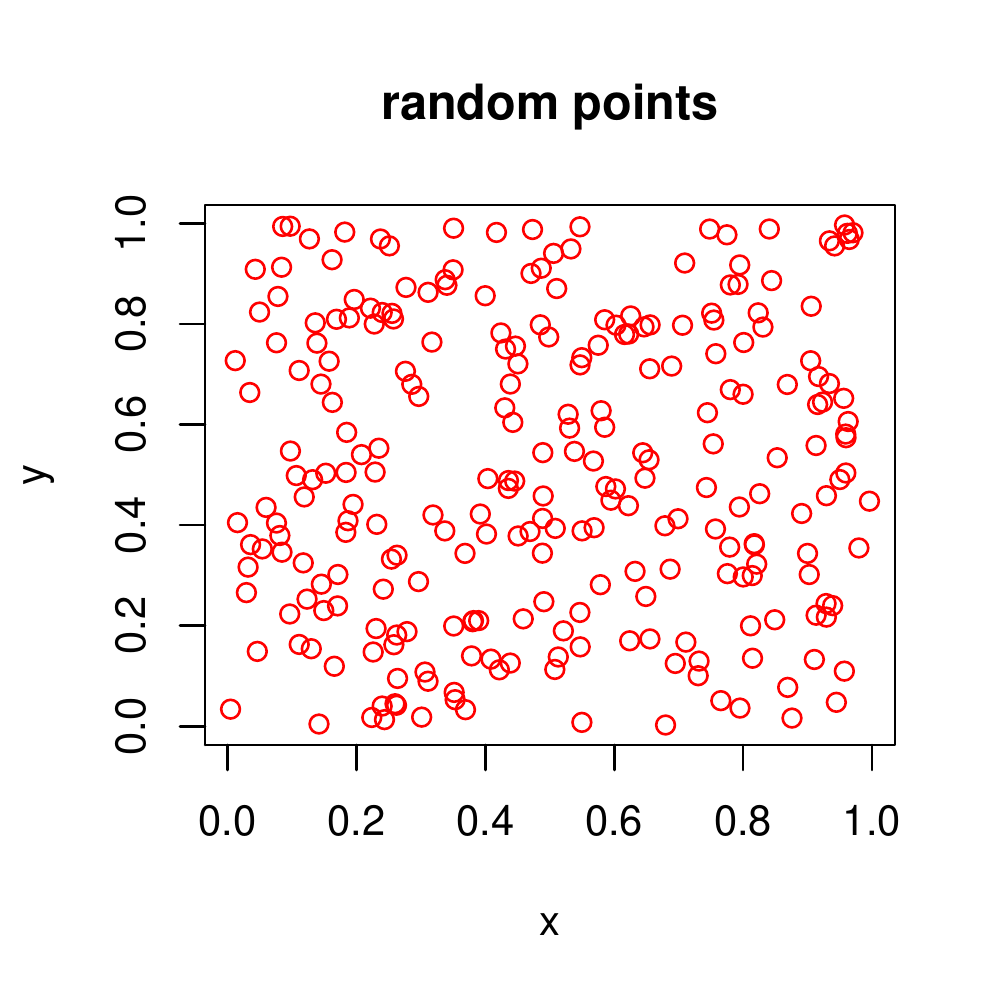} \includegraphics[width=0.45\linewidth]{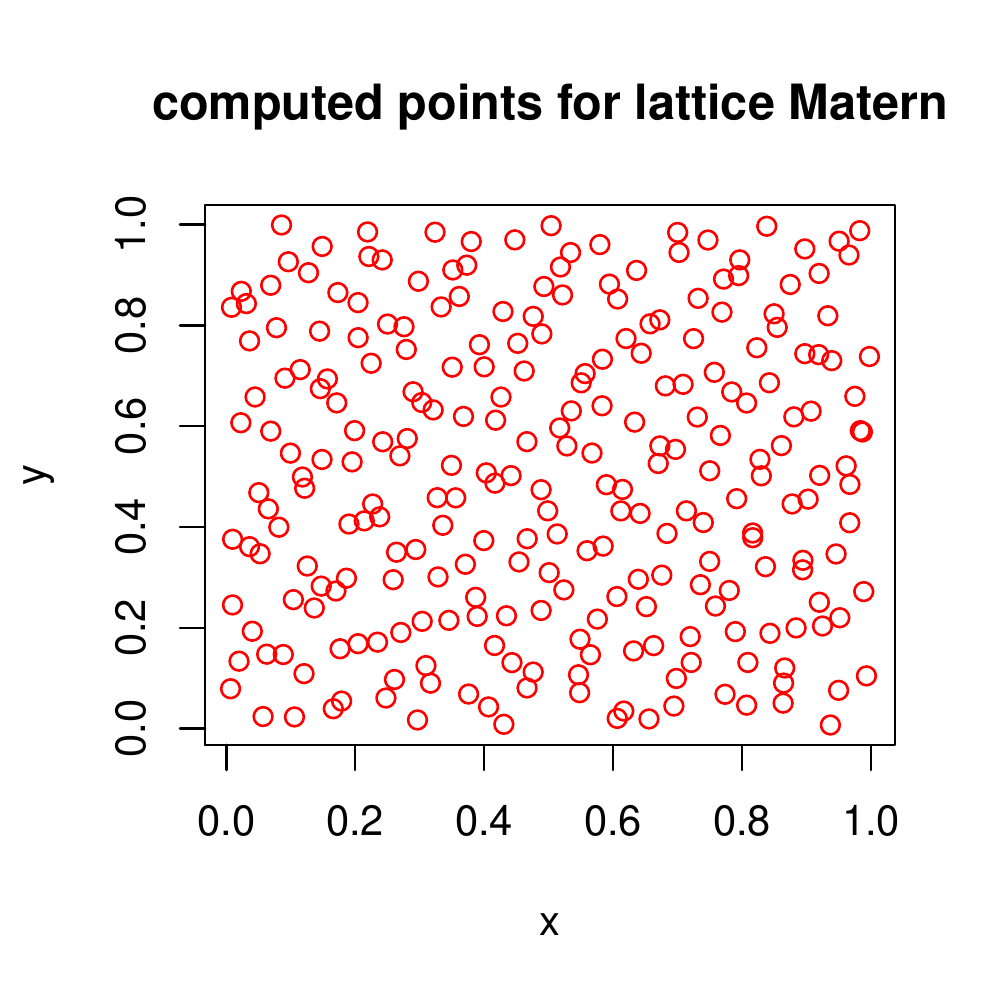} \includegraphics[width=0.45\linewidth]{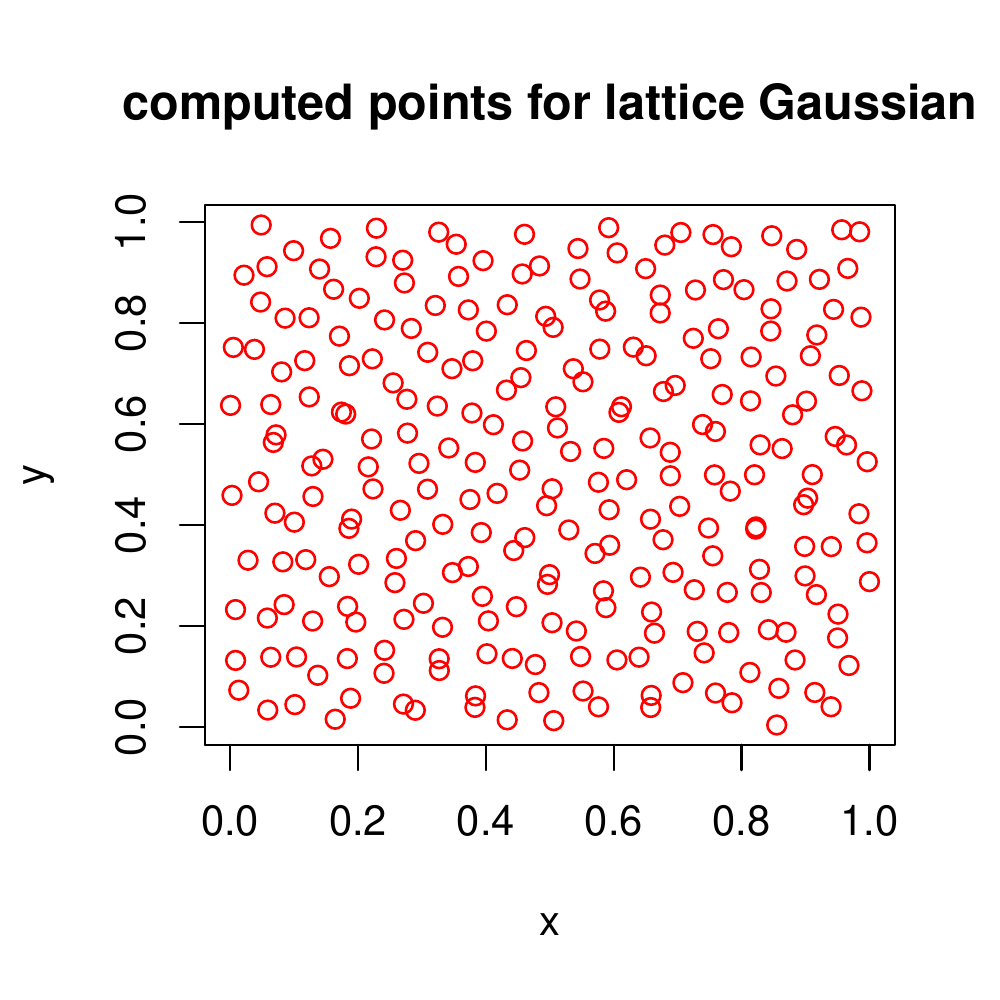} \includegraphics[width=0.45\linewidth]{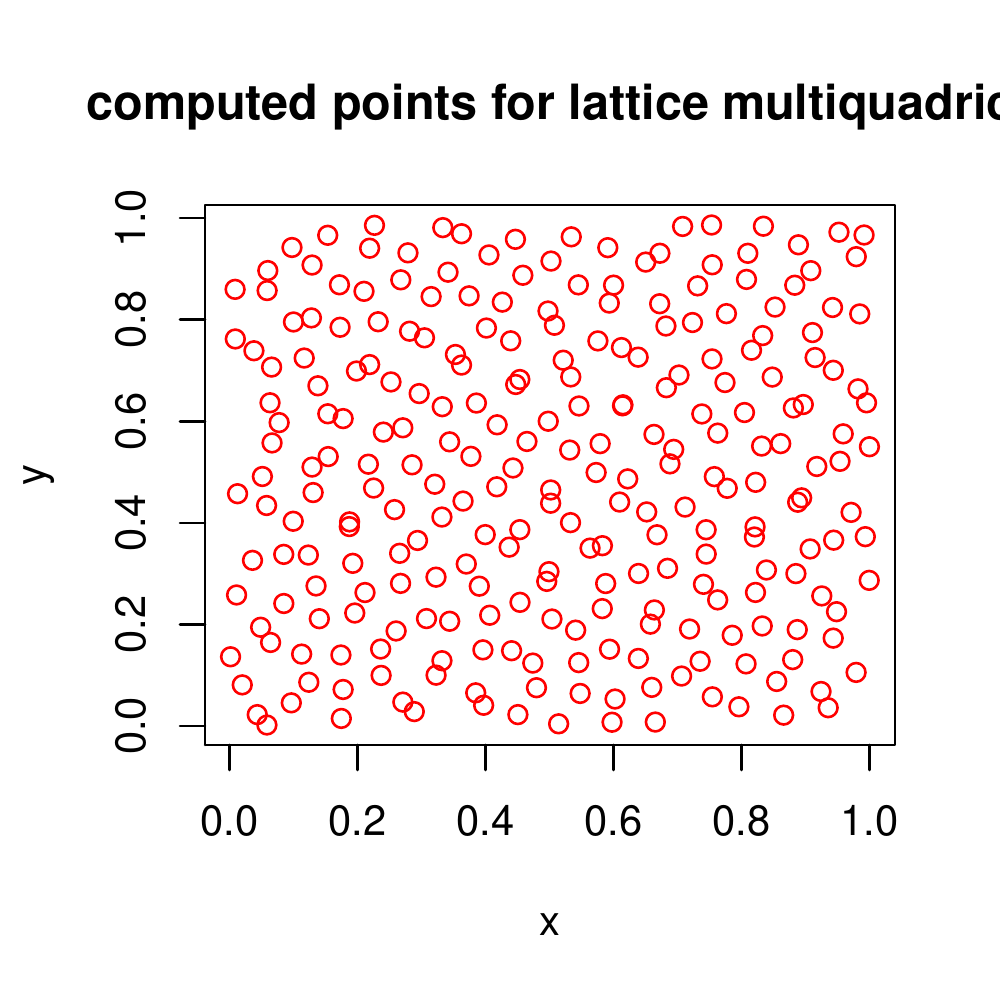} \includegraphics[width=0.45\linewidth]{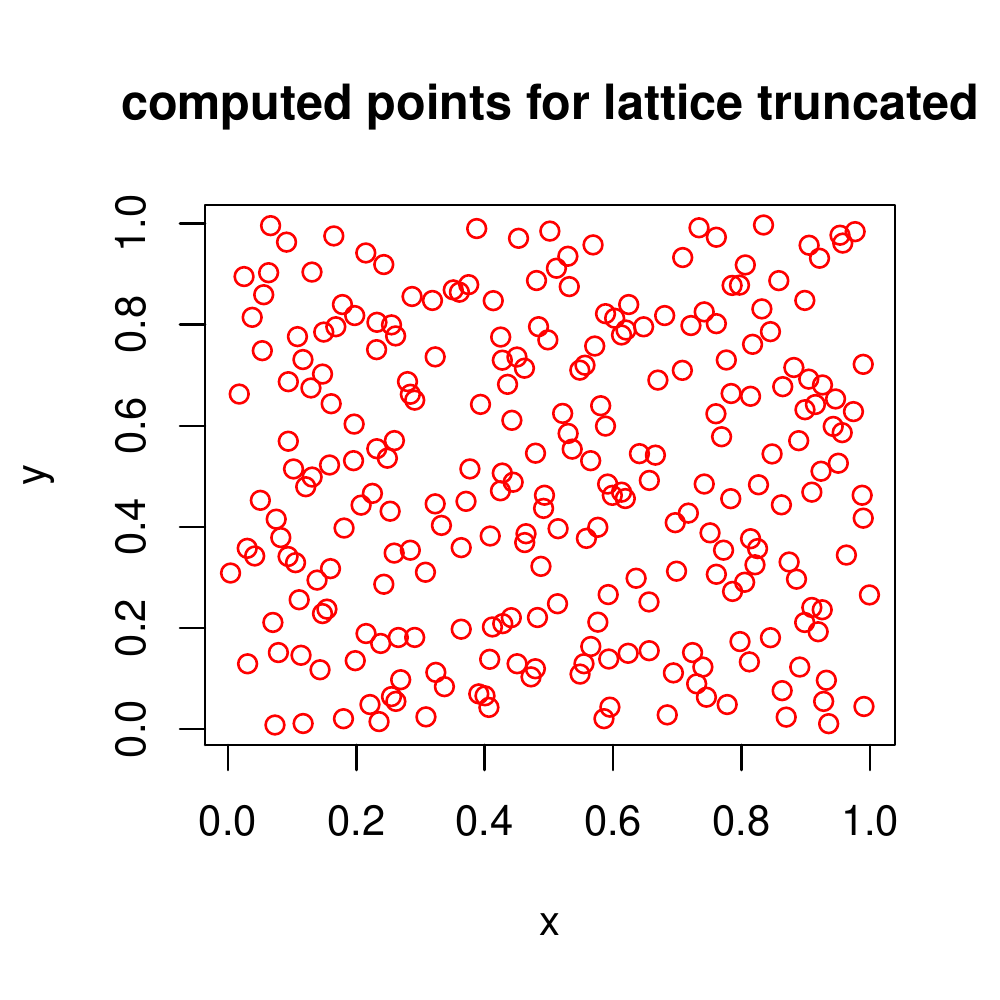} 
}
\caption{\label{DISTRIBLATTICE}  Random and numerical sequences for periodic kernels with $N=256$}\label{fig:DISTRIBLATTICE}
\end{figure}

%---------------------------------------------------------------------------------------------------------------

\vskip.05cm \par{\bf Numerical results for the periodic tensorial exponential kernel.}
For the periodic kernel \eqref{exponential} the level-set method \eqref{LSM} applies and lead us to the bound
% %  (see appendix): %\bel{MALS}
$E_K(N,D) \lesssim \frac{\log(N)^{D-1}}{N}$, which is probably far from 
being optimal but coincides with the Koksma-Hlawka inequality for
low-discrepancy sequences.
%%%%%%%% \footnote{see for instance en.wikipedia.org/wiki/Low-discrepancy-sequence}.
Here, the space $\Hcal_K^{1,1}([0,1]^D)$ (see \eqref{MAT1P}) is similar to the space of functions with bounded variation  $BV([0,1]^D)$. The three relevant tables are given here. 
Observe that the discrepancy error for the numerical points is
always much smaller than the discrepancy error for the random points, as was expected. 

Note that the asymptotic convergence rate is quite close to the one obtained with the numerical sequence in Table~\ref{decadix}. There exist cases for which this convergence rate is smaller than the
theoretical bound. This might seem to be a contradiction, since the latter bound is supposed to be the minimum over all possible sequences.  
However, several sources of  (yet small) numerical error arise due to the search algorithm of the decreasing sequence $\alpha^{*n}$. In addition, a second (small) error arises from the estimate \ref{sumnN}, which is probably sharp but yet only an approximation of the theoretical bound.

%------------------------------------------------------------------------------------------------------------------

\begin{longtable}[]{@{}lrrrrrrrr@{}}
\caption{$E_K(Y,N,D)$ for periodic tensorial exponential with random points $Y$}\tabularnewline
\toprule
& D= 1 & D= 2 & D= 4 & D= 8 & D= 16 & D= 32 & D= 64 & D=
128\tabularnewline
\midrule
\endfirsthead
\toprule
& D= 1 & D= 2 & D= 4 & D= 8 & D= 16 & D= 32 & D= 64 & D=
128\tabularnewline
\midrule
\endhead
N= 16 & 0.228 & 0.365 & 0.355 & 0.312 & 0.304 & 0.308 & 0.326 &
0.319\tabularnewline
N= 32 & 0.307 & 0.222 & 0.216 & 0.199 & 0.210 & 0.228 & 0.231 &
0.234\tabularnewline
N= 64 & 0.154 & 0.142 & 0.138 & 0.139 & 0.150 & 0.168 & 0.163 &
0.165\tabularnewline
N= 128 & 0.117 & 0.088 & 0.087 & 0.097 & 0.111 & 0.116 & 0.114 &
0.115\tabularnewline
N= 256 & 0.060 & 0.053 & 0.061 & 0.077 & 0.075 & 0.081 & 0.081 &
0.082\tabularnewline
N= 512 & 0.035 & 0.038 & 0.050 & 0.049 & 0.054 & 0.057 & 0.057 &
0.059\tabularnewline
\bottomrule
\end{longtable}

\begin{longtable}[]{@{}lrrrrrrrr@{}} 
\caption{$E_K(Y,N,D)$ for periodic tensorial exponential with numerical points $Y$}\tabularnewline
\toprule
& D= 1 & D= 2 & D= 4 & D= 8 & D= 16 & D= 32 & D= 64 & D=
128\tabularnewline
\midrule
\endfirsthead
\toprule
& D= 1 & D= 2 & D= 4 & D= 8 & D= 16 & D= 32 & D= 64 & D=
128\tabularnewline
\midrule
\endhead
N= 16 & 0.062 & 0.126 & 0.172 & 0.195 & 0.211 & 0.217 & 0.221 &
0.223\tabularnewline
N= 32 & 0.031 & 0.075 & 0.114 & 0.131 & 0.143 & 0.149 & 0.151 &
0.153\tabularnewline
N= 64 & 0.016 & 0.049 & 0.076 & 0.090 & 0.099 & 0.103 & 0.106 &
0.107\tabularnewline
N= 128 & 0.008 & 0.030 & 0.051 & 0.063 & 0.069 & 0.073 & 0.074 &
0.077\tabularnewline
N= 256 & 0.004 & 0.020 & 0.034 & 0.043 & 0.048 & 0.051 & 0.054 &
0.061\tabularnewline
N= 512 & 0.002 & 0.012 & 0.022 & 0.030 & 0.034 & 0.037 & 0.042 &
0.049\tabularnewline
\bottomrule
\label{decadix}
\end{longtable}

%--------------------------------------------------------------------------------------------------------------

\begin{longtable}[]{@{}lrrrrrrrr@{}}
\caption{$E_K(N,D)$ for periodic tensorial exponential with the asymptotic formula}\tabularnewline
\toprule
& D= 1 & D= 2 & D= 4 & D= 8 & D= 16 & D= 32 & D= 64 & D=
128\tabularnewline
\midrule
\endfirsthead
\toprule
& D= 1 & D= 2 & D= 4 & D= 8 & D= 16 & D= 32 & D= 64 & D=
128\tabularnewline
\midrule
\endhead
N= 16 & 0.069 & 0.143 & 0.202 & 0.245 & 0.288 & 0.308 & 0.318 &
0.323\tabularnewline
N= 32 & 0.034 & 0.082 & 0.129 & 0.157 & 0.179 & 0.207 & 0.220 &
0.226\tabularnewline
N= 64 & 0.017 & 0.046 & 0.078 & 0.102 & 0.116 & 0.129 & 0.147 &
0.156\tabularnewline
N= 128 & 0.009 & 0.026 & 0.048 & 0.067 & 0.077 & 0.084 & 0.092 &
0.105\tabularnewline
N= 256 & 0.004 & 0.014 & 0.029 & 0.042 & 0.052 & 0.056 & 0.060 &
0.066\tabularnewline
N= 512 & 0.002 & 0.008 & 0.018 & 0.027 & 0.034 & 0.038 & 0.040 &
0.043\tabularnewline
\bottomrule
\end{longtable}

%============================================================================
%--------------------------------------------------------------------------------------------------------------------------------

\vskip.05cm \par{\bf Numerical results for the periodic multiquadric kernel.}
For the periodic multiquadric kernel and using the level-set argument in \eqref{LSM} we can derive the bound  
% \bel{MULS}
$E_K(N,D) \lesssim
 \frac{2^D}{\tau_D^D}\Big(1 - \exp\big(\frac{-N^{1/D} D}{2\tau_D}\big) \Big)$. 
The three tables are now presented and the same observations as above for the exponential kernel can be made. Note again the existence of a small numerical error for the two-dimensional case, as the error discrepancy for the optimized sequence is not decreasing, namely for $D=2$ and $N=256,512$. 
Observe that the error vanishes for the one-dimensional case provides $N\ge 64$. This was expected since we are dealing with a kernel generating a space consisting of periodic function, whose Fourier coefficient decreases at exponential rate, as predicted by our theoretical convergence rate above. From a numerical point of view, it is expected to be a finite dimensional space, having few basis functions in low dimensions.

%-------------------------------------------------------------------------------------------------------------------
\begin{longtable}[]{@{}lrrrrrrrr@{}}
\caption{$E_K(Y,N,D)$ for periodic tensorial multiquadric with random points $Y$}\tabularnewline
\toprule
& D= 1 & D= 2 & D= 4 & D= 8 & D= 16 & D= 32 & D= 64 & D=
128\tabularnewline
\midrule
\endfirsthead
\toprule
& D= 1 & D= 2 & D= 4 & D= 8 & D= 16 & D= 32 & D= 64 & D=
128\tabularnewline
\midrule
\endhead
N= 16 & 0.249 & 0.410 & 0.392 & 0.315 & 0.302 & 0.301 & 0.325 &
0.311\tabularnewline
N= 32 & 0.349 & 0.245 & 0.230 & 0.196 & 0.201 & 0.227 & 0.226 &
0.234\tabularnewline
N= 64 & 0.171 & 0.150 & 0.140 & 0.136 & 0.143 & 0.169 & 0.163 &
0.167\tabularnewline
N= 128 & 0.130 & 0.090 & 0.086 & 0.093 & 0.108 & 0.117 & 0.115 &
0.114\tabularnewline
N= 256 & 0.066 & 0.050 & 0.059 & 0.075 & 0.075 & 0.081 & 0.081 &
0.082\tabularnewline
N= 512 & 0.036 & 0.036 & 0.049 & 0.048 & 0.055 & 0.057 & 0.057 &
0.059\tabularnewline
\bottomrule
\end{longtable}
%---------------------------------------------------------------------------------------------------------------------
\begin{longtable}[]{@{}lrrrrrrrr@{}}
\caption{$E_K(Y,N,D)$ for periodic tensorial multiquadric with numerical points $Y$}\tabularnewline
\toprule
& D= 1 & D= 2 & D= 4 & D= 8 & D= 16 & D= 32 & D= 64 & D=
128\tabularnewline
\midrule
\endfirsthead
\toprule
& D= 1 & D= 2 & D= 4 & D= 8 & D= 16 & D= 32 & D= 64 & D=
128\tabularnewline
\midrule
\endhead
N= 16 & 0.002 & 0.078 & 0.172 & 0.204 & 0.261 & 0.277 & 0.313 &
0.306\tabularnewline
N= 32 & 0.000 & 0.030 & 0.095 & 0.128 & 0.149 & 0.198 & 0.210 &
0.227\tabularnewline
N= 64 & 0.000 & 0.005 & 0.045 & 0.081 & 0.103 & 0.106 & 0.140 &
0.157\tabularnewline
N= 128 & 0.000 & 0.001 & 0.018 & 0.044 & 0.067 & 0.074 & 0.075 &
0.098\tabularnewline
N= 256 & 0.000 & 0.002 & 0.007 & 0.024 & 0.042 & 0.051 & 0.052 &
0.053\tabularnewline
N= 512 & 0.000 & 0.007 & 0.003 & 0.014 & 0.021 & 0.034 & 0.037 &
0.037\tabularnewline
\bottomrule
\end{longtable}

%-----------------------------------------------------------------------------------------------------------------------

\begin{longtable}[]{@{}lrrrrrrrr@{}}
\caption{$E_K(N,D)$ for periodic tensorial multiquadric with the asymptotic formula}\tabularnewline
\toprule
& D= 1 & D= 2 & D= 4 & D= 8 & D= 16 & D= 32 & D= 64 & D=
128\tabularnewline
\midrule
\endfirsthead
\toprule
& D= 1 & D= 2 & D= 4 & D= 8 & D= 16 & D= 32 & D= 64 & D=
128\tabularnewline
\midrule
\endhead
N= 16 & 0.004 & 0.081 & 0.171 & 0.207 & 0.272 & 0.301 & 0.314 &
0.321\tabularnewline
N= 32 & 0.000 & 0.027 & 0.092 & 0.134 & 0.148 & 0.194 & 0.213 &
0.223\tabularnewline
N= 64 & 0.000 & 0.005 & 0.044 & 0.085 & 0.100 & 0.105 & 0.137 &
0.151\tabularnewline
N= 128 & 0.000 & 0.001 & 0.017 & 0.043 & 0.067 & 0.073 & 0.075 &
0.097\tabularnewline
N= 256 & 0.000 & 0.000 & 0.008 & 0.025 & 0.043 & 0.050 & 0.052 &
0.053\tabularnewline
N= 512 & 0.000 & 0.000 & 0.003 & 0.014 & 0.021 & 0.034 & 0.036 &
0.037\tabularnewline
\bottomrule
\end{longtable}
%-----------------------------------------------------------------------------------------------------------------------------------
%-----------------------------------------------------------------------------------------------------------------------------------

\vskip.05cm \par{\bf Numerical results for the periodic Gaussian kernel.}
For the periodic kernel \eqref{KJ} we can use our level-set arguments in \eqref{LSM} and we c an arribve at the  bound
% (cf. the appendix): \bel{JTLS}
$
E_K(N,D) \le N^{1- 1/D} \exp(-N^{2/D})$. 
We now present the three tables of interest and again the results are quite similar to the exponential and multiquadric cases.

%-----------------------------------------------------------------------------------------------------------------------------------
\begin{longtable}[]{@{}lrrrrrrrr@{}}
\caption{$E_K(Y,N,D)$ for periodic Gaussian with random points $Y$}\tabularnewline
\toprule
& D= 1 & D= 2 & D= 4 & D= 8 & D= 16 & D= 32 & D= 64 & D=
128\tabularnewline
\midrule
\endfirsthead
\toprule
& D= 1 & D= 2 & D= 4 & D= 8 & D= 16 & D= 32 & D= 64 & D=
128\tabularnewline
\midrule
\endhead
N= 16 & 0.279 & 0.445 & 0.400 & 0.316 & 0.301 & 0.301 & 0.325 &
0.311\tabularnewline
N= 32 & 0.415 & 0.267 & 0.235 & 0.196 & 0.200 & 0.227 & 0.226 &
0.234\tabularnewline
N= 64 & 0.205 & 0.159 & 0.138 & 0.136 & 0.142 & 0.168 & 0.163 &
0.167\tabularnewline
N= 128 & 0.152 & 0.090 & 0.085 & 0.092 & 0.107 & 0.117 & 0.115 &
0.114\tabularnewline
N= 256 & 0.075 & 0.047 & 0.059 & 0.074 & 0.075 & 0.081 & 0.081 &
0.082\tabularnewline
N= 512 & 0.036 & 0.034 & 0.046 & 0.048 & 0.055 & 0.057 & 0.057 &
0.059\tabularnewline
\bottomrule
\end{longtable}
%-----------------------------------------------------------------------------------------------------------------------------------
%-----------------------------------------------------------------------------------------------------------------------------------
\begin{longtable}[]{@{}lrrrrrrrr@{}}
\caption{$E_K(Y,N,D)$ for periodic Gaussian with numerical points $Y$}\tabularnewline
\toprule
& D= 1 & D= 2 & D= 4 & D= 8 & D= 16 & D= 32 & D= 64 & D=
128\tabularnewline
\midrule
\endfirsthead
\toprule
& D= 1 & D= 2 & D= 4 & D= 8 & D= 16 & D= 32 & D= 64 & D=
128\tabularnewline
\midrule
\endhead
N= 16 & 0 & 0.008 & 0.164 & 0.191 & 0.258 & 0.276 & 0.313 &
0.306\tabularnewline
N= 32 & 0 & 0.000 & 0.051 & 0.123 & 0.145 & 0.197 & 0.209 &
0.227\tabularnewline
N= 64 & 0 & 0.000 & 0.013 & 0.075 & 0.099 & 0.105 & 0.140 &
0.157\tabularnewline
N= 128 & 0 & 0.000 & 0.003 & 0.033 & 0.064 & 0.072 & 0.075 &
0.098\tabularnewline
N= 256 & 0 & 0.002 & 0.000 & 0.019 & 0.041 & 0.050 & 0.052 &
0.053\tabularnewline
N= 512 & 0 & 0.008 & 0.000 & 0.008 & 0.018 & 0.033 & 0.036 &
0.037\tabularnewline
\bottomrule
\end{longtable} 
%-----------------------------------------------------------------------------------------------------------------------------------
\begin{longtable}[]{@{}lrrrrrrrr@{}}
\caption{$E_K(N,D)$ for periodic Gaussian with the asymptotic formula}\tabularnewline
\toprule
& D= 1 & D= 2 & D= 4 & D= 8 & D= 16 & D= 32 & D= 64 & D=
128\tabularnewline
\midrule
\endfirsthead
\toprule
& D= 1 & D= 2 & D= 4 & D= 8 & D= 16 & D= 32 & D= 64 & D=
128\tabularnewline
\midrule
\endhead
N= 16 & 0 & 0.018 & 0.145 & 0.198 & 0.270 & 0.300 & 0.314 &
0.321\tabularnewline
N= 32 & 0 & 0.000 & 0.052 & 0.126 & 0.145 & 0.193 & 0.213 &
0.223\tabularnewline
N= 64 & 0 & 0.000 & 0.012 & 0.077 & 0.097 & 0.104 & 0.137 &
0.151\tabularnewline
N= 128 & 0 & 0.000 & 0.002 & 0.032 & 0.065 & 0.072 & 0.074 &
0.097\tabularnewline
N= 256 & 0 & 0.000 & 0.000 & 0.020 & 0.041 & 0.050 & 0.052 &
0.053\tabularnewline
N= 512 & 0 & 0.000 & 0.000 & 0.008 & 0.018 & 0.033 & 0.036 &
0.037\tabularnewline
\bottomrule
\end{longtable}
%-----------------------------------------------------------------------------------------------------------------------------------

\vskip.05cm \par{\bf Numerical results for the periodic truncated kernel.} 
For the periodic kernel \eqref{KJ} we can use our level-set method arguments and we arrive at the
 bound
%  (see appendix for a sketch of proof)
$E_K(N,D) \le \frac{\log(N)^{D-1}}{N^2}$. 
We now present the three relevant  tables as above and the observations we made concerning these results are quite similar.

%-----------------------------------------------------------------------------------------------------------------------------------
\begin{longtable}[]{@{}lrrrrrrrr@{}}
\caption{$E_K(Y,N,D)$ for periodic tensorial truncated with random points $Y$}\tabularnewline
\toprule
& D= 1 & D= 2 & D= 4 & D= 8 & D= 16 & D= 32 & D= 64 & D=
128\tabularnewline
\midrule
\endfirsthead
\toprule
& D= 1 & D= 2 & D= 4 & D= 8 & D= 16 & D= 32 & D= 64 & D=
128\tabularnewline
\midrule
\endhead
N= 16 & 0.279 & 0.398 & 0.354 & 0.327 & 0.315 & 0.317 & 0.324 &
0.328\tabularnewline
N= 32 & 0.382 & 0.240 & 0.216 & 0.210 & 0.225 & 0.225 & 0.232 &
0.231\tabularnewline
N= 64 & 0.189 & 0.148 & 0.142 & 0.149 & 0.159 & 0.163 & 0.163 &
0.163\tabularnewline
N= 128 & 0.142 & 0.092 & 0.090 & 0.105 & 0.113 & 0.116 & 0.115 &
0.115\tabularnewline
N= 256 & 0.071 & 0.051 & 0.063 & 0.081 & 0.079 & 0.082 & 0.082 &
0.082\tabularnewline
N= 512 & 0.038 & 0.040 & 0.055 & 0.052 & 0.056 & 0.057 & 0.057 &
0.058\tabularnewline
\bottomrule
\end{longtable} 
%-----------------------------------------------------------------------------------------------------------------------------------
\begin{longtable}[]{@{}lrrrrrrrr@{}}
\caption{$E_K(Y,N,D)$ for periodic tensorial truncated with numerical points $Y$}\tabularnewline
\toprule
& D= 1 & D= 2 & D= 4 & D= 8 & D= 16 & D= 32 & D= 64 & D=
128\tabularnewline
\midrule
\endfirsthead
\toprule
& D= 1 & D= 2 & D= 4 & D= 8 & D= 16 & D= 32 & D= 64 & D=
128\tabularnewline
\midrule
\endhead
N= 16 & 0.062 & 0.100 & 0.176 & 0.209 & 0.233 & 0.276 & 0.303 &
0.315\tabularnewline
N= 32 & 0.031 & 0.058 & 0.116 & 0.140 & 0.156 & 0.173 & 0.196 &
0.214\tabularnewline
N= 64 & 0.016 & 0.035 & 0.079 & 0.096 & 0.107 & 0.119 & 0.125 &
0.139\tabularnewline
N= 128 & 0.007 & 0.021 & 0.049 & 0.067 & 0.075 & 0.081 & 0.085 &
0.090\tabularnewline
N= 256 & 0.004 & 0.013 & 0.030 & 0.046 & 0.052 & 0.056 & 0.058 &
0.061\tabularnewline
N= 512 & 0.002 & 0.010 & 0.021 & 0.032 & 0.036 & 0.039 & 0.040 &
0.042\tabularnewline
\bottomrule
\end{longtable} 
%-----------------------------------------------------------------------------------------------------------------------------------
\begin{longtable}[]{@{}lrrrrrrrr@{}}
\caption{$E_K(N,D)$ for periodic tensorial truncated with the asymptotic formula}\tabularnewline
\toprule
& D= 1 & D= 2 & D= 4 & D= 8 & D= 16 & D= 32 & D= 64 & D=
128\tabularnewline
\midrule
\endfirsthead
\toprule
& D= 1 & D= 2 & D= 4 & D= 8 & D= 16 & D= 32 & D= 64 & D=
128\tabularnewline
\midrule
\endhead
N= 16 & 0.062 & 0.127 & 0.217 & 0.289 & 0.314 & 0.322 & 0.325 &
0.327\tabularnewline
N= 32 & 0.031 & 0.077 & 0.133 & 0.188 & 0.218 & 0.227 & 0.230 &
0.231\tabularnewline
N= 64 & 0.016 & 0.042 & 0.086 & 0.114 & 0.148 & 0.159 & 0.162 &
0.163\tabularnewline
N= 128 & 0.007 & 0.023 & 0.054 & 0.073 & 0.096 & 0.110 & 0.114 &
0.115\tabularnewline
N= 256 & 0.004 & 0.013 & 0.034 & 0.050 & 0.059 & 0.075 & 0.080 &
0.081\tabularnewline
N= 512 & 0.002 & 0.007 & 0.022 & 0.034 & 0.038 & 0.049 & 0.055 &
0.057\tabularnewline
\bottomrule
\end{longtable}

%---------------------------------------------------------------------------------------------------------------------------------
%---------------------------------------------------------------------------------------------------------------------------------
%-----------------------------------------------------------------------------------------------------------------------

\subsection{Transported kernels}

The numerical results we performed on transported kernel kerne are very similar to the one we just described for periodic kernels. We content with pointing out here significant differences between periodic and transported kernels:
\bei

\item We did not numerically compute the theoretical convergence rates for the four transported kernels. Indeed, we would need to compute the Fourier coefficients appearing in Proposition~\ref{PEFFV}. 
This appears to be too costly and computationally out of reach with the existing techniques. 

\item On the other hand, we did compute the error function thanks to its expression \eqref{tatitata}, evaluated via a direct Monte-Carlo method. This is still computationally expensive, and adds some additional white noise to the final results.
\eei

Thus, in this section, we restrict ourselves with presenting the plot in Figure~\ref{DISTRIBTRANSPORTED} displaying the distribution of points that approximates the best discrepancy sequences $\Yb$. This corresponds to the choice $N=256$ and the dimension $D=2$. Some observations are in order:
\bei
\item Among our four transportation-based kernels, three are radially-symmetric, that is, Gaussian, multiquadric, and truncated. Their best discrepancy sequences are expected to be radially-symmetric. 
Indeed, our numerical tests confirm this property for the multiquadric and truncated kernels, which enjoy similar properties. 

\item However the Gaussian kernel did not lead us to a radially-symmetric result. This is due to a numerical challenge we can refer to as the ``vanishing gradient problem'' (a terminology used in the artificial intelligence community), due to the fact that the functional we want to minimize is almost flat.  
This problem did arise also with the multiquadric kernel, but we solved it by using the same trick discussed already in Section~\ref{lattice-based-kernels}.

\item The exponential kernel is not radially-symmetric so that its best discrepancy sequence is not expected to be so, and this is fully consistent with our results in  Figure~\ref{DISTRIBTRANSPORTED}.
\eei

%-------------------------------------------------------------------------------------------------------------------------------------------
\begin{figure}

\centering{\includegraphics[width=0.45\linewidth]{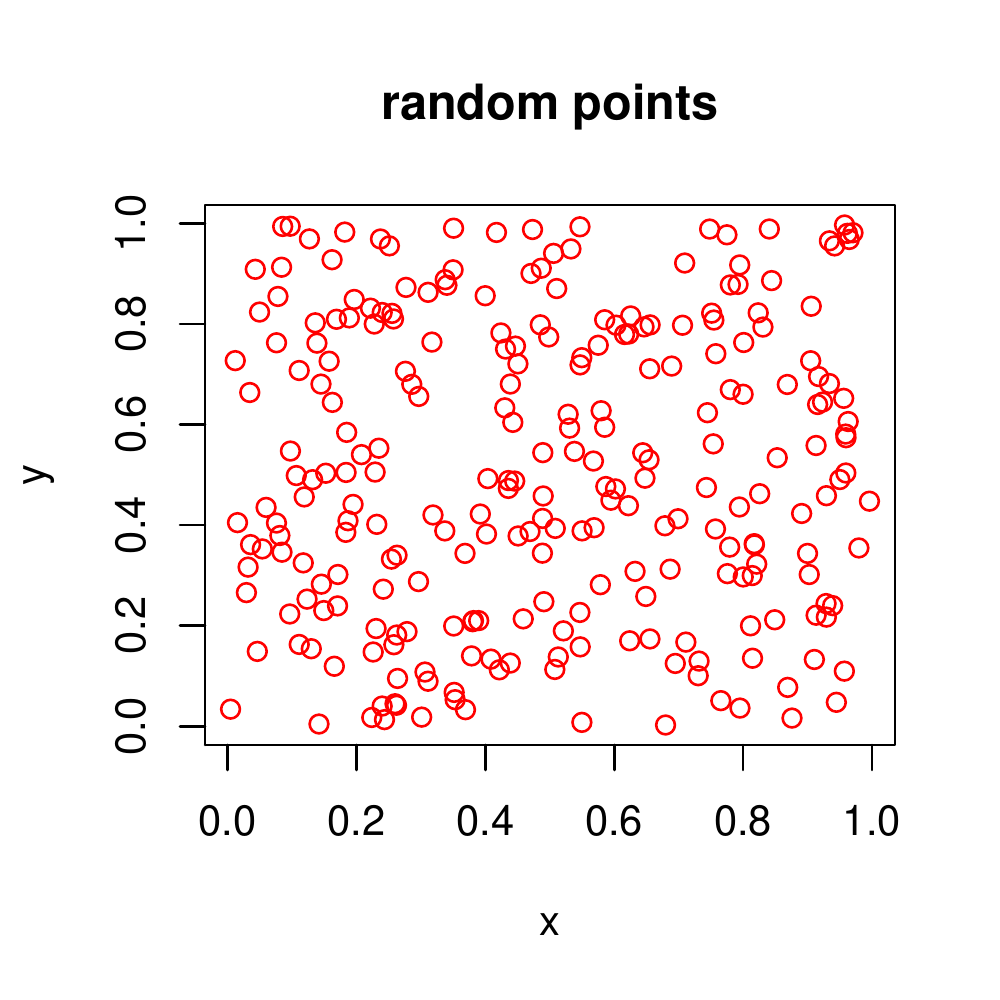} \includegraphics[width=0.45\linewidth]{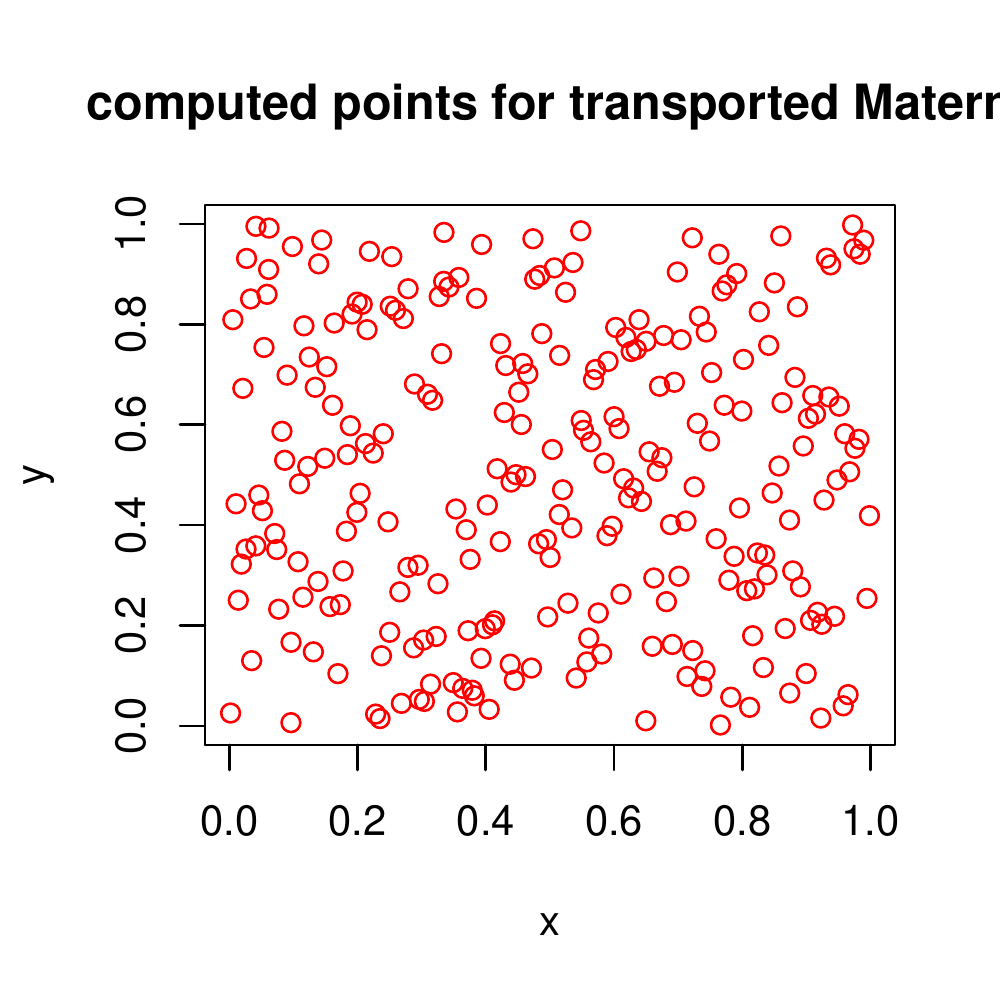} \includegraphics[width=0.45\linewidth]{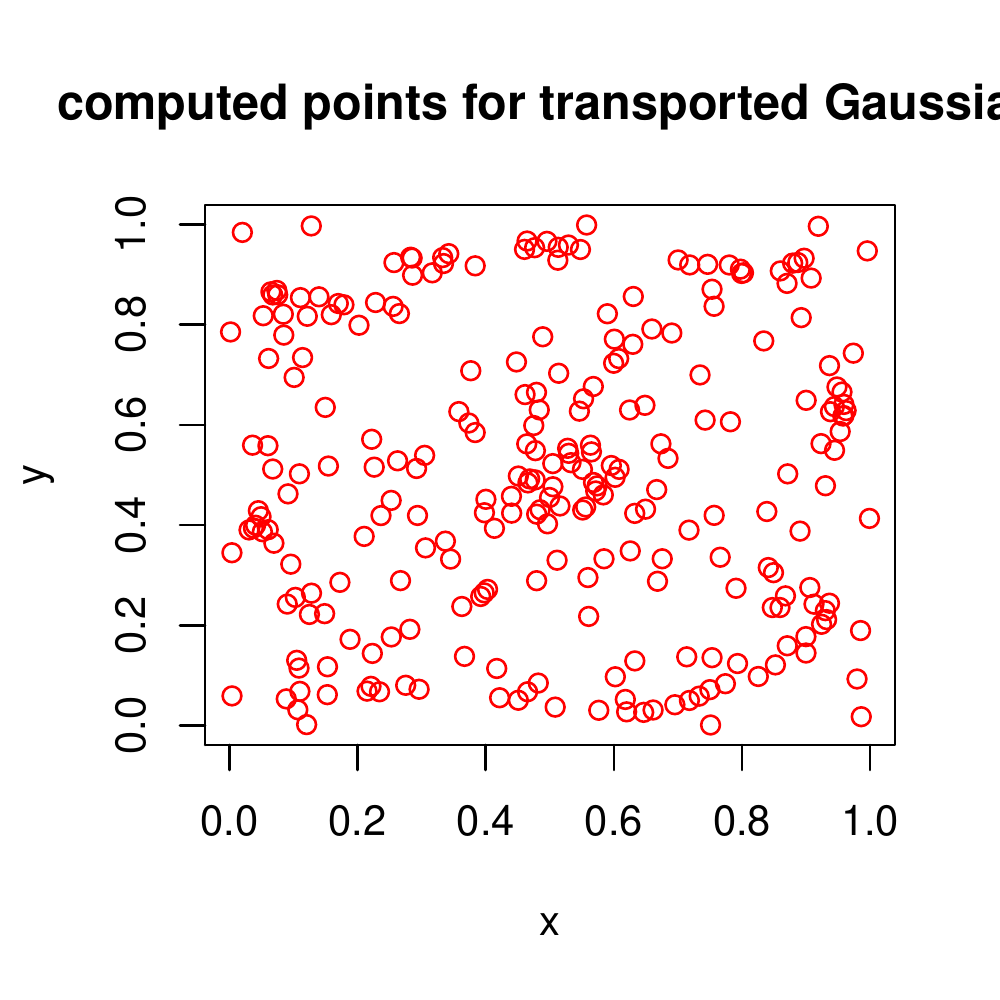} \includegraphics[width=0.45\linewidth]{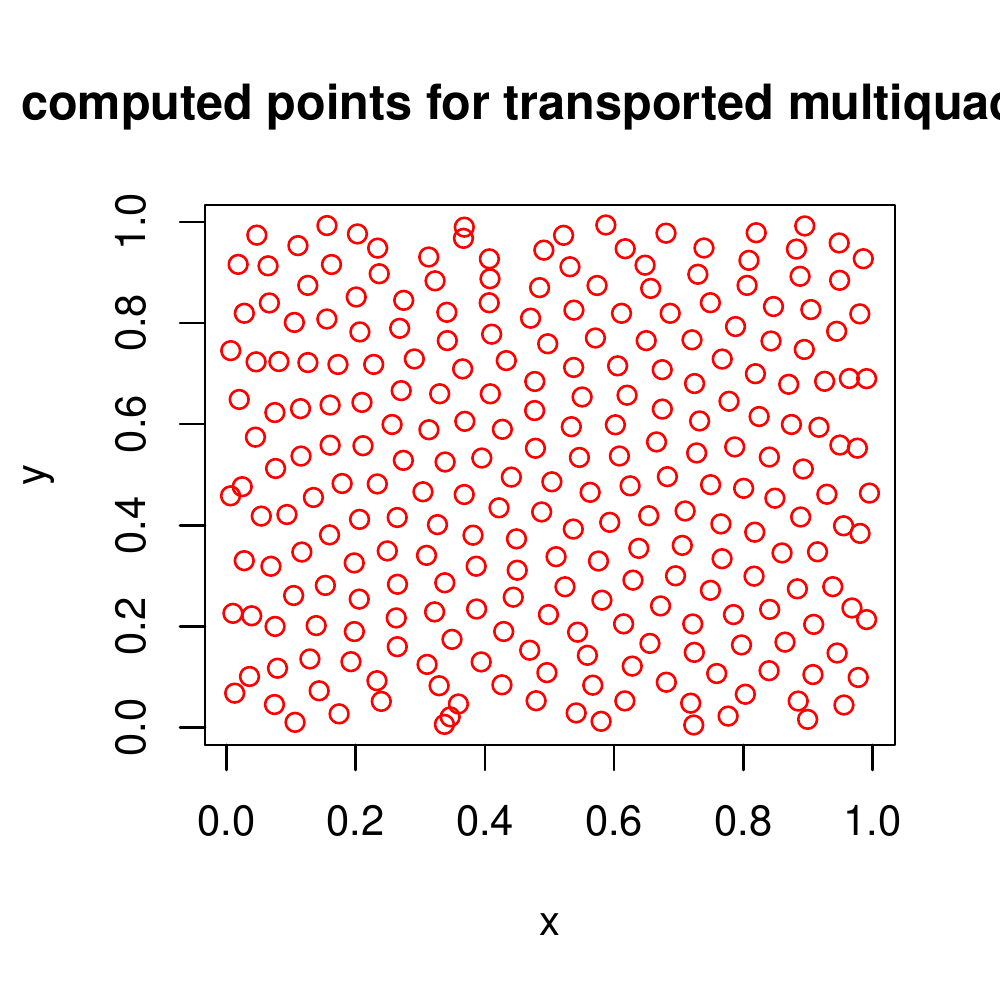} \includegraphics[width=0.45\linewidth]{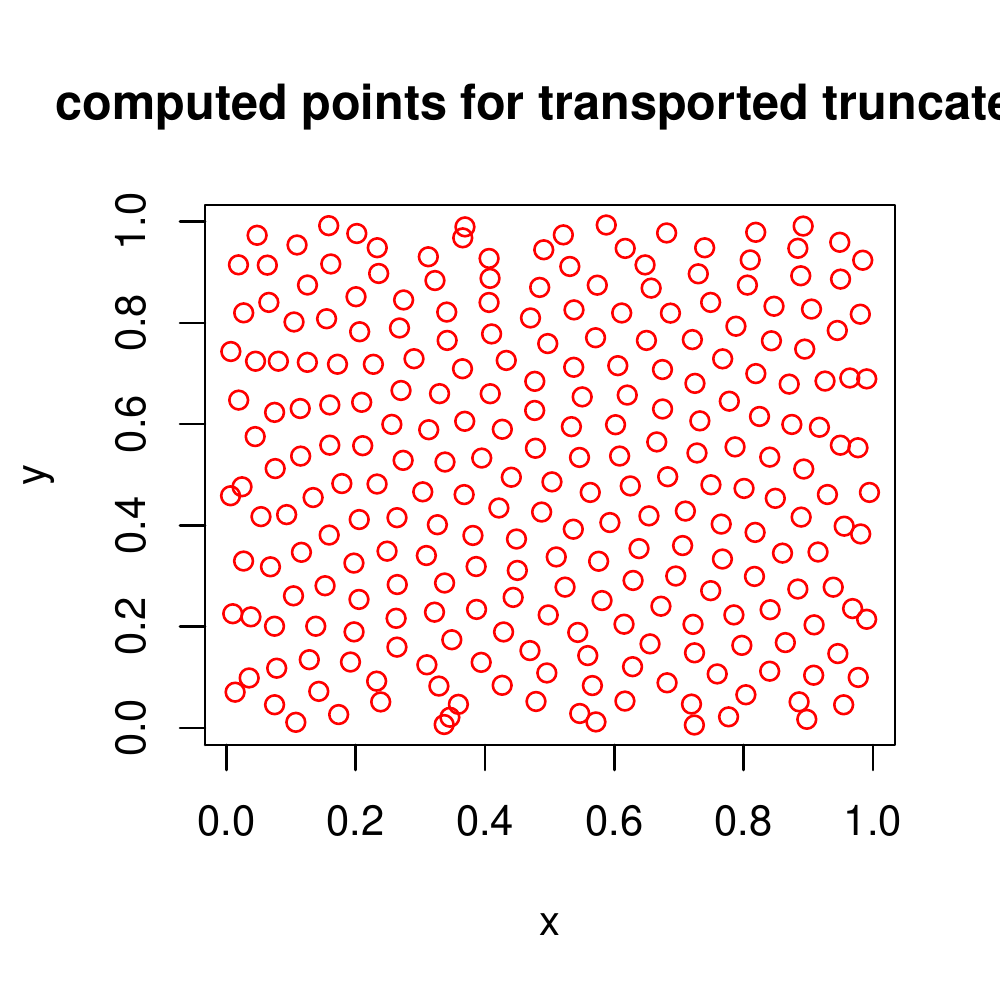} 
}
\caption{\label{DISTRIBTRANSPORTED} 
Random and numerical sequences for transported kernels with $N=256$}
\label{fig:DISTRIBTRANSPORTED}
\end{figure}

%=======================================================================

\vskip.05cm \par{\bf Acknowledgments.} 
The first author (PLF) gratefully acknowledges support from the Simons Center for Geometry and Physics, Stony Brook University, as well as from the Innovative Training Network (ITN) grant 642768 (ModCompShocks). This paper was written during the Academic year 2018-2019 when he was visiting the Courant Institute for Mathematical Sciences, New York University. 

%=======================================================================

%================================================================================

\end{document}